\documentclass[12pt]{amsart}

\usepackage{latexsym}
\usepackage{youngtab}           
\usepackage{latexsym}
\usepackage{amsfonts}
\usepackage{amsxtra}
\usepackage{amsmath}
\usepackage{amssymb}
\usepackage{mathrsfs}
\usepackage{amsthm}
\usepackage{fullpage}

\theoremstyle{plain}
\newtheorem{theorem}{Theorem}[section]
\newtheorem{proposition}[theorem]{Proposition}
\newtheorem{lemma}[theorem]{Lemma}
\newtheorem{corollary}[theorem]{Corollary}
\newtheorem{conjecture}[theorem]{Conjecture}
\newtheorem{claim}[theorem]{Claim}

\theoremstyle{definition}

\newtheorem{example}{Example}[section]
\newtheorem{remark}{Remark}[section]
\numberwithin{equation}{section}

 \DeclareMathOperator{\Hom}{Hom}
 \DeclareMathOperator{\STD}{Std}
\DeclareMathOperator{\RAD}{rad} 

\DeclareMathOperator{\SHAPE}{Shape}
\DeclareMathOperator{\RES}{Res}

\input xy
\xyoption{all}

\newcommand{\nc}{\newcommand}
\nc{\dten}{10} \nc{\deleven}{11} \nc{\dtwelve}{12}
\nc{\dthirteen}{13} \nc{\dfourteen}{14} \nc{\dfifteen}{15}
\nc{\dsixteen}{16}

\begin{document}
\title[Brauer and B--M--W Algebras]
{Specht Modules and Semisimplicity Criteria for Brauer and Birman--Murakami--Wenzl Algebras}

\author[J. Enyang]{John Enyang}


\keywords{Birman--Murakami--Wenzl algebra; Brauer algebra; Specht
module; cellular algebra; Jucys--Murphy operators.}

\begin{abstract}
A construction of bases for cell modules of the Birman--Murakami--Wenzl (or B--M--W) algebra $B_n(q,r)$ by lifting bases for cell modules of $B_{n-1}(q,r)$ is given. By iterating this procedure, we produce cellular bases for B--M--W algebras on which a large abelian subalgebra, generated by elements which generalise the Jucys--Murphy elements from the representation theory of the Iwahori--Hecke algebra of the symmetric group, acts triangularly. The triangular action of this abelian subalgebra is used to provide explicit criteria, in terms of the defining parameters $q$ and $r$, for B--M--W algebras to be semisimple. The aforementioned constructions provide generalisations, to the algebras under consideration here, of certain results from the Specht module theory of the Iwahori--Hecke algebra of the symmetric group.
\end{abstract}

\maketitle

\section{Introduction}
Using a recursive procedure which lifts bases of $B_{i-1}(q,r)$ to bases for $B_i(q,r)$, for $i=1,2,\dots,n$, we obtain new cellular bases (in the sense of~\cite{grahamlehrer}) for the B--M--W algebra $B_n(q,r)$, indexed by paths in an appropriate Bratteli diagram, whereby 
\begin{enumerate}
\item each cell module for $B_n(q,r)$ admits a filtration by cell modules for $B_{n-1}(q,r)$, and 
\item certain commuting elements in $B_n(q,r)$, which generalise the Jucys--Murphy elements in the Iwahori--Hecke algebra of the symmetric group, act triangularly on each cell module for the algebra $B_n(q,r)$. 
\end{enumerate}
The triangular action of the generalised Jucys--Murphy elements, combined with the machinery of cellular algebras from~\cite{grahamlehrer}, allows us to obtain explicit criteria, in terms of defining parameters, for any given B--M--W algebra to be semisimple. The aforementioned provide generalisations of classical results from the representation theory of the Iwahori--Hecke algebra of the symmetric group to the algebras under investigation here.

The contents of this article are presented as follows. 
\begin{enumerate}
    \item Definitions concerning partitions and tableaux, along with standard facts from the representation theory of the Iwahori--Hecke algebra of the symmetric group are stated in Section~\ref{prelsec}.
    \item In Section~\ref{b-m-w-a}, we define a generic version of the B--M--W algebras and restate in a more transparent notation the main results of~\cite{saru} on cellular bases of the same algebras.
    \item In Section~\ref{afield}, we state for reference some
    consequences following from the statements in Section~\ref{b-m-w-a}
    and the theory of cellular algebras given in~\cite{grahamlehrer}.
    \item In Section~\ref{resf}, an explicit
    description of the behaviour of the cell modules for
    generic B--M--W algebras under restriction is obtained.
    \item In Section~\ref{newbasis}, the results of
    Section~\ref{resf} are used to construct new bases for
    B--M--W algebras, indexed by pairs of paths in the
    Bratteli diagram associated with B--M--W algebras and
    generalising Murphy's construction~\cite{murphy} of bases
    for the Iwahori--Hecke algebras of the symmetric group. A demonstration of 
    the iterative procedure is given in detail in Examples~\ref{bas-ex:1} and~\ref{b-m-w-murphex:2}.
    \item Certain results of R.~Dipper and G.~James on the Jucys--Murphy operators of the Iwahori-Hecke algebra of the symmetric group is extended to generic B--M--W algebras in Theorem~\ref{utrangular:2}.
    \item Theorems~\ref{ss:1} and~\ref{d-w:analogue} use the above mentioned results to give sufficient criteria for the B--M--W algebras over a field to be semisimple.
    \item Theorem~\ref{br:utran} shows that the Jucys--Murphy elements act triangularly on each cell module of the Brauer algebra, while the semisimplicity criterion of Theorem~\ref{ss:3} is a weak version of a result of H.~Rui~\cite{rui:brss}. 
    \item Some conjectures on the semisimplicity of the Brauer algebras are given in Section~\ref{farce}.
\end{enumerate}

The author is indebted to B.~Srinivasan for guidance, to A.~Ram for remarks on a previous version of this paper, and to I.~Terada for discussions during the period this work was undertaken. The author is grateful to T.~Shoji and H.~Miyachi for comments and thanks the referees for numerous suggestions and corrections.
\section{Preliminaries}\label{prelsec}
\subsection{Combinatorics and Tableaux}
Throughout, $n$ will denote a positive integer and $\mathfrak{S}_n$ will be the symmetric group acting on $\{1,\dots,n\}$ on the right. For $i$ an integer, $1\le i<n$, let $s_i$ denote the transposition $(i,i+1)$. Then $\mathfrak{S}_n$ is generated as a Coxeter group by $s_1,s_2,\dots,s_{n-1}$, which satisfy the defining relations 
\begin{align*}
&s_i^2=1&&\text{for $1\le i<n$;}\\
&s_is_{i+1}s_i=s_{i+1}s_is_{i+1}&&\text{for $1\le i<n-1$;}\\
&s_is_j=s_js_i&&\text{for $2\le|i-j|$.}
\end{align*}
An expression $w=s_{i_1}s_{i_2}\cdots s_{i_k}$ in which $k$ is
minimal is called a \emph{reduced expression} for $w$, and
$\ell(w)=k$ is the \emph{length} of $w$.

Let $f$ be an integer, $0\le f\le [n/2]$. If $n-2f>0$, a
\emph{partition} of $n-2f$ is a non--increasing sequence
$\lambda=(\lambda_1,\dots,\lambda_k)$ of integers, $\lambda_i\ge0$,
such that $\sum_{i=1}^k\lambda_i=n-2f$; otherwise, if $n-2f=0$, write
$\lambda=\varnothing$ for the empty partition. The fact that $\lambda$ is
a partition of $n-2f$ will be denoted by $\lambda\vdash n-2f$.
We will also write $|\lambda|=\sum_{i\ge 1}\lambda_i$. The
integers $\{\lambda_i:\text{for $i\ge1$}\}$ are the \emph{parts}
of $\lambda$. If $\lambda$ is a partition of $n-2f$, the
\emph{Young diagram} of $\lambda$ is the set
\begin{align*}
[\lambda]=\{(i,j)\,:\,\text{$\lambda_i\ge j\ge1$ and $i\ge
1$}\,\}\subseteq \mathbb{N}\times\mathbb{N}.
\end{align*}
The elements of $[\lambda]$ are the \emph{nodes} of $\lambda$ and
more generally a node is a pair
$(i,j)\in\mathbb{N}\times\mathbb{N}$. The diagram $[\lambda]$ is
traditionally represented as an array of boxes with $\lambda_i$
boxes on the $i$--th row. For example, if $\lambda=(3,2)$, then
$[\lambda]=\text{\tiny\Yvcentermath1$\yng(3,2)$}$\,. Let
$[\lambda]$ be the diagram of a partition. A node $(i,j)$ is an
\emph{addable} node of $[\lambda]$ if $(i,j)\not\in[\lambda]$ and
$[\mu]=[\lambda] \cup\{(i,j)\}$ is the diagram of a partition; in
this case $(i,j)$ is also referred to as a \emph{removable} node
of $[\mu]$.

For our purposes, a \emph{dominance order} on partitions is
defined as follows: if $\lambda$ and $\mu$ are partitions, then
$\lambda\unrhd\mu$ if either
\begin{enumerate}
\item $|\mu|>|\lambda|$ or \item
 $|\mu|=|\lambda|$ and
$\sum_{i=1}^k\lambda_i\ge\sum_{i=1}^k\mu_i$ for all $k>0$.
\end{enumerate}
We will write $\lambda\rhd\mu$ to mean that $\lambda\unrhd\mu$ and
$\lambda\ne\mu$. Although the definition of the dominance order on
partitions employed here differs from the conventional
definition~\cite{macdonald} of the dominance order on partitions,
when restricted to the partitions of the odd integers
$\{1,3,\dots,n\}$ or to partitions of the even integers
$\{0,2,\dots,n\}$, depending as $n$ is odd or even, the order
$\unrhd$ as defined above is compatible with a cellular structure
of the Birman--Murakami--Wenzl and Brauer algebras, as shown
in~\cite{saru}, \cite{grahamlehrer} and~\cite{xi:qheredity}.

Let $f$ be an integer, $0\le f\le [n/2]$, and $\lambda$ be a
partition of $n-2f$. A $\lambda$--tableau labeled by
$\{2f+1,2f+2,\dots,n\}$ is a bijection $\mathfrak{t}$ from the nodes
of the diagram $[\lambda]$ to the integers $\{2f+1,2f+2,\dots,n\}$.
A given $\lambda$--tableau
$\mathfrak{t}:[\lambda]\to\{2f+1,2f+2,\dots,n\}$ can be visualised
by labeling the nodes of the diagram $[\lambda]$ with the integers
$2f+1,2f+2,\dots,n$. For example, if $n=10$, $f=2$ and
$\lambda=(3,2,1)$,
\begin{align}\label{tabex0.0}
\mathfrak{t}=\text{\tiny\Yvcentermath1$\young(58\dten,67,9)$}
\end{align}
represents a $\lambda$--tableau. A $\lambda$--tableau $\mathfrak{t}$
labeled by $\{2f+1,2f+2,\dots,n\}$ is said to be \emph{standard} if
\begin{align*}
\mathfrak{t}(i_1,j_1)\ge\mathfrak{t}(i_2,j_2),&&\text{whenever
$i_1\ge i_2$ and $j_1\ge j_2$.}
\end{align*}
If $\lambda$ is a partition of $n-2f$, write $\STD_{n}(\lambda)$ for
the set of standard $\lambda$--tableaux labeled by the integers
$\{2f+1,2f+2,\dots,n\}$. We let $\mathfrak{t}^\lambda$ denote the
element of $\STD_{n}(\lambda)$ in which $2f+1,2f+2,\dots,n$ are
entered in increasing order from left to right along the rows of
$[\lambda]$. Thus in the above example where $n=10$, $f=2$ and
$\lambda=(3,2,1)$,
\begin{align}\label{tabex1}
\mathfrak{t}^\lambda=\text{\tiny\Yvcentermath1$\young(567,89,\dten)$}\,.
\end{align}
The tableau $\mathfrak{t}^\lambda$ is referred to as the
\emph{superstandard tableau} in $\STD_{n}(\lambda)$. If
$\mathfrak{t}\in\STD_{n}(\lambda)$, we will write
$\lambda=\SHAPE(\mathfrak{t})$ and, abiding by the convention used
in the literature, $\STD(\lambda)$ will be used to denote the set of
standard tableaux
$\mathfrak{t}:[\lambda]\to\{1,2,\dots,|\lambda|\}$; we will refer to
elements of $\STD(\lambda)$ simply as standard $\lambda$--tableaux.
If $\mathfrak{s}\in\STD_{n}(\lambda)$, we will write
$\hat{\mathfrak{s}}$ for the tableau in $\STD(\lambda)$ which is
obtained by relabelling the nodes of $\mathfrak{s}$ by the map
$i\mapsto i-2f$.

If $\mathfrak{t}\in\STD_{n}(\lambda)$ and $i$ is an integer $2f<i\le
n$, define $\mathfrak{t}|_{i}$ to be the tableau obtained by
deleting each entry $k$ of $\mathfrak{t}$ with $k>i$ (compare
Example~\ref{btenex1} below). The set $\STD_{n}(\lambda)$ admits an
order $\unrhd$ wherein $\mathfrak{s}\unrhd\mathfrak{t}$ if
$\SHAPE(\mathfrak{s}|_i)\unrhd\SHAPE(\mathfrak{t}|_i)$ for each
integer $i$ with $2f< i\le n$. We adopt the usual convention of
writing $\mathfrak{s}\rhd\mathfrak{t}$ to mean that
$\mathfrak{s}\unrhd\mathfrak{t}$ and $\mathfrak{s}\ne\mathfrak{t}$.

The subgroup $\mathfrak{S}_{n-2f}=\langle
s_i:2f<i<n\rangle\subset\mathfrak{S}_n$ acts on the set of
$\lambda$--tableaux on the right in the usual manner, by permuting
the integer labels of the nodes of $[\lambda]$. For example,
\begin{align}\label{tabex0}
\text{\Yvcentermath1$\young(567,89,\dten)$}\,(6,8)(7,10,9)\,
=\text{\Yvcentermath1$\young(58\dten,67,9)$} \,.
\end{align}
If $\lambda$ is a partition of $n-2f$, then for our purposes the
\emph{Young subgroup} $\mathfrak{S}_\lambda$ is defined to be the
row stabiliser of $\mathfrak{t}^\lambda$ in $\mathfrak{S}_{n-2f}$.
For instance, when $n=10$, $f=2$ and $\lambda=(3,2,1)$, as
in~\eqref{tabex1} above, then $\mathfrak{S}_\lambda=\langle
s_5,s_6,s_8\rangle$. To each $\lambda$--tableau $\mathfrak{t}$,
associate a unique permutation $d(\mathfrak{t})\in\mathfrak{S}_{n-2f}$ by
the condition $\mathfrak{t}=\mathfrak{t}^\lambda d(\mathfrak{t})$.
If we refer to the tableau $\mathfrak{t}$ in~\eqref{tabex0.0} above
for instance, then $d(\mathfrak{t})=(6,8)(7,10,9)$
by~\eqref{tabex0}.
\subsection{The Iwahori--Hecke Algebra of the Symmetric Group}\label{ihsec}
For the purposes of this section, let $R$ denote an integral domain
and $q$ be a unit in $R$. The Iwahori--Hecke algebra (over $R$) of
the symmetric group is the unital associative $R$--algebra
$\mathscr{H}_n(q^2)$ with generators $X_1,X_2,\dots, X_{n-1},$ which
satisfy the defining relations
\begin{align*}
&(X_i-q)(X_i+q^{-1})=0&&\text{for $1\le i<n$;}\\
&X_iX_{i+1}X_i=X_{i+1}X_iX_{i+1}&&\text{for $1\le i<n-1$;}\\
&X_iX_j=X_jX_i&&\text{for $2\le|i-j|$.}
\end{align*}
If $w\in\mathfrak{S}_n$ and $s_{i_1}s_{i_2}\cdots s_{i_k}$ is a
reduced expression for $w$, then
\begin{align*}
X_w=X_{i_1}X_{i_2}\cdots X_{i_k}
\end{align*}
is a well defined element of $\mathscr{H}_n(q^2)$ and the set
$\{X_w\,:\,w\in\mathfrak{S}_n\}$ freely generates
$\mathscr{H}_n(q^2)$ as an $R$--module (theorems~1.8 and 1.13
of~\cite{mathas:ih}).

Below we state for later reference standard facts from the
representation theory of the Iwahori--Hecke algebra of the symmetric
group, of which details can be found in~\cite{mathas:ih}
or~\cite{murphy}. If $\mu$ is a partition of $n$, define the element
\begin{align*}
c_\mu=\sum_{w\in\mathfrak{S}_\mu} q^{l(w)}X_w.
\end{align*}
In this section, let $*$ denote the algebra anti--involution of
$\mathscr{H}_{n}(q^2)$ mapping $X_w \mapsto X_{w^{-1}}$. If $\lambda$ is a partition of $n$, $\check{\mathscr{H}}^\lambda_n$ is defined to be the two--sided ideal in $\mathscr{H}_n(q^2)$ generated by
\begin{align*}
\big\{c_\mathfrak{uv}=X_{d(\mathfrak{u})}^*c_\mu
X_{d(\mathfrak{v})}:\text{$\mathfrak{u},\mathfrak{v}\in\STD(\mu)$,
where $\mu\rhd\lambda$ }\big\}.
\end{align*}
The next statement is due to E.~Murphy in~\cite{murphy}.
\begin{theorem}\label{murphybasis}
The Iwahori--Hecke algebra $\mathscr{H}_n(q^2)$ is free as an
$R$--module with basis
\begin{align*}
\mathscr{M}=\left\{c_\mathfrak{uv}= X_{d(\mathfrak{u})}^*c_\lambda
X_{d(\mathfrak{v})}\,\bigg|\,
\begin{matrix}
\text{ for $\mathfrak{u},\mathfrak{v}\in\STD(\lambda)$ and }\\
\text{$\lambda$ a partition of $n$}
\end{matrix}
\right\}.
\end{align*}
Moreover, the following statements hold.
\begin{enumerate}
\item The $R$--linear anti--involution $*$ satisfies $*:c_{\mathfrak{st}}\mapsto c_\mathfrak{ts}$ for all $\mathfrak{s},\mathfrak{t}\in\STD(\lambda)$.
\item Suppose that $h\in\mathscr{H}_n(q^2)$,
and that $\mathfrak{s}$ is a standard $\lambda$--tableau. Then
there exist $a_\mathfrak{u}\in R$, for
$\mathfrak{u}\in\STD(\lambda)$, such that for all
$\mathfrak{v}\in\STD(\lambda)$,
\begin{align}\label{tsu}
c_{\mathfrak{vs}}h\equiv
\sum_{\mathfrak{u}\in\STD(\lambda)}a_\mathfrak{u}
c_{\mathfrak{vu}} \mod \check{\mathscr{H}}^\lambda_n.
\end{align}
\end{enumerate}
\end{theorem}
The basis $\mathscr{M}$ is cellular in the sense
of~\cite{grahamlehrer}. If $\lambda$ is a partition of $n$, the
cell (or Specht) module $C^\lambda$ for $\mathscr{H}_n(q^2)$ is
the $R$--module freely generated by
\begin{align}\label{specht1}
\{c_{\mathfrak{s}}=c_\lambda X_{d(\mathfrak{s})}+\check{\mathscr{H}}_n^\lambda\,:\,\mathfrak{s}\in
\STD(\lambda)\},
\end{align}
and given the right $\mathscr{H}_n(q^2)$--action
\begin{align*}
c_{\mathfrak{s}}h=
\sum_{\mathfrak{u}\in\STD(\lambda)}
a_{\mathfrak{u}}c_{\mathfrak{u}},&&\text{for
$h\in\mathscr{H}_n(q^2)$,}
\end{align*}
where the coefficients $a_\mathfrak{u}\in R$, for
$\mathfrak{u}\in\STD(\lambda)$, are determined by the
expression~\eqref{tsu}. The basis~\eqref{specht1} is referred to
as the Murphy basis for $C^\lambda$ and $\mathscr{M}$ is the
Murphy basis for $\mathscr{H}_n(q^2)$.
\begin{remark}
The $\mathscr{H}_n(q^2)$--module $C^\lambda$ is the contragradient dual of the Specht module defined in~\cite{dipper-james:1}. 
\end{remark}

Let $\lambda$ and $\mu$ be partitions of $n$. A $\lambda$--tableau
of type $\mu$ is a map $\mathsf{T}:[\lambda]\to \{1,2,\dots,d\}$
such that $\mu_i=|\{y\in[\lambda]\,:\,\mathsf{T}(y)=i\}|$ for
$i\ge 1$. A $\lambda$--tableau $\mathsf{T}$ of type $\mu$ is said
to be \emph{semistandard} if (i) the entries in each row of
$\mathsf{T}$ are non--decreasing, and (ii) the entries in each
column of $\mathsf{T}$ are strictly increasing. If $\mu$ is a
partition, the semistandard tableau $\mathsf{T}^\mu$ is defined to
be the tableau of type $\mu$ with $\mathsf{T}^\mu(i,j)=i$ for
$(i,j)\in[\mu]$.
\begin{example}\label{sstabex:1}
Let $\mu=(3,2,1)$. Then the semistandard tableaux of type $\mu$
are $\mathsf{T}^\mu=\text{\tiny\Yvcentermath1$\young(111,22,3)$}$,
$\text{\tiny\Yvcentermath1$\young(1112,2,3)$}$,
$\text{\tiny\Yvcentermath1$\young(111,223)$}$\,,
$\text{\tiny\Yvcentermath1$\young(1112,23)$}$\,,
$\text{\tiny\Yvcentermath1$\young(1113,22)$}$\,,
$\text{\tiny\Yvcentermath1$\young(11122,3)$}$\,,
$\text{\tiny\Yvcentermath1$\young(11123,2)$}$\,, and
$\text{\tiny\Yvcentermath1$\young(111223)$}$\,, as in Example~4.1
of~\cite{mathas:ih}. All the semistandard tableaux of type $\mu$
are obtainable from $\mathsf{T}^\mu$ by ``moving nodes up" in
$\mathsf{T}^\mu$.
\end{example}
If $\lambda$ and $\mu$ are partitions of $n$, the set of
semistandard $\lambda$--tableaux of type $\mu$ will be denoted by
$\mathcal{T}_0(\lambda,\mu)$. Further, given a $\lambda$--tableau
$\mathfrak{t}$ and a partition $\mu$ of $n$, then
$\mu(\mathfrak{t})$ is defined to be the $\lambda$--tableau of
type $\mu$ obtained from $\mathfrak{t}$ by replacing each entry
$i$ in $\mathfrak{t}$ with $k$ if $i$ appears in the $k$--th row
of the superstandard tableau $\mathfrak{t}^\mu\in\STD(\mu)$.
\begin{example}\label{sstabex:2}
Let $n=7$, and $\mu=(3,2,1,1)$, so that
$\mathfrak{t}^\mu=\text{\tiny\Yvcentermath1$\young(123,45,6,7)$}$\,.
If $\nu=(4,3)$ and
$\mathfrak{t}=\text{\tiny\Yvcentermath1$\young(1237,456)$}$\,,
then
$\mu(\mathfrak{t})=\text{\tiny\Yvcentermath1$\young(1114,223)$}$\,.
\end{example}
Let $\mu$ and $\nu$ be partitions of $n$. If $\mathsf{S}$ is a
semistandard $\nu$--tableau of type $\mu$, and $\mathfrak{t}$ is a
standard $\nu$--tableau, define in $\mathscr{H}_n(q^2)$ the
element
\begin{align}\label{some:1}
c_{\mathsf{S}\mathfrak{t}}=\sum_{\substack{\mathfrak{s}\in\STD(\nu)\\
\mu(\mathfrak{s})=\mathsf{S}}}
q^{\ell(d(\mathfrak{s}))}c_{\mathfrak{st}}.
\end{align}
Given a partition $\mu$ of $n$, let $M^\mu$ be the right
$\mathscr{H}_n(q^2)$--module generated by $c_\mu$. The next statement
is a special instance of a theorem of E.~Murphy (Theorem~4.9
of~\cite{mathas:ih}).
\begin{theorem}\label{permod}
Let $\mu$ be a partition of $n$. Then the collection
\begin{align*}
\{c_{\mathsf{S}\mathfrak{t}}: \mathsf{S}\in
\mathcal{T}_0(\nu,\mu),\mathfrak{t}\in\STD(\nu),\text{ for $\nu$ a
partition of $n$}\}
\end{align*}
freely generates $M^\mu$ as an $R$--module.
\end{theorem}

If $\mu$ and $\lambda$ are partitions of $n-1$ and $n$ respectively,
for the purposes of the present Section~\ref{ihsec}, we write
$\mu\to\lambda$ to mean that the diagram $[\lambda]$ is obtained by
adding a node to the diagram $[\mu]$, as exemplified by the
truncated Bratteli diagram associated with $\mathscr{H}_n(q^2)$
displayed in~\eqref{h-bratteli} below (Section~4
of~\cite{ramleduc:rh}).
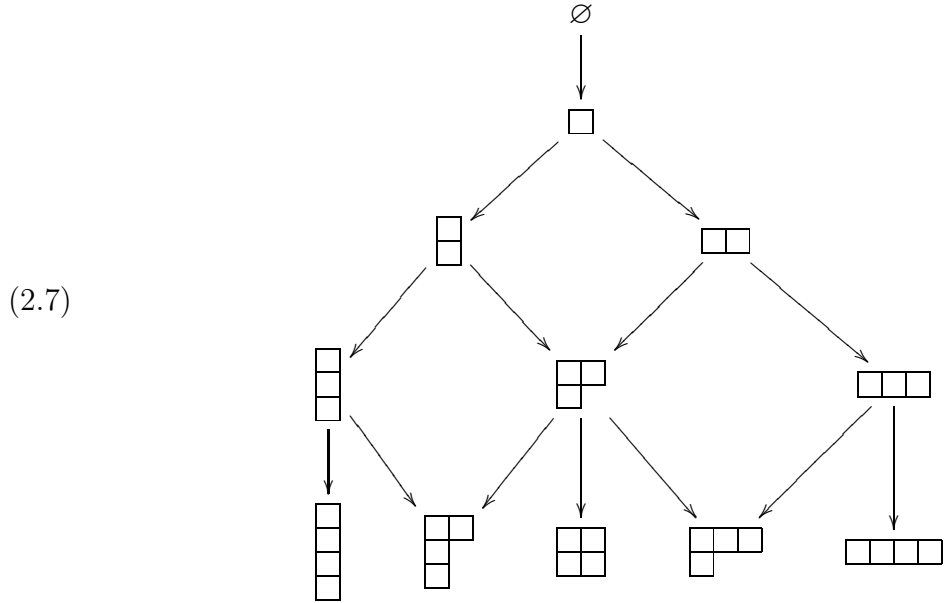
\begin{align}\label{h-bratteli}
\begin{matrix}
\xymatrix{ & & \varnothing \ar[d] & & \\
 & & \text{\tiny\Yvcentermath1$\yng(1)$}\ar[dl]\ar[dr] & & \\
 & \text{\tiny\Yvcentermath1$\yng(1,1)$}\ar[dl]\ar[dr]&  &\text{\tiny\Yvcentermath1$\yng(2)$}\ar[dl]\ar[dr]&\\
 \text{\tiny\Yvcentermath1$\yng(1,1,1)$}\ar[d]\ar[dr] & & \text{\tiny\Yvcentermath1$\yng(2,1)$}\ar[dl]\ar[d]\ar[dr]
 & &\text{\tiny\Yvcentermath1$\yng(3)$} \ar[d]\ar[dl]\\
 \text{\tiny\Yvcentermath1$\yng(1,1,1,1)$}&\text{\tiny\Yvcentermath1$\yng(2,1,1)$}
 &\text{\tiny\Yvcentermath1$\yng(2,2)$} & \text{\tiny\Yvcentermath1$\yng(3,1)$}& \text{\tiny\Yvcentermath1$\yng(4)$}}
\end{matrix}
\end{align}
If $\lambda$ is a partition of $n$ then, as in~\cite{ramleduc:rh},
define a \emph{path} of shape $\lambda$ in the Bratteli diagram
associated with $\mathscr{H}_{n}(q^2)$ to be a sequence of
partitions
\begin{align*}
\left(\lambda^{(0)},\lambda^{(1)},\dots,\lambda^{(n)}\right)
\end{align*}
satisfying the conditions that $\lambda^{(0)}=\varnothing$ is the empty
partition, $\lambda^{(n)}=\lambda$, and
$\lambda^{(i-1)}\to\lambda^{(i)}$, for $1\le i\le n$. As observed in Section~4 of~\cite{ramleduc:rh}, there is a natural correspondence between the paths in the Bratteli diagram associated with $\mathscr{H}_{n}(q^2)$ and the elements of $\STD(\lambda)$ whereby $\mathfrak{t}\mapsto(\lambda^{(0)},\lambda^{(1)},\dots,\lambda^{(n)})$  and $\lambda^{(i)}=\SHAPE(\mathfrak{t}|_i)$ for $1\le i \le n$.
\begin{example}\label{tabdiag}
Let $n=6$ and $\lambda=(3,2,1)$. Then the identification of standard
$\lambda$--tableau with paths of shape $\lambda$ in the Bratteli
diagram associated with $\mathscr{H}_n(q^2)$ maps
\begin{align*}
\mathfrak{t}=\text{\tiny\Yvcentermath1$\young(136,24,5)$}\,\mapsto
\left(\,\text{\tiny\Yvcentermath1$\yng(1)$}\,,\text{\tiny\Yvcentermath1$\yng(1,1)$}\,,
\text{\tiny\Yvcentermath1$\yng(2,1)$}\,,\text{\tiny\Yvcentermath1$\yng(2,2)$}\,,
\text{\tiny\Yvcentermath1$\yng(2,2,1)$}\,,\text{\tiny\Yvcentermath1$\yng(3,2,1)$}\,\right).
\end{align*}
\end{example}
Taking advantage of the bijection between the standard
$\lambda$--tableaux and the paths of shape $\lambda$ in the Bratteli
diagram of $\mathscr{H}_n(q^2)$, we will have occasion to write
\begin{align*}
\mathfrak{t}=\left(\lambda^{(0)},\lambda^{(1)},\dots,\lambda^{(n)}\right),
\end{align*}
explicitly identifying each standard $\lambda$--tableau
$\mathfrak{t}$ with a path of shape $\lambda$ in the Bratteli
diagram.

For each integer $i$ with $1\le i\le n$, consider
$\mathscr{H}_i(q^2)$ as the subalgebra of $\mathscr{H}_n(q^2)$
generated by the elements $X_1,X_2,\dots,X_{i-1}$, thereby obtaining
the tower of algebras
\begin{align}\label{tower:1}
R=\mathscr{H}_1(q^2)\subseteq\mathscr{H}_{2}(q^2)\subseteq\cdots
\subseteq\mathscr{H}_n(q^2).
\end{align}
Given a right $\mathscr{H}_{n}(q^2)$-module $V$, write $\RES(V)$ for
the restriction of $V$ to $\mathscr{H}_{n-1}(q^2)$ by the
identifications~\ref{tower:1}. Lemma~\ref{hres} below, which is a consequence of Theorem~7.2 of~\cite{murphy}, shows that the Bratteli diagram associated
with $\mathscr{H}_n(q^2)$ describes the behaviour of the cell modules
for $\mathscr{H}_n(q^2)$ under restriction to
$\mathscr{H}_{n-1}(q^2)$.
\begin{lemma}\label{hres}
Let $\lambda$ be a partition of $n$. For each partition $\mu$ of
$n-1$ with $\mu\to\lambda$, let $A^\mu$ denote the $R$--submodule of $C^\lambda$ freely generated by
\begin{align*}
\{\,c_\mathfrak{v}:\text{$\mathfrak{v}\in\STD(\lambda)$
and $\SHAPE(\mathfrak{v}|_{n-1})\unrhd\mu$}\,\}
\end{align*}
and write $\check{A}^\mu$ for the $R$--submodule of $S^\lambda$
freely generated by
\begin{align*}
\{\,c_\mathfrak{v}:\text{$\mathfrak{v}\in\STD(\lambda)$
and $\SHAPE(\mathfrak{v}|_{n-1})\rhd\mu$}\,\}.
\end{align*}
If $\mathfrak{v}\in\STD_n(\lambda)$ and $\mathfrak{v}|_{n-1}=\mathfrak{t}^\mu$, then the $R$--linear map determined on generators by 
\begin{align*}
c_{\mathfrak{v}} X_{d(\mathfrak{u})}+\check{A}^\mu\mapsto c_\mathfrak{u},&&\text{for $\mathfrak{u}\in\STD(\mu)$,}
\end{align*}
is an isomorphism $A^\mu/\check{A}^\mu\cong C^\mu$ of
$\mathscr{H}_{n-1}(q^2)$--modules.
\end{lemma}

The Jucys--Murphy operators $\tilde{D}_i$ in $\mathscr{H}_n(q^2)$
are usually defined (Section~3 of~\cite{mathas:ih}) by
$\tilde{D}_1=0$ and
\begin{align}\label{h-jmdef}
\tilde{D}_i=\sum_{k=1}^{i-1} X_{(k,i)},&&\text{for $i=1,\dots,n$}
\end{align}
As per an exercise in~\cite{mathas:ih}, we define $D_1=1$ and set
$D_i=X_{i-1}D_{i-1}X_{i-1}$. Since $D_i=1+(q-q^{-1})\tilde{D}_i$,
and the $\tilde{D}_i$ can be cumbersome, we work with the $D_i$
rather than the $\tilde{D}_i$. We also refer to the $D_i$ as
Jucys--Murphy elements; this should cause no confusion. The following proposition is well known.
\begin{proposition}\label{murphyop:1}
Let $i$ and $k$ be integers, $1\le i<n$ and $1\le k\le n$.
\begin{enumerate}
\item $X_i$ and $D_k$ commute if $i\ne k-1,k$. \item $D_i$ and
$D_k$ commute. \item $X_i$ commutes with $D_iD_{i+1}$ and
$D_i+D_{i+1}$.
\end{enumerate}
\end{proposition}
Let
$\mathfrak{t}=\left(\lambda^{(0)},\lambda^{(1)},\dots,\lambda^{(n)}\right)$
be a standard $\lambda$--tableau identified with the corresponding
path in the Bratteli diagram of $\mathscr{H}_n(q^2)$. For each
integer $k$ with $1\le k\le n$, define
\begin{align}\label{mon:1}
P_\mathfrak{t}(k)=q^{2(j-i)} &&\text{where
$[\lambda^{(k)}]=[\lambda^{(k-1)}]\cup \{(i,j) \}$}.
\end{align}
The next statement is due to R.~Dipper and G.~James (Theorem 3.32
of~\cite{mathas:ih}).
\begin{theorem}\label{utrangular:1}
Suppose that $\lambda$ is a partition of $n$ and let
$\mathfrak{s}$ be a standard $\lambda$--tableau. If $k$ is an
integer, $1\le k\le n$, then there exist $a_\mathfrak{v}\in R$,
for $\mathfrak{v}\rhd\mathfrak{s}$, such that
\begin{align*}
c_{\mathfrak{s}}\,D_k=
P_{\mathfrak{s}}(k)c_{\mathfrak{s}}+\sum_{\substack{\mathfrak{v}\in\STD(\lambda)\\
\mathfrak{v}\rhd\mathfrak{s}}}a_{\mathfrak{v}}
c_{\mathfrak{v}}.
\end{align*}
\end{theorem}
One objective at hand is to provide an extension of
Lemma~\ref{hres} and Theorem~\ref{utrangular:1} to the Brauer and
Birman--Murakami--Wenzl algebras.
\section{The Birman--Murakami--Wenzl Algebras}\label{b-m-w-a}
Let $q,r$ be indeterminates over $\mathbb{Z}$ and
$R=\mathbb{Z}[q^{\pm1},r^{\pm1},(q-q^{-1})^{-1}]$. The
Birman--Murakami--Wenzl algebra $B_n(q,r)$ over $R$ is the unital
associative $R$--algebra generated by the elements
$T_1,T_2,\dots,T_{n-1}$, which satisfy the defining relations
\begin{align*}
&(T_i-q)(T_i+q^{-1})(T_i-r^{-1})=0&&\text{for $1\le i<n$;}\\
&T_iT_{i+1}T_i=T_{i+1}T_iT_{i+1}&&\text{for $1\le i\le n-2$;}\\
&T_iT_j=T_jT_i&&\text{for $2\le|i-j|$;}\\
&E_iT_{i-1}^{\pm1}E_i=r^{\pm1}E_i&&\text{for $2\le i\le n-1$;}\\
&E_iT_{i+1}^{\pm1}E_i=r^{\pm1}E_i&&\text{for $1\le i\le n-2$;}\\
&T_iE_i=E_iT_i=r^{-1}E_i&&\text{for $1\le i\le n-1$,}
\end{align*}
where $E_i$ is the element defined by the expression
\begin{align*}
(q-q^{-1})(1-E_i)=T_i-T_i^{-1}.
\end{align*}
Writing
\begin{align}\label{zdef}
z=\frac{(q+r)(qr-1)}{r(q+1)(q-1)},
\end{align}
then (Section~3 of~\cite{wenzlqg}) one derives additional
relations
\begin{align*}
&E_i^2=zE_i,\\
&E_iT_i^{\pm1}=r^{\mp1}E_i=T_i^{\pm1}E_i,\\
&T^2_i=1+(q-q^{-1})(T_i-r^{-1}E_i) \\
&E_{i\pm1}T_iT_{i\pm1}=T_iT_{i\pm1}E_i\\
&E_iT_{i\pm1}E_i=rE_i\\
&E_iT_{i\pm1}^{-1}E_i=r^{-1}E_i\\
&E_iE_{i\pm1}E_i=E_i\\
&E_iE_{i\pm1}=E_iT_{i\pm1}T_i=T_{i\pm1}T_iE_{i\pm1}.
\end{align*}
If $w\in\mathfrak{S}_n$ is a permutation and
$w=s_{i_1}s_{i_2}\cdots s_{i_k}$ is a reduced expression for $w$,
then
\begin{align*}
T_w=T_{i_1}T_{i_2}\cdots T_{i_k}
\end{align*}
is a well defined element of $B_n(q,r)$.
\begin{remark} The generator $T_i$ above
differs by a factor of $q$ from the generator used in~\cite{saru}
but coincides with the element $g_i$ of~\cite{ramleduc:rh}
and~\cite{wenzlqg}.
\end{remark}
If $f$ is an integer, $0\le f\le [n/2]$, define $B^{f}_n$ to be
the two sided ideal of $B_n(q,r)$ generated by the element
$E_1E_3\cdots E_{2f-1}$. Then
\begin{align}\label{filt1}
(0)\subseteq B^{[n/2]}_n\subseteq
B^{[n/2]-1}_n\subseteq\cdots\subseteq B^{1}_n\subseteq
B^{0}_n=B_{n}(q,r)
\end{align}
gives a filtration of $B_n(q,r)$. As in Theorem~4.1 of~\cite{saru}
(see also~\cite{xi:qheredity}), refining the
filtration~\eqref{filt1} gives the cell modules, in the sense
of~\cite{grahamlehrer}, for the algebra $B_n(q,r)$. If  $f$ is an
integer, $0\le f\le [n/2]$, and $\lambda$ is a partition of
$n-2f$, define the element
\begin{align*}
x_\lambda=\sum_{w\in\mathfrak{S}_\lambda} q^{\ell(w)}T_w,
\end{align*}
where $\mathfrak{S}_\lambda$ is row stabiliser in the subgroup
$\langle s_i:2f<i<n\rangle$ of the superstandard tableau
$\mathfrak{t}^\lambda\in\STD_{n}(\lambda)$ as defined in
Section~\ref{prelsec}; finally define
\begin{align*}
m_\lambda=E_1E_3\cdots E_{2f-1}x_\lambda
\end{align*}
which is the analogue to the element $c_\lambda$ in the
Iwahori-Hecke algebra of the symmetric group.
\begin{example}
Let $n=10$ and $\lambda=(3,2,1)$. From the $\lambda$--tableau
displayed in~\eqref{tabex1} comes the subgroup
$\mathfrak{S}_\lambda=\langle s_5,s_6, s_8\rangle$, and
\begin{align*}
m_\lambda&=E_1E_3\sum_{w\in\mathfrak{S}_\lambda}q^{\ell(w)}T_w\\
&=E_1E_3(1+qT_5)(1+qT_6+q^2T_6T_5)(1+qT_8).
\end{align*}
\end{example}
If $f$ is an integer, $0\le f\le [n/2]$, define
\begin{align*}
\mathscr{D}_{f,n}=\left\{v\in\mathfrak{S}_n \,\Bigg|\,
\begin{matrix}
\text{$(2i+1)v<(2j+1)v$ for $0\le i<j<f$;}\\
\text{$(2i+1)v < (2i+2)v$ for $0\le i<f$;}\\
\text{and $(i)v<(i+1)v$ for $2f<i<n$}
\end{matrix}
\right\}.
\end{align*}
As shown in Section~3 of~\cite{saru}, the collection
$\mathscr{D}_{f,n}$ is a complete set of right coset representatives
for the subgroup $\mathfrak{B}_f\times\mathfrak{S}_{n-2f}$ in
$\mathfrak{S}_n$, where  $\mathfrak{S}_{n-2f}$ is identified with
the subgroup $\langle s_i:2f<i<n\rangle$ of $\mathfrak{S}_n$
and  $\mathfrak{B}_0=\langle 1 \rangle$, $\mathfrak{B}_1=\langle s_1\rangle$ and, for $f>1$,
\begin{align*}
\mathfrak{B}_f=
\langle
s_{2i-1},s_{2i}s_{2i+1}s_{2i-1}s_{2i}: 1\le i \le f  
\rangle.
\end{align*}
Additionally, it is evident (Proposition~3.1 of~\cite{saru}) that
if $v$ is an element of $\mathscr{D}_{f,n}$, then
$\ell(uv)=\ell(u)+\ell(v)$ for all $u$ in $\langle s_i : 2f< i
<n\rangle$.
\begin{remark}
After fixing a choice of over and under crossings, there is a natural bijection between the coset representatives $\mathscr{D}_{f,n}$ and the $(n,n-2f)$--dangles of Definition~3.3
of~\cite{xi:qheredity}.
\end{remark}
For each partition $\lambda$ of $n-2f$, define
$\mathcal{I}_{n}(\lambda)$ to be the set of ordered pairs
\begin{align}\label{index:1}
{\mathcal{I}}_{n}(\lambda)=\left\{(\mathfrak{s},v):
\mathfrak{s}\in\STD_{n}(\lambda)\text{ and } v\in\mathscr{D}_{f,n},
\right\}
\end{align}
and define $B_n^\lambda$ to be the two--sided ideal in $B_n(q,r)$ generated by $m_\lambda$ and let
\begin{align*}
\check{B}_n^\lambda=\sum_{\mu\rhd\lambda}B_n^\mu
\end{align*}
so that $B^{f+1}_n\subseteq\check{B}_n^\lambda$, by the definition of the dominance order on partitions given in Section~\ref{prelsec}. Let $*$ be the algebra anti--involution of $B_n(q,r)$ which maps $T_w\mapsto T_{w^{-1}}$ and $E_i\mapsto E_i$.

That $B_n(q,r)$ is cellular in the sense of~\cite{grahamlehrer} was shown in~\cite{xi:qheredity}; the next statement which is Theorem~4.1 of~\cite{saru}, gives an explicit cellular basis for $B_n(q,r)$.
\begin{theorem}\label{saruthm}
The algebra $B_n(q,r)$ is freely generated as an $R$--module by
the collection
\begin{align*}
\left\{T^*_vT_{d(\mathfrak{s})}^*m_\lambda
T_{d(\mathfrak{t})}T_u\,\bigg|\,
\begin{matrix}
\text {$(\mathfrak{s},v),(\mathfrak{t},u)\in \mathcal{I}_{n}(\lambda)$, for $\lambda$ a partition} \\
\text{of $n-2f$, and $0\le f\le [n/2]$\,}
\end{matrix}
\right\}.
\end{align*}
Moreover, the following statements hold.
\begin{enumerate}
\item The algebra anti--involution $*$ satisfies
\begin{align*}
*:T^*_vT_{d(\mathfrak{s})}^*m_\lambda T_{d(\mathfrak{t})}T_u\mapsto
T^*_uT_{d(\mathfrak{t})}^*m_\lambda T_{d(\mathfrak{s})}T_v
\end{align*}
for all $(\mathfrak{s},v),(\mathfrak{t},u)\in\mathcal{I}_n(\lambda)$.
\item  Suppose that $b\in B_n(q,r)$ and let $f$ be an integer, $0\le f\le [n/2]$. If
$\lambda$ is a partition of $n-2f$ and
$(\mathfrak{t},u)\in\mathcal{I}_{n}(\lambda)$, then there exist
$a_{(\mathfrak{u},w)}\in R$, for
$(\mathfrak{u},w)\in\mathcal{I}_{n}(\lambda)$, such that for all
$(\mathfrak{s},v)\in\mathcal{I}_{n}(\lambda)$,
\begin{align}\label{btsu:1}
T^*_vT_{d(\mathfrak{s})}^*m_\lambda T_{d(\mathfrak{t})}T_u b\equiv
\sum_{(\mathfrak{u},w)}a_{(\mathfrak{u},w)}
T^*_vT_{d(\mathfrak{s})}^*m_\lambda T_{d(\mathfrak{u})}T_w \mod
\check{B}^\lambda_n.
\end{align}
\end{enumerate}
\end{theorem}
As a consequence of the above theorem,  $\check{B}_n^\lambda$ is the $R$--module freely generated by the collection
\begin{align*}
\big\{T^*_vT_{d(\mathfrak{s})}^*m_\mu T_{d(\mathfrak{t})}T_u
:(\mathfrak{s},v),(\mathfrak{t},u)\in\mathcal{I}_{n}(\mu), \text
{ for }\mu\rhd\lambda\big\}.
\end{align*}
If $f$ is an integer, $0\le f\le [n/2]$, and $\lambda$ is a
partition of $n-2f$, the cell module $S^\lambda$ is defined to be
the $R$--module freely generated by
\begin{align}\label{bcell:1}
\left \{m_\lambda T_{d(\mathfrak{t})}T_u
+\check{B}_n^\lambda\,|\,(\mathfrak{t},u)\in
\mathcal{I}_{n}(\lambda)\right\}
\end{align}
and given the right $B_n(q,r)$ action
\begin{align*}
m_\lambda T_{d(\mathfrak{t})}T_u b+\check{B}_n^\lambda=
\sum_{(\mathfrak{u},w)}a_{(\mathfrak{u},w)} m_\lambda
T_{d(\mathfrak{u})}T_w +\check{B}_n^\lambda&&\text{for $b\in
B_n(q,r)$,}
\end{align*}
where the coefficients $a_{(\mathfrak{u},w)}\in R$, for
$(\mathfrak{u},w)$ in $\mathcal{I}_{n}(\lambda)$, are determined by
the expression~\eqref{btsu:1}.
\begin{example}\label{cellex:0}
Let $n=6$, $f=1$, and $\lambda=(3,1)$. If $i,j$ are integers with
$1\le i<j\le n$, write $v_{i,j}= s_2 s_3\cdots s_{j-1}
s_1s_2\cdots s_{i-1}$, so that
\begin{align*}
\mathscr{D}_{f,n}=\{v_{i,j}:1\le i<j\le n\}.
\end{align*}
Since
\begin{align*}
\STD_{n}(\lambda)=\left\{\mathfrak{t}^\lambda=\text{\tiny\Yvcentermath1$\young(345,6)$}\,\,,
\mathfrak{t}^\lambda
s_5=\text{\tiny\Yvcentermath1$\young(346,5)$}\,\,,
\mathfrak{t}^\lambda
s_5s_4=\text{\tiny\Yvcentermath1$\young(356,4)$}\,\right\}
\end{align*}
and $m_\lambda=E_1(1+qT_4)(1+qT_3+q^2T_3T_4)$, the basis for
$S^\lambda$, of the form displayed in~\eqref{bcell:1}, is
\begin{align*}
\big\{m_\lambda T_{d(\mathfrak{s})} T_{v_{i,j}}+\check{B}^\lambda_n
:\text{$\mathfrak{s}\in\STD_{n}(\lambda)$ and $1\le i<j\le
n$}\big\}.
\end{align*}
\end{example}
As in Proposition~2.4 of~\cite{grahamlehrer}, the cell module
$S^\lambda$ for $B_n(q,r)$ admits a symmetric associative bilinear
form $\langle\,\,,\,\rangle:S^\lambda\times S^\lambda\to R$
defined by
\begin{align}\label{formdef:1}
\langle m_\lambda T_{d(\mathfrak{u})}T_v,m_\lambda
T_{d(\mathfrak{v})}T_w \rangle m_\lambda\equiv m_\lambda
T_{d(\mathfrak{u})}T_v T_w^* T_{d(\mathfrak{v})}^*m_\lambda\mod
\check{B}^\lambda_n.
\end{align}
We return to the question of using the bilinear
form~\eqref{formdef:1} to extract explicit information about the
structure of the B--W--W algebras in Section~\ref{b-m-w-sc}, but
record the following example for later reference.
\begin{example}\label{bilinear:ex}
Let $n=3$ and $\lambda=(1)$ so that $\check{B}_n^\lambda=(0)$ and
$m_\lambda=E_1$. We order the basis~\eqref{bcell:1} for the module
$S^\lambda$ as $\mathbf{v}_1=E_1$, $\mathbf{v}_2=E_1T_2$ and
$\mathbf{v}_3=E_1T_2T_1$ and, with respect to this ordered basis,
the Gram matrix $\langle\mathbf{v}_i,\mathbf{v}_j\rangle$ of the
bilinear form~\eqref{formdef:1} is
\begin{align*}
\left[
\begin{matrix}
z &r &1\\
r &z+(q-q^{-1})(r-r^{-1})&r^{-1}\\
1 &r^{-1}& z
\end{matrix}
\right].
\end{align*}
The determinant of the Gram matrix given above is
\begin{align}\label{det3.1}
\frac{(r-1)^2(r+1)^2(q^3+r)(q^3r-1)}{r^3(q-1)^3(q+1)^3}.
\end{align}
\end{example}
\begin{remark}
(i) Let $\kappa$ be a field and
$\hat{r},\hat{q},(\hat{q}-\hat{q}^{-1})$ be units in $\kappa$. The
assignments $\varphi:r\mapsto \hat{r}$ and $\varphi:q\mapsto\hat{q}$
determine a homomorphism $R\to \kappa$, giving $\kappa$ an
$R$--module structure.  We refer to the specialisation
$B_n(\hat{q},\hat{r})=B_n(q,r)\otimes_R \kappa$ as a B--M--W algebra
over $\kappa$. If $0\le f\le[n/2]$ and $\lambda$ is a partition of
$n-2f$ then the cell module $S^\lambda\otimes_R\kappa$ for
$B_n(\hat{q},\hat{r})$ admits a symmetric associative bilinear form
which is related to the generic form~\eqref{formdef:1} in an obvious
way.

(ii) Whenever the context is clear and no possible confusion will
arise, the abbreviation $S^\lambda$ will be used for the
$B_n(\hat{q},\hat{r})$--module $S^\lambda\otimes_R\kappa$.
\end{remark}
The proof of Theorem~\ref{saruthm} given in~\cite{saru} rests upon
the following facts, respectively Proposition~3.2
of~\cite{wenzlqg} and Proposition~3.3 of~\cite{saru}, stated below
for later reference.
\begin{lemma}\label{bquot}
Let $f$ be an integer, $0\le f\le [n/2]$, write $C_f$ for the
subalgebra of $B_n(q,r)$ generated by the elements
$T_{2f+1},\dots,T_{n-1}$, and $I_f$ for the two sided ideal of
$C_f$ generated by the element $E_{2f+1}$. Then the map defined on algebra generators of $\mathscr{H}_{n-2f}(q^2)$ by
\begin{align*}
\phi:X_i\mapsto T_{2f+i}+I_f,&&\text{for $1\le i<n-2f$,}
\end{align*}
and extended to all of $\mathscr{H}_{n-2f}$ by $\phi(h_1h_2)=\phi(h_1)\phi(h_2)$ whenever $h_1,h_2\in\mathscr{H}_{n-2f}$, is an an algebra isomorphism $\mathscr{H}_{n-2f}(q^2)\cong
C_f/I_f$.
\end{lemma}
\begin{lemma}\label{gomi}
Let $f$ be an integer, $0\le f<[n/2]$, and $C_f$ and $I_f$ be as
in Lemma~\ref{bquot} above. If $i$ is an integer, $2f< i<n$, and
$b\in C_f$, then
\begin{align*}
E_1E_3\cdots E_{2f-1} bE_i \equiv E_1E_3\cdots E_{2f-1}E_ib \equiv
0\mod B^{f+1}_n.
\end{align*}
\end{lemma}
Since $\mathscr{H}_{n-2f}(q^2)\subseteq\mathscr{H}_n(q^2)$ is
generated by $\{X_j:1\le j<n-2f\}$, from
Lemmas~\ref{bquot} and~\ref{gomi} we obtain Corollary~\ref{twiso};
\emph{cf.} Section~3 of~\cite{saru}.
\begin{corollary}\label{twiso}
If $f$ is an integer, $0\le f<[n/2]$, then there is a well defined
$R$--module homomorphism $\vartheta_f:\mathscr{H}_{n-2f}(q^2)\to
B^f_n/B^{f+1}_n$, determined by
\begin{align*}
\vartheta_f:X_{\hat{v}} \to E_1E_3\cdots E_{2f-1}T_{v}+B^{f+1}_n,
\end{align*}
where $v=s_{i_1}s_{i_2}\cdots s_{i_d}$ is a permutation in
$\langle s_i : 2f<i<n \rangle$ and $\hat{w}$ is the permutation
$\hat{v}=s_{i_1-2f}s_{i_2-2f}\cdots s_{i_d-2f}$. Additionally, the map
$\vartheta_f$ satisfies the property
\begin{align}\label{nearh}
\vartheta_f(X_{\hat{v}} X_j)=\vartheta_f(X_{\hat{v}})T_{2f+j},
\end{align}
whenever $1\le j<n-2f$.
\end{corollary}
\begin{remark}
The fact that $\vartheta_f$ is an isomorphism of $R$--modules was not used in the proof of Theorem~\ref{saruthm}; however it may be deduced from Theorem~\ref{saruthm} which implies that the dimension over $R$ of the image space of $\vartheta_f$ is equal to the dimension of $\mathscr{H}_{n-2f}(q^2)$ over $R$. 
\end{remark}
\begin{lemma}\label{ecor}
Let $f$ be an integer, $0< f \le [n/2]$. If $b\in B_n(q,r)$, $w\in\mathscr{D}_{f,n}$, and $1\le i<n$, then there exist $a_{u,v}$ in $R$, for $u$ in $\langle s_i : 2f < i <n\rangle$ and $v$ in $\mathscr{D}_{f,n}$, uniquely determined by
\begin{align}\label{blah}
E_1E_3\cdots E_{2f-1} T_wb\equiv \sum_{u,v} a_{u,v}E_1E_3\cdots
E_{2f-1}T_uT_v\mod{B^{f+1}_n}.
\end{align} 
\end{lemma}
\begin{proof}
For the uniqueness of the expression~\eqref{blah}, observe that there is a one--to--one map
\begin{align*}
E_1E_3\cdots E_{2f-1} T_u T_v+B_n^{f+1}\mapsto \sum_{\substack{\mathfrak{s},\mathfrak{t}\in\STD_n(\lambda)\\\lambda\vdash n-2f}}a_{\mathfrak{s},\mathfrak{t}}\,T_{d(\mathfrak{s})}^*m_\lambda T_{d(\mathfrak{t})}T_v+{B_n^{f+1}},
\end{align*}
for $u\in \langle s_j\,:\,2f<j<n\rangle$ and $v\in\mathscr{D}_{f,n}$, determined by the map $\vartheta_f$ and the transition between the basis $\{X_w:w\in\mathfrak{S}_{n-2f}\}$ and the Murphy basis for $\mathscr{H}_{n-2f}(q^2)$, where the expression on the right hand side above is an $R$--linear sum of the basis elements for $B_n^f/B_n^{f+1}$ given by Theorem~\ref{saruthm}.  

The proof of the lemma makes repeated use of the following fact. If $u'\in\langle s_i: 2f<i<n\rangle$ and $v'\in\mathfrak{S}_n$, then $E_1E_3\cdots E_{2f-1} T_{u'} T_{v'}$ is expressible as a sum of the form that appears on the right hand side of~\eqref{blah}. To see this, first note that, given an integer $i$ with $2f<i<n$ and $(i+1)v'<(i)v'$,
\begin{align*}
T_{u'} T_{v'}
=
\begin{cases}
T_{u's_i} T_{s_iv'},&\text{if $\ell(u')<\ell(u's_i)$;}\\
(T_{u's_i}+(q-q^{-1})(T_{u'}-r^{-1}T_{u's_i}E_i))T_{s_iv},&\text{otherwise.}
\end{cases}
\end{align*}
Thus, using Lemma~\ref{gomi}, we have $a_{u,v}\in R$, for $u\in\langle s_i\, :\, 2f< i <n\rangle$ and $v\in\mathfrak{S}_n$, such that  
\begin{align*}
E_1E_3\cdots E_{2f-1} T_{u'}T_{v'}\equiv\sum_{u,v} a_{u,v}E_1E_3\cdots E_{2f-1}T_uT_v\mod{B_n^{f+1}},
\end{align*}
where $(i)v<(i+1)v$, for $2f<i<n$, whenever $a_{u,v}\ne0$ in the above expression. Noting that $E_1E_3\cdots E_{2f-1}T_v=r^{-1}E_1E_3\cdots E_{2f-1}T_{s_{2i-1}v}$ if $1\le i \le f$ and $\ell(s_{2i-1}v)<\ell(v)$, and applying Proposition~3.7 or Corollary~3.1 of~\cite{saru}, we may assume that $v\in\mathscr{D}_{f,n}$, whenever $a_{u,v}\ne 0$ in the above expression.

Proceeding with the proof of the lemma, first consider the case where $b=E_i$ for some $1\le i<n$. Let $k=(i)w^{-1}$ and $l=(i+1)w^{-1}$. If $(i+1)w^{-1}<(i)w^{-1}$, then $T_wE_i=r^{-1}T_{ws_i}E_i$, where $ws_i\in\mathscr{D}_{f,n}$. We may therefore suppose that $k<l$. Using Proposition~3.4 of~\cite{saru},
\begin{align}\label{barca}
T_wE_i=
\begin{cases}
E_k T_w,&\text{if $l=k+1$;}\\
T_{l-1}^{\varepsilon_{l-1}}T_{l-2}^{\varepsilon_{l-2}}\cdots T_{k+1}^{\varepsilon_{k+1}}E_kT_{w'},&\text{otherwise,} 
\end{cases}
\end{align}
where $w'=s_{k+1}s_{k+2}\cdots s_{l-1}w$ and, for $k< j <l$, 
\begin{align*}
\varepsilon_j=
\begin{cases}
1,&\text{if $(j)w<i+1$;}\\
-1,&\text{otherwise.}
\end{cases}
\end{align*}
Considering the two cases in~\eqref{barca} separately, multiply both sides of the expression~\eqref{barca} by $E_1E_3\cdots E_{2f-1}$. If $l=k+1$, then
\begin{align}\label{mago}
E_1E_3\cdots E_{2f-1}T_wE_i=
\begin{cases}
zE_1E_3\cdots E_{2f-1}T_w,&\text{if $k<2f$ and $k$ is odd;}\\
E_1E_3\cdots E_{2f-1} T_kT_{k-1}T_w, &\text{if $k\le 2f$ and $k$ is even;}\\
E_1E_3\cdots E_{2f-1} E_{k}T_w, &\text{if $2f<k$.}
\end{cases}
\end{align}
By Proposition~3.8 of~\cite{saru}, there exist $a_{v'}\in R$, for $v'\in\mathfrak{S}_n$ such that, given $w'\in\mathfrak{S}_n$ satisfying $(2j)w'+1=(2j+1)w'$, together with $\varepsilon_{2j-1},\varepsilon_{2j}\in\{\pm1\}$, 
\begin{align}\label{saru:prop3.8}
E_{2j-1}T_{2j}^{\varepsilon_{2j}}T_{2j-1}^{\varepsilon_{2j-1}}T_{w'}=
\sum_{v'\in\mathfrak{S}_n} a_{v'}E_{2j-1}T_{v'}.
\end{align}
Using~\eqref{saru:prop3.8} with $k=2j$, the term appearing in the second case on the right hand side of~\eqref{mago} can be rewritten as
\begin{align*}
E_1E_3\cdots E_{2f-1}T_kT_{k-1}T_{w'}=\sum_{v'\in\mathfrak{S}_n} a_{v'}E_{1}E_3\cdots E_{2j-1}T_{v'}.
\end{align*}
As already noted, the right hand side of the above expression may be rewritten modulo $B_n^{f+1}$ as an $R$--linear combination of the required form. On the other hand, the term appearing on the right in the last case in~\eqref{mago} above is zero modulo $B_n^{f+1}$.

The second case on the right hand side of~\eqref{barca} gives rise to three sub--cases as follows. First, if $2f<k<n$, then  
\begin{align*}
E_1E_3\cdots E_{2f-1}T_{l-1}^{\varepsilon_{l-1}}T_{l-2}^{\varepsilon_{l-2}}\cdots T_{k+1}^{\varepsilon_{k+1}}E_kT_{w'}\equiv 0\mod{B_n^{f+1}};
\end{align*}
if $1\le k<2f$ and $k$ is odd, then 
\begin{align}\label{gromple:0}
 E_kT_{l-1}^{\varepsilon_{l-1}}T_{l-2}^{\varepsilon_{l-2}}\cdots T_{k+1}^{\varepsilon_{k+1}}E_kT_{w'}=
r^{\varepsilon_{k+1}}E_kT_{l-1}^{\varepsilon_{l-1}}T_{l-2}^{\varepsilon_{l-2}}\cdots T_{k+2}^{\varepsilon_{k+2}}T_{w'};
\end{align}
if $1 < k\le 2f$ and $k$ is even, then 
\begin{align}\label{gromple:2}
 E_{k-1}T_{l-1}^{\varepsilon_{l-1}}T_{l-2}^{\varepsilon_{l-2}}\cdots T_{k+1}^{\varepsilon_{k+1}}E_kT_{w'}
=E_{k-1}T_{l-1}^{\varepsilon_{l-1}}T_{l-2}^{\varepsilon_{l-2}}\cdots T_{k+1}^{\varepsilon_{k+1}}T_{k}T_{k-1}T_{w'}.
\end{align}
When $1\le k<2f$ and $k$ is odd, using~\eqref{barca} and~\eqref{gromple:0}, and successively applying~\eqref{saru:prop3.8} with $j=k,k-2,\dots,$ we obtain
\begin{align*}
E_1E_3\cdots E_{2f-1}T_{l-1}^{\varepsilon_{l-1}}T_{l-2}^{\varepsilon_{l-2}}\cdots T_{k+1}^{\varepsilon_{k+1}} E_{k}T_{w'}
=\sum_{v'\in\mathfrak{S}_n} a_{v'} T_{l-1}^{\varepsilon_{l-1}}T_{l-2}^{\varepsilon_{l-2}}\cdots T_{2f+1}^{\varepsilon_{2f+1}} E_1E_3\cdots E_{2f-1} T_{v'}
\end{align*}
where $T_{l-1}^{\varepsilon_{l-1}}T_{l-2}^{\varepsilon_{l-2}}\cdots T_{2f+1}^{\varepsilon_{2f+1}}$ can be expressed as a sum
\begin{align*}
T_{l-1}^{\varepsilon_{l-1}}T_{l-2}^{\varepsilon_{l-2}}\cdots T_{2f+1}^{\varepsilon_{2f+1}}
=\sum_{u'\in \langle s_j \,:\, 2f<j<n\rangle} a_{u'}T_{u'} +b',
\end{align*}
and $b'$ lies in the two sided ideal of $\langle T_j \,:\,2f<j<n\rangle$ generated by $E_{2f+1}$. Since $b'$ satisfies $E_1E_3\cdots E_{2f-1}b'\in B_n^{f+1}$, it follows that 
\begin{multline*}
E_1E_3\cdots E_{2f-1}T_{l-1}^{\varepsilon_{l-1}}T_{l-2}^{\varepsilon_{l-2}}\cdots T_{k+1}^{\varepsilon_{k+1}} E_{k}T_{w'}\\
\equiv\sum_{\substack{v'\in\mathfrak{S}_n \\ u'\in\langle s_j \,:\, 2f<j<n\rangle }} a_{u',v'} E_1E_3\cdots E_{2f-1} T_{u'}T_{v'}\mod{B_n^{f+1}}.
\end{multline*}
As already noted, the right hand side of the above expression may be rewritten modulo $B_n^{f+1}$ as an $R$--linear combination of the required form. In the same way, if $1<k\le 2f$ and $k$ is even, then using~\eqref{gromple:2}, we obtain the product
\begin{align*}
E_1E_3\cdots E_{2f-1}T_{l-1}^{\varepsilon_{l-1}}T_{l-2}^{\varepsilon_{l-2}}\cdots T_{k+1}^{\varepsilon_{k+1}}T_{k}T_{k-1}T_{w'}
\end{align*}
which is also expressible as a sum of the required form using the arguments above. Thus we have shown that the lemma holds in case $1\le i<n$ and $b=E_i$. 

Let $w\in\mathscr{D}_{f,n}$. If $1\le i <n$, and $\ell(w)<\ell(ws_i)$ then
\begin{align*}
E_1E_3\cdots E_{2f-1} T_w T_i =E_1E_3\cdots E_{2f-1} T_{ws_i},
\end{align*}
and, if $\ell(ws_i)<\ell(w)$, then 
\begin{align*}
E_1E_3\cdots E_{2f-1} T_{w}T_i
=
E_1E_3\cdots E_{2f-1}(T_{ws_i} + (q-q^{-1})(T_w-r^{-1}T_{ws_i}E_i)).
\end{align*}
We have already observed that the terms appearing on the right hand side in each of the two above expressions may be expressed as an $R$--linear combination of the required form. Thus we have shown that the lemma holds when $b\in\{T_i:1\le i<n\}$. 

Now, given that the lemma holds when $b\in\{T_i:1\le i<n\}$, Lemma~\ref{gomi} shows that any product 
\begin{align*}
E_1E_3\cdots E_{2f-1} T_u T_v T_i, &&\text{for $u\in\langle s_i : 2f<i <n\rangle$ and $v\in\mathscr{D}_{f,n}$,}
\end{align*}
can also be written as an $R$--linear combination of the form appearing on the right hand side of~\eqref{blah}. Since $\{T_i:1\le i < n\}$ generates $B_n(q,r)$, the proof of the lemma is complete.
\end{proof}
If $f$ is an integer, $0\le f\le [n/2]$, and $\mu$ is a partition of
$n-2f$, define $L^{\mu}$ to be the right $B_n(q,r)$--submodule of
$B_n^f/B_n^{f+1}$ generated by the element $m_\mu+B^{f+1}_n$. The
next result will be used in Section~\ref{resf} below; we refer to
the definition of the element
$c_{\mathsf{S}\mathfrak{t}}\in\mathscr{H}_{n}(q^2)$ given
in~\eqref{some:1}.
\begin{lemma}\label{permod:2}
Let $f$ be an integer, $0\le f\le [n/2]$, and given partitions
$\lambda,\mu$ of $n-2f$, with $\lambda\unrhd\mu$, define
\begin{align*}
m_{\mathsf{S}\mathfrak{t}}=\sum_{\substack{\mathfrak{s}\in \STD_n(\lambda)\\ \mu(\hat{\mathfrak{s}})=\mathsf{S}}}q^{\ell(d(\mathfrak{s}))}T_{d(\mathfrak{s})}^*m_\lambda T_{d(\mathfrak{t})}, &&\text{for $\mathsf{S}\in\mathcal{T}_0(\lambda,\mu)$ and $\mathfrak{t}\in\STD_n(\lambda)$.}
\end{align*}
Then the collection
\begin{align}\label{lbasis}
\left\{m_{\mathsf{S}\mathfrak{t}}T_v+B_n^{f+1}\,\bigg|\,
\begin{matrix}
\text{for $\mathsf{S}\in\mathcal{T}_0(\lambda,\mu)$, $\mathfrak{t} \in \STD_n(\lambda)$, }\\
\text{$\lambda\vdash n-2f$ and $v\in\mathscr{D}_{f,n}$}
\end{matrix}
\right\}
\end{align}
freely generates $L^\mu$ as an $R$--module.
\end{lemma}
\begin{proof}
If $b\in B_{n}(q,r)$ and $w\in\mathscr{D}_{f,n}$, then by the previous lemma, there exist $a_{u,v}\in R$, for $u\in\langle s_i : 2f<i<n\rangle$ and $v\in\mathscr{D}_{f,n}$ such that 
\begin{align*}
E_1E_3\cdots E_{2f-1} T_wb\equiv \sum_{u,v} a_{u,v}E_1E_3\cdots
E_{2f-1}T_uT_v\mod{B^{f+1}_n}.
\end{align*}
Multiplying both sides of the above expression by $x_\mu$ on the left, and using the property~\eqref{nearh} and Theorem~\ref{permod}, we obtain $a_{\mathsf{S},\mathfrak{t}}\in R$, for $\mathsf{S}\in\mathcal{T}_0(\lambda,\mu)$, $\mathfrak{t}\in\STD_n(\lambda)$ and $\lambda\vdash n-2f$, such that
\begin{align*}
&m_\mu T_w b + B_{n}^{f+1}
= \sum_{u,v} a_{u,v}\,E_1E_3\cdots E_{2f-1}x_\mu T_uT_v+B_{n}^{f+1}\\
&\quad =\sum_{u,v} a_{u,v}\,\vartheta_f(c_\mu X_{\hat{u}})T_v=\sum_{u,v} a_{u,v}\sum_{\substack{\mathsf{S}\in\mathcal{T}_0(\lambda,\mu)\\ \mathfrak{t}\in\STD_n(\lambda) }} a_{\mathsf{S},\mathfrak{t}}\,\vartheta_f(c_{\mathsf{S}\hat{\mathfrak{t}}})T_v\\
&\quad=\sum_{u,v} a_{u,v}\sum_{\substack{\mathsf{S}\in\mathcal{T}_0(\lambda,\mu)\\ \mathfrak{t}\in\STD_n(\lambda)}} a_{\mathsf{S},\mathfrak{t}}\, m_{\mathsf{S}\mathfrak{t}} T_v +{B}_n^{f+1}.
\end{align*}
This proves the spanning property of the set~\eqref{lbasis}. The fact that each element of the set~\eqref{lbasis} lies in $L^\mu$ follows from an argument similar to the above, using Theorem~\ref{permod} and the property~\eqref{nearh}. We now outline the proof of the linear independence of the elements of~\eqref{lbasis} over $R$. 

(i) Let $\{\mathsf{S}_i : 1 \le i \le k \}$ be the semistandard tableaux of type $\mu$, ordered so that $\mathsf{S}_i\in\mathcal{T}_0(\lambda_i,\mu)$ and $j\ge i$ whenever $\lambda_i\unrhd\lambda_j$, and take $L_i$ to denote the $R$--module generated by 
\begin{align*}
\big\{m_{\mathsf{S}_j\mathfrak{t}}T_v+B_n^{f+1} :
\text{ $1 \le j \le i$, $\mathfrak{t} \in \STD_n(\lambda_j)$ and $v\in\mathscr{D}_{f,n}$}
\big\}.
\end{align*}

(ii) Using the property~\eqref{nearh} and Theorem~\ref{permod} as above, it is shown that the $R$--module homomorphism $L_i/L_{i-1}\to S^{\lambda_i}$ defined, for $\mathfrak{t}\in\STD_n(\lambda_i)$ and $w\in\mathscr{D}_{f,n}$, by
\begin{align}\label{pest:10}
 m_{\mathsf{S}_i\mathfrak{t}} T_w + L_{i-1} \mapsto m_{\lambda_i} T_{d(\mathfrak{t})} T_w+\check{B}_n^{\lambda_i}
 \end{align}
is an isomorphism of right $B_n(q,r)$--modules. Thus, analogous to the filtration of each permutation module of the Iwahori--Hecke algebra of the symmetric group given in Corollary~4.10 of~\cite{mathas:ih}, there is a filtration of $L^\mu$ by $B_n(q,r)$--modules
\begin{align}\label{lfilt:0}
(0) = L_0 \subseteq L_1 \subseteq\cdots \subseteq L_{k}=L^\mu,
\end{align}
wherein each factor $L_{i}/L_{i-1}$ is isomorphic to a cell module $S^{\lambda_i}$ for $B_n(q,r)$. 

(iii) From~\eqref{lfilt:0}, it is deduced that $\dim_R (L^\mu) =\sum_{i=1}^k \dim_R(S^{\lambda_i})$. Since this sum coincides with the order of the set~\eqref{lbasis} obtained by simply counting, the linear independence over $R$ of the elements of~\eqref{lbasis} now follows.
\end{proof}
\section{Representation Theory Over a Field}\label{afield}
We state for later reference some consequences, for B--M--W algebras over a field, of the theory of cellular algebras constructed in~\cite{grahamlehrer}. These results of C.C.~Xi appeared in~\cite{xi:qheredity}.
\begin{proposition}
Let $B_n(\hat{q},\hat{r})$ be a B--M--W algebra over a field
$\kappa$. If $f$ is an integer, $0\le f\le[n/2]$, and $\lambda$ is a
partition of $n-2f$, then the radical
\begin{align*}
\RAD(S^\lambda)=\{\mathbf{v}\in
S^\lambda:\text{$\langle\mathbf{v},\mathbf{u}\rangle = 0$
 for all $\mathbf{u}\in S^\lambda$}\}
\end{align*}
of the form on $S^\lambda$ determined by~\eqref{formdef:1} is a
$B_n(\hat{q},\hat{r})$--submodule of $S^\lambda$.
\end{proposition}
\begin{proposition}
Let $B_n(\hat{q},\hat{r})$ be a B--M--W algebra over a field $\kappa$, and suppose that $f,f'$ are integers $0\le f,f'\le [n/2]$, and $\lambda,\mu$ are partitions of $n-2f$ and $n-2f'$ respectively. If $M$ is a $B_n(\hat{q},\hat{r})$--submodule of $S^\lambda$, and $\psi:S^\mu\to S^\lambda/M$ is a non--trivial $B_n(\hat{q},\hat{r})$--module homomorphism, then $\lambda\unrhd\mu$. 
\end{proposition}
Let $B_n(\hat{q},\hat{r})$ be a B--M--W algebra over a field
$\kappa$. If $f$ is an integer with $0\le f\le [n/2]$, and $\lambda$
is a partition of $n-2f$, define the $B_n(\hat{q},\hat{r})$--module
$D^\lambda=S^\lambda/\RAD(S^\lambda)$.
\begin{theorem}\label{g-lthm:1}
If $\kappa$ is a field and $B_n(\hat{q},\hat{r})$ is a B--M--W
algebra over $\kappa$, then
\begin{align*}
\{D^\lambda:\text{$D^\lambda\ne 0$, $\lambda\vdash n-2f$ and $0\le
f\le[n/2]$}\}
\end{align*}
is a complete set of pairwise inequivalent
irreducible $B_n(\hat{q},\hat{r})$--modules.
\end{theorem}
\begin{theorem}\label{g-lthm:2}
Let $\kappa$ be a field and $B_n(\hat{q},\hat{r})$ be a B--M--W
algebra over $\kappa$. Then the following statements are equivalent.
\begin{enumerate}
    \item $B_n(\hat{q},\hat{r})$ is (split) semisimple.
    \item $S^\lambda=D^\lambda$ for all $\lambda\vdash n-2f$ and
    $0\le f\le[n/2]$.
    \item $\RAD(S^\lambda)=0$ for all $\lambda\vdash n-2f$ and
    $0\le f\le[n/2]$.
\end{enumerate}
\end{theorem}
\section{Restriction}\label{resf}
Given an integer, $1\le i\le n$, regard $B_{i}(q,r)$ as the
subalgebra of $B_n(q,r)$ generated by the elements
$T_1,T_2,\cdots,T_{i-1}$, thereby obtaining the tower
\begin{align}\label{tower:2}
R=B_1(q,r)\subseteq B_{2}(q,r)\subseteq \cdots \subseteq
B_{n}(q,r).
\end{align}
If $V$ is a $B_{n}(q,r)$--module, using the
identification~\eqref{tower:2}, we write $\RES(V)$ for the
restriction of $V$ to $B_{n-1}(q,r)$.

In order to construct a basis for the cell module
$S^\lambda$ which behaves well with respect to both
restriction in the tower~\eqref{tower:2} and with respect to the
action of the Jucys--Murphy operators in $B_n(q,r)$, we first
consider in this section the behaviour of the cell
module $S^\lambda$ under restriction from $B_n(q,r)$ to
$B_{n-1}(q,r)$. This description of the restriction functor on the
cell modules for the B--M--W algebras given here will be used in
Section~\ref{newbasis} to construct a basis for the cell module
$S^\lambda$ which behaves regularly with respect to restriction in
the tower~\eqref{tower:2} and with respect to the Jucys--Murphy
operators in $B_n(q,r)$.

The material of this section is motivated by the 
Wedderburn decomposition of the semisimple B--M--W algebras over a field
$\mathbb{C}(\hat{q},\hat{r})$ given by H.~Wenzl in~\cite{wenzlqg},
and by the bases for the B--M--W algebras indexed by paths in the
Bratteli diagram associated with the B--M--W algebras, constructed
in the semisimple setting over $\mathbb{C}(\hat{q},\hat{r})$, by
R.~Leduc and A.~Ram in~\cite{ramleduc:rh}. As made clear
by~\cite{ramleduc:rh} and~\cite{wenzlqg}, paths in the Bratteli
diagram associated with the B--M--W algebras provide the most
natural generalisation to our setting of the standard tableaux from
the representation theory of the symmetric group. However, while the
bases constructed in Section~\ref{newbasis} and
in~\cite{ramleduc:rh} are both indexed by paths in the appropriate
Bratteli diagram, we have sought a generic basis over a ring
$R=\mathbb{Z}[q^{\pm1},r^{\pm1},(q-q^{-1})^{-1}]$. Thus the
construction here will not require the assumptions about
semisimplicity which are necessary in~\cite{ramleduc:rh}.

Let $f$ be an integer, $0\le f\le [n/2]$, and $\lambda$ be a
partition of $n-2f$. Henceforth, write $\mu\to \lambda$ to mean that
either
\begin{enumerate}
\item $\mu$ is a partition of $n-2f+1$ and the diagram
$[\mu]$ is obtained by adding a node to the diagram $[\lambda]$
or, \item $\mu$ is a partition of $n-2f-1$ and the diagram
$[\mu]$ is obtained by deleting a node from the diagram
$[\lambda]$,
\end{enumerate}
as illustrated in the truncated Bratteli diagram associated with
$B_n(q,r)$ displayed in~\eqref{b-m-wbrattelli} below (Section~5
of~\cite{ramleduc:rh}).
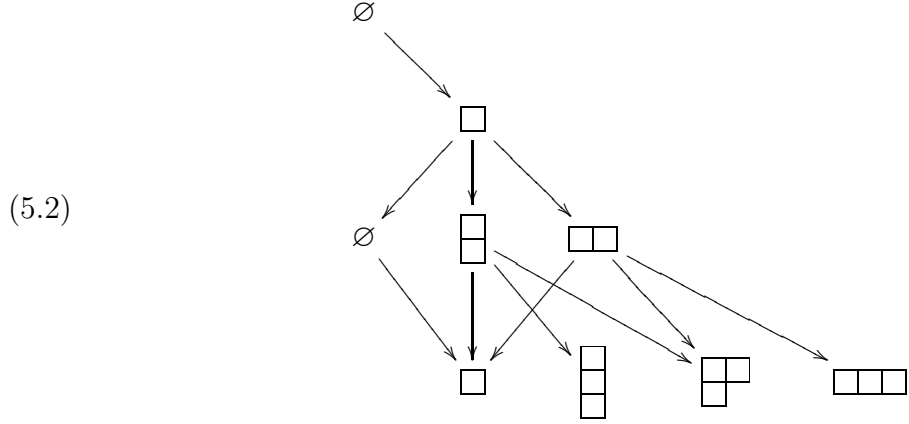
\begin{align}\label{b-m-wbrattelli}
\begin{matrix}
\xymatrix{ \varnothing\ar[dr]  & & & & \\
 & \text{\tiny\Yvcentermath1$\yng(1)$}\ar[dl]\ar[d]\ar[dr] & & & \\
 \varnothing\ar[dr] & \text{\tiny\Yvcentermath1$\yng(1,1)$}\ar[d]\ar[dr]\ar[drr]&
 \text{\tiny\Yvcentermath1$\yng(2)$}\ar[dl]\ar[dr]\ar[drr]& & \\
& \text{\tiny\Yvcentermath1$\yng(1)$} &
\text{\tiny\Yvcentermath1$\yng(1,1,1)$}&
 \text{\tiny\Yvcentermath1$\yng(2,1)$}
 & \text{\tiny\Yvcentermath1$\yng(3)$}}
\end{matrix}
\end{align}

Let $f$ be an integer, $0\le f\le [n/2]$, and $\lambda$ be a
partition of $n-2f$ with $t$ removable nodes and suppose that
\begin{align}\label{chain:0}
\mu^{(1)}\rhd\mu^{(2)}\rhd\cdots\rhd\mu^{(t)}
\end{align}
is the ordering of the set $\{\mu:\text{$\mu \to\lambda$ and
$|\lambda|>|\mu|$}\}$ by dominance order on partitions. For each
partition $\mu^{(k)}$ in the list~\eqref{chain:0}, define an element
\begin{align}\label{ydef:1}
 y^\lambda_{\mu^{(k)}}=m_\lambda
T_{d(\mathfrak{s})}+\check{B}_n^\lambda &&\text{where
$\mathfrak{s}|_{n-1}=\mathfrak{t}^{\mu^{(k)}}$,}
\end{align}
and let $N^{\mu^{(k)}}$ denote the $B_{n-1}(q,r)$--submodule of
$S^\lambda$ generated by
\begin{align*}
\{\,y^\lambda_{\mu^{(k)}}
T_{d(\mathfrak{u})}:\text{$\mathfrak{u}\in\STD_{n-1}(\mu^{(k)})$}\};
\end{align*}
write $\check{N}^{\mu^{(k)}}$ for the
$B_{n-1}(q,r)$--submodule of $S^\lambda$ generated by
\begin{align*}
\{\,y^\lambda_{\mu^{(j)}}
T_{d(\mathfrak{u})}:\text{$\mathfrak{u}\in\STD_{n-1}(\mu^{(j)})$ and $j<k$}\,\}.
\end{align*}
\begin{example}\label{btenex1}
 Let $n=10$, $f=2$ and $\lambda=(3,2,1)$. Then
\begin{align*}
m_\lambda=E_1E_3\sum_{w\in\mathfrak{S}_\lambda}q^{\ell(w)}T_w=E_1E_3(1+qT_5)(1+qT_6+q^2T_6T_5)(1+qT_8)
\end{align*}
and the elements $y^\lambda_{\mu^{(k)}}$, for each partition
$\mu^{(k)}\to\lambda$ with $|\lambda|>|\mu^{(k)}|$, are as
follows.
\begin{enumerate}
\item If $\mu^{(1)}=(3,2)$, then
$\mathfrak{t}^\mu=\mathfrak{s}|_{n-1}$, where
$\mathfrak{s}=\text{\tiny\Yvcentermath1$\young(567,89,\dten)$}$\,,
so
\begin{align*}
y_{\mu^{(1)}}^\lambda=m_\lambda+\check{B}_n^\lambda.
\end{align*}
\item If $\mu^{(2)}=(3,1,1)$ and
$\mathfrak{s}=\text{\tiny\Yvcentermath1$\young(567,8\dten,9)$}$\,,
then $\mathfrak{t}^{\mu^{(2)}}=\mathfrak{s}|_{n-1}$, so
\begin{align*}
y_{\mu^{(2)}}^\lambda=m_\lambda
T_{d(\mathfrak{s})}+\check{B}_n^\lambda=m_\lambda
T_9+\check{B}_n^\lambda.
\end{align*}
\item If $\mu^{(3)}=(2,2,1)$, then
$\mathfrak{t}^{\mu^{(3)}}=\mathfrak{s}|_{n-1}$, where
$\mathfrak{s}=\text{\tiny\Yvcentermath1$\young(56\dten,78,9)$}$\,,
so
\begin{align*}
y_{\mu^{(3)}}^\lambda=m_\lambda
T_{d(\mathfrak{s})}+\check{B}_n^\lambda=m_\lambda
T_7T_8T_9+\check{B}_n^\lambda.
\end{align*}
\end{enumerate}
\end{example}
Write
$\mathscr{D}_{f,n-1}=\{v\in\mathscr{D}_{f,n}:(n)v=n\}$, so identifying $\mathscr{D}_{f,n-1}\subseteq\mathscr{D}_{f,n}$.
\begin{lemma}\label{bres1}
Let $f$ be an integer, $0\le f\le [n/2]$, and $\lambda$ be a
partition of $n-2f$. If $\mu$ is a partition with
$|\lambda|>|\mu|$ and $\mu\to\lambda$, then $N^{\mu}/\check{N}^\mu$ is
the $R$--module freely generated by
\begin{align*}
\{y_\mu^\lambda T_{d(\mathfrak{u})}T_w+\check{N}^\mu :
\text{$\mathfrak{u}\in\STD_{n-1}(\mu)$ and $w\in\mathscr{D}_{f,n-1}$}
\}.
\end{align*}
Additionally, the map defined, for
$\mathfrak{u}\in\STD_{n-1}(\mu)$ and $w\in\mathscr{D}_{f,n-1}$,
by
\begin{align}\label{reshom}
y^\lambda_\mu T_{d(\mathfrak{u})}T_w+\check{N}^\mu\mapsto m_\mu
T_{d(\mathfrak{u})}T_w+\check{B}_{n-1}^\mu
\end{align}
determines an isomorphism $N^\mu/\check{N}^\mu\cong S^\mu$ of
$B_{n-1}(q,r)$--modules.
\end{lemma}
\begin{proof}
Let $b\in B_{n-1}(q,r)$ and $w\in\mathscr{D}_{f,n-1}$. By Lemma~\ref{ecor}, there exist $a_{u,v}\in R$, for $u\in\langle s_i : 2f< i<n-1\rangle$ and $v\in\mathscr{D}_{f,n-1}$, determined uniquely by 
\begin{align}\label{neet}
E_1E_3\cdots E_{2f-1} T_w b\equiv \sum_{u,v} a_{u,v} E_1E_3\cdots E_{2f-1} T_u T_v\mod{B_{n-1}^{f+1}}
\end{align}
Let $\mathfrak{v}\in\STD_n(\lambda)$ satisfy $\mathfrak{v}|_{n-1}=\mathfrak{t}^\mu$ so that $y_\mu^\lambda =m_\lambda T_{d(\mathfrak{v})}+\check{B}_{n}^\lambda$, and let $\mathfrak{u}\in\STD_{n-1}(\mu)$. Since $B_{n-1}^{f+1}\subset B_n^{f+1}$, we use~\eqref{neet} and Lemma~\ref{hres} to obtain $a_\mathfrak{s},a_\mathfrak{t}\in R$, for $\mathfrak{s}\in\STD_{n-1}(\mu)$ and $\mathfrak{t}\in\STD_{n}(\lambda)$ such that 
\begin{multline*}
m_\lambda T_{d(\mathfrak{v})}T_{d(\mathfrak{u})}T_wb +B_{n}^{f+1}=\sum_{u,v} a_{u,v}\,\vartheta_f(c_\lambda X_{d(\hat{\mathfrak{v}})}X_{d(\hat{\mathfrak{u}})}X_{\hat{u}})T_v\\
=\sum_{u,v} a_{u,v}\sum_{\mathfrak{s}\in\STD_{n-1}(\mu)} a_{\mathfrak{s}}\,\vartheta_f(c_\lambda X_{d(\hat{\mathfrak{v}})} X_{d(\hat{\mathfrak{s}})}) T_v\\
+ \sum_{u,v} a_{u,v}\sum_{\substack{\mathfrak{t}\in\STD_n(\lambda)\\ \SHAPE(\mathfrak{t}|_{n-1})\rhd\mu}} a_{\mathfrak{t}}\,\vartheta_f(c_\lambda X_{d(\hat{\mathfrak{t}})})T_v +\sum_{u,v}a_{u,v}\, \vartheta_f(h)T_v, 
\end{multline*}
where $h\in\check{\mathscr{H}}_{n-2f}^\lambda$ and $\vartheta_f(h)\subseteq \check{B}_n^\lambda$. We thus obtain,
\begin{multline*}
m_\lambda T_{d(\mathfrak{v})}T_{d(\mathfrak{u})} T_wb +B_{n}^{f+1}
=\sum_{u,v} a_{u,v}\sum_{\mathfrak{s}\in\STD_{n-1}(\mu)} a_{\mathfrak{s}}\,m_\lambda T_{d(\mathfrak{v})} T_{d(\mathfrak{s})} T_v\\
+ \sum_{u,v} a_{u,v}\sum_{\substack{\mathfrak{t}\in\STD_n(\lambda)\\ \SHAPE(\mathfrak{t}|_{n-1})\rhd\mu}} a_{\mathfrak{t}}\, m_\lambda T_{d(\mathfrak{t})}T_v +b', 
\end{multline*}
where $b'\in\check{B}_n^\lambda$. Since $\check{N}^{\mu}$ is generated as a $B_{n-1}(q,r)$ module by 
\begin{align*}
\{ m_\lambda T_{d(\mathfrak{t})}+\check{B}_{n}^\lambda:\text{$\mathfrak{t}\in\STD_n(\lambda)$ and $\SHAPE(\mathfrak{t}|_{n-1})\rhd\mu$}\},
\end{align*}
it follows that
\begin{align}\label{hora:1}
y_\mu^\lambda T_{d(\mathfrak{u})} T_w b \equiv \sum_{u,v} a_{u,v} \sum_{\mathfrak{s}\in\STD_{n-1}(\mu)} a_{\mathfrak{s}} \,y_\mu^\lambda T_{d(\mathfrak{s})}T_v\mod{\check{N}^\mu}.
\end{align}
 
Using~\eqref{neet} and Lemma~\ref{hres} again the $a_\mathfrak{s}$, for $\mathfrak{s}\in\STD_{n-1}(\mu)$, given above also satisfy 
\begin{multline*}
m_\mu T_{d(\mathfrak{u})}T_wb +B_{n-1}^{f+1}=\sum_{u,v} a_{u,v}\,\vartheta_f(c_\mu X_{d(\hat{\mathfrak{u}})}X_{\hat{u}})T_v\\
=\sum_{u,v} a_{u,v}\sum_{\mathfrak{s}\in\STD_{n-1}(\mu)} a_{\mathfrak{s}}\,\vartheta_f(c_\mu X_{d(\hat{\mathfrak{s}})}) T_v+\sum_{u,v}a_{u,v}\, \vartheta_f(h')T_v, 
\end{multline*}
where $h'\in\check{\mathscr{H}}_{n-2f-1}^\mu$. Since $\vartheta_f(h')\subseteq\check{B}_{n-1}^\mu$,
\begin{align}\label{hora:2}
m_\mu T_{d(\mathfrak{u})}T_wb +\check{B}_{n-1}^\mu =\sum_{u,v} a_{u,v}\sum_{\mathfrak{s}\in\STD_{n-1}(\mu)} a_{\mathfrak{s}}\,m_\mu T_{d(\mathfrak{s})}T_v +\check{B}_{n-1}^\mu.
\end{align}
Comparing coefficients in~\eqref{hora:1} and~\eqref{hora:2} shows that the $R$--module isomorphism~\eqref{reshom} is also a $B_{n-1}(q,r)$--module homomorphism.
\end{proof}
\begin{corollary}\label{resbcor1}
Let $f$ be an integer, $0\le f\le [n/2]$, and $\lambda$ be a
partition of $n-2f$. If $\mu$ is a partition of $n-2f-1$ with $\mu\to\lambda$, then ${N}^\mu$ is the $R$--module freely generated by
\begin{align*}
\big\{m_\lambda T_{d(\mathfrak{s})}T_v+\check{B}_n^\lambda:
\text{$\mathfrak{s}\in\STD_n(\lambda)$, $\SHAPE(\mathfrak{s}|_{n-1})\unrhd\mu$ and
$v\in\mathscr{D}_{f,n-1}$}\big\}.
\end{align*}
\end{corollary}
Let $f$ be an integer, $0< f\le [n/2]$, with $\lambda$ a partition
of $n-2f$ having $t$ removable nodes and $(p-t)$ addable nodes,
and suppose that
\begin{align}\label{chain:1}
\mu^{(t+1)}\rhd\mu^{(t+2)}\rhd\cdots\rhd\mu^{(p)}
\end{align}
is the ordering of  $\{\mu:\text{$\mu\to\lambda$ and
$|\mu|>|\lambda|$}\}$ by dominance order on partitions. By the
definition of the dominance order on partitions which we use here, the
list~\eqref{chain:0} can be extended as
\begin{align}\label{chain:2}
\mu^{(1)}\rhd\mu^{(2)}\rhd\dots\rhd\mu^{(t)}\rhd\mu^{(t+1)}\rhd\mu^{(t+2)}\rhd\cdots\rhd\mu^{(p)}.
\end{align}
In the manner of Lemma~\ref{bres1}, we seek to assign to each
partition $\mu^{(k)}$, with $k>t$, in the list~\eqref{chain:1}, a
$B_{n-1}(q,r)$--submodule $N^{\mu^{(k)}}$ of $S^\lambda$, and an
associated generator $y_{\mu^{(k)}}^\lambda+\check{N}^{\mu^{(k)}}$
in $S^\lambda/\check{N}^{\mu^{(k)}}$. To this end, first let
\begin{align}\label{wpdef}
w_p=s_{n-2}s_{n-3}\cdots s_{2f-1}s_{n-1}s_{n-2}\cdots s_{2f}
\end{align}
and write $N^{\mu^{(p)}}$ for the $B_{n-1}(q,r)$--submodule of
$S^\lambda$ generated by the element
\begin{align}\label{otherdef}
y_{\mu^{(p)}}^\lambda=m_\lambda T_{w_p}^{-1} +\check{B}^\lambda_n.
\end{align}

From the defining relations for $B_n(q,r)$, or using the
presentation for $B_n(q,r)$ in terms of tangles given
in~\cite{birwen}, it is readily observed that
$E_{2f-1}T_{w_p}^{-1}=E_{2f-1}T_{w^{-1}_p}$, and consequently that
$m_{\lambda}T_{w_p}^{-1}=m_\lambda T_{w^{-1}_p}$. Since $w_p^{-1}$
is an element of $\mathscr{D}_{f,n}$ with $(2f)w_p^{-1}=n$,
Corollary~\ref{resbcor1} implies that the element $m_\lambda
T_{w_p^{-1}}+\check{B}^\lambda_n$ is contained in the complement
of $N^{\mu^{(t)}}$ in $S^\lambda$. Furthermore, using the relation
$E_iT_{i+1}T_i=T_{i+1}T_iE_{i+1}$ it can be seen that
\begin{align*}
E_{2f-1}T_{w_p^{-1}}&=E_{2f-1}T_{2f}T_{2f+1}\cdots T_{n-2}T_{n-1}
T_{2f-1}T_{2f}
\cdots T_{n-3}T_{n-2}\\
&=T_{2f}T_{2f-1}T_{2f+1}T_{2f}\cdots T_{n-2}T_{n-3}
E_{n-2}T_{n-1}T_{n-2},
\end{align*}
whence, if $\mathfrak{s}\in\STD_{n}(\lambda)$,
\begin{equation}\label{horse}
\begin{split}
m_\lambda T_{d(\mathfrak{s})} T_{w_p}^{-1}&=m_\lambda
T_{d(\mathfrak{s})} T_{w_p^{-1}}=E_1E_3\cdots
E_{2f-3}E_{2f-1}x_\lambda T_{d(\mathfrak{s})}
T_{w_p^{-1}}\\
&=E_1E_3\cdots E_{2f-3}x_\lambda T_{d(\mathfrak{s})}
T_{v}E_{n-2}T_{n-1}T_{n-2},
\end{split}
\end{equation}
where $v=w_{p}^{-1}s_{n-2}s_{n-1}$ lies in $\mathscr{D}_{f,n-1}$.
From the defining relations of $B_n(q,r)$,
\begin{align*}
E_{n-2}T_{n-1}T_{n-2}E_{n-2}=E_{n-2},
\end{align*}
and, multiplying both sides of~\eqref{horse} on the right by the
element $E_{n-2}$,
\begin{align*}
m_\lambda T_{d(\mathfrak{s})} T_{w_p^{-1}} E_{n-2}= m_\lambda
T_{d(\mathfrak{s})} T_v,&&\text{where $v=w_p^{-1}s_{n-2}s_{n-1}$.}
\end{align*}
Since $v\in\mathscr{D}_{f,n-1}$, Corollary~\ref{resbcor1}  implies a
strict inclusion $N^{\mu^{(t)}}\subseteq N^{\mu^{(p)}}$ of
$B_{n-1}(q,r)$--modules.

Recall that if $\lambda$ is a partition of $n-2f$ and
$\mathfrak{s}\in\STD_{n}(\lambda)$, then $\hat{\mathfrak{s}}$ is
defined as the standard tableau obtained after relabelling the
entries of $\mathfrak{s}$ by $i\mapsto i-2f$ and $d(\mathfrak{s})$
is the permutation in $\langle s_i:2f<i<n\rangle$ defined by the
condition that $\mathfrak{s}=\mathfrak{t}^\lambda d(\mathfrak{s})$.
For the lemmas following, we also recall the definition of the
permutation $w_p$ in~\eqref{wpdef} above.

\begin{lemma}\label{big:0}
Let $f$ be an integer, $0< f\le[n/2]$, and $\lambda$ be a partition
of $n-2f$. Suppose that $\mu^{(p)}$ is minimal in
$\{\nu:\text{$\nu\to\lambda$ and $|\nu|>|\lambda|$}\}$ with
respect to the dominance order on partitions, let $\mu$ be a partition of $n-2f+1$ with $\mu\unrhd\mu^{(p)}$ and $\mathfrak{s}\in\STD_{n-1}(\mu)$ be a tableau such that $\mu^{(p)}(\hat{\mathfrak{s}})\in\mathcal{T}_0(\mu,\mu^{(p)})$. If
$\tau=\SHAPE(\mathfrak{s}|_{n-2})\rhd\lambda$, then
\begin{align*}
E_{2f-1}T_{w_p}^{-1}T_{d(\mathfrak{s})}^*m_\mu= E_1E_3\cdots
E_{2f-1}T_{w_p}^{-1}T_{d(\mathfrak{s})}^*x_\mu\equiv
0\mod{\check{B}^\lambda_n}.
\end{align*}
\end{lemma}
\begin{proof}
Recall that $x_{\mu}=\sum_{w\in\mathfrak{S}_\mu}q^{\ell(w)}T_w$ where
$\mathfrak{S}_\mu$ is the row stabiliser of
$\mathfrak{t}^\mu\in\STD_{n-1}(\mu)$ in $\langle s_i:2f-1\le
i<n-1\rangle$. Let
\begin{align*}
k=\min\{i:\text{$2f-1\le i\le n-2$ and
$(n-1)d(\mathfrak{s})^{-1}\le(i)d(\mathfrak{s})^{-1}$}\},
\end{align*}
so that
\begin{align*}
\ell(d(\mathfrak{s})s_{n-2}s_{n-3}\cdots
s_{k})=\ell(d(\mathfrak{s}))-n+k+1.
\end{align*}
If we write $v=d(\mathfrak{s})s_{n-2}s_{n-3}\cdots s_{k}$ and
$u=s_{k}s_{k+1}\cdots s_{n-2}w_p$, then
\begin{align}
\begin{split}\label{tango}
&E_{2f-1}T_{w_p}^{-1}T_{d(\mathfrak{s})}^*m_\mu=
E_{2f-1}T_{w_p}^{-1}E_1E_3\cdots
E_{2f-3}T_{d(\mathfrak{s})}^*x_\mu\\
&\quad=E_1E_3\cdots E_{2f-1}T_{w_p}^{-1}T_{d(\mathfrak{s})}^*x_\mu=E_1E_3\cdots E_{2f-1}T_{u}^{-1}T_{v}^*x_\mu.
\end{split}
\end{align}
Since $v$ has a reduced expression $v=s_{i_1}s_{i_2}\cdots
s_{i_l}$ in the subgroup $\langle s_i:2f-1\le i<n-2\rangle$,
we define $v'=s_{i_1+2}s_{i_2+2}\cdots s_{i_l+2}$ and, using the
braid relation
$T_i^{-1}T_{i+1}^{-1}T_i=T_{i+1}T_{i}^{-1}T_{i+1}^{-1}$, obtain
\begin{align}\label{braidc:2}
T_u^{-1}T_i=
\begin{cases}
T_{i+2}T_u^{-1}&\text{if $2f-1\le i<k$;} \\
T_{i+1}T_u^{-1}&\text{if $k< i<n$,}
\end{cases}
\end{align}
which allows us to rewrite~\eqref{tango} as
\begin{align}\label{tango:2}
E_{2f-1}T_{w_p}^{-1}T_{d(\mathfrak{s})}^*m_\mu=E_1E_3\cdots
E_{2f-1}T_{v'}^*T_{u}^{-1}x_\mu.
\end{align}

Now, to each row $i$ of $\mathfrak{t}^{\mu}\in\STD_{n-1}(\mu)$, associate the subgroup
\begin{align*}
\mathfrak{R}_{\mathfrak{t}^\mu,i}=\langle
s_{i'}:\text{$i',i'+1$ appear in row $i$ of
$\mathfrak{t}^\mu$} \rangle
\end{align*}
and define $\mathfrak{R}_{\mathfrak{t}^\tau,i}$ analogously for $\mathfrak{t}^\tau\in\STD_{n}(\tau)$. Let us suppose that $n-1$ appears as an entry in row $j$ of $\mathfrak{s}$; if $i\ne j$, then by~\eqref{braidc:2}
\begin{align}\label{gruffle}
\sum_{w\in\mathfrak{R}_{\mathfrak{t}^\mu,i}}q^{\ell(w)}T_u^{-1}T_w=
\sum_{w\in\mathfrak{R}_{\mathfrak{t}^\tau,i}}q^{\ell(w)}T_wT_u^{-1}.
\end{align}
On the other hand, within $\mathfrak{R}_{\mathfrak{t}^\mu,j}$
take the parabolic subgroup \begin{align*}
\mathfrak{P}_{\mathfrak{t}^\mu,j}=\langle
w\in\mathfrak{R}_{\mathfrak{t}^\mu,j}:(k)w=k\rangle
\end{align*}
and, noting that the set of distinguished right coset
representatives for $\mathfrak{P}_{\mathfrak{t}^\mu,j}$ in
$\mathfrak{R}_{\mathfrak{t}^\mu,j}$ (Proposition~3.3 of~\cite{mathas:ih}) is
\begin{align*}
\mathscr{D}=\{v_i:\text{$v_0=1$ and $v_i=v_{i-1}s_{k-i}$ for
$0<i\le\tau_j$}\},
\end{align*}
we write
\begin{align*}
\sum_{w\in\mathfrak{R}_{\mathfrak{t}^\mu,j}}q^{\ell(w)}T_{u}^{-1}T_w&=
\sum_{w\in\mathfrak{P}_{\mathfrak{t}^\mu,j}}q^{\ell(w)}T_{u}^{-1}T_w
\sum_{v\in\mathscr{D}}  q^{\ell(v)}T_v.
\end{align*}
Using the last expression and~\eqref{braidc:2}, we obtain
\begin{align*}
T_{u}^{-1}\sum_{w\in\mathfrak{P}_{\mathfrak{t}^\mu,j}}
q^{\ell(w)}T_w=\sum_{w\in\mathfrak{R}_{\mathfrak{t}^\tau,j}}q^{\ell(w)}T_wT_{u}^{-1},
\end{align*}
which, together with~\eqref{gruffle}, implies that
\begin{align*}
\begin{split}
T_{v'}^*T_{u}^{-1} x_{\mu}&=T_{v'}^*
\sum_{i\ge1}\, \sum_{w\in
\mathfrak{R}_{\mathfrak{t}^\tau,i}}q^{\ell(w)}T_wT_{u}^{-1}\,
\sum_{v\in\mathscr{D}} q^{\ell(v)}T_{v}\\
&=T_{v'}^*x_{\tau}T_{u}^{-1}\sum_{v\in\mathscr{D}}
q^{\ell(v)}T_v.
\end{split}
\end{align*}
Since $v'\in\langle s_i : 2f< i<n \rangle$, multiplying both sides of the last expression by $E_1E_3\cdots E_{2f-1}$ on the left and referring to~\eqref{tango:2}, we obtain
\begin{align*}
E_{2f-1}T_{w_p}^{-1}T_{d(\mathfrak{s})}^*m_\mu= T_{v'}^*E_1E_3\cdots E_{2f-1}x_{\tau}T_{u}^{-1}\sum_{v\in\mathscr{D}}
q^{\ell(v)}T_v.
\end{align*}
As the term on the right hand side of the above expression lies in $\check{B}_n^\lambda$, the result now follows.
\end{proof}

The next example illustrates Lemma~\ref{big:0}.
\begin{example}\label{btenex:2}
In parts (a) and (b) below, let $n=10$, $f=2$ and $\lambda=(3,2,1)$. Since $\lambda$ has three removable nodes and four addable nodes, the partitions
$\mu^{(i)}$ with $\mu^{(i)}\to\lambda$ and $|\mu^{(i)}|>|\lambda|$
are
\begin{align*}
\mu^{(4)}=(4,2,1)\rhd\mu^{(5)}=(3,3,1)\rhd\mu^{(6)}=(3,2,2)\rhd\mu^{(7)}=(3,2,1,1).
\end{align*}
(a) Taking $p=7$, we have
$w_p=s_8s_7s_6s_5s_4s_3s_9s_8s_7s_6s_5s_4$,
\begin{align*}
\mathfrak{t}^{\lambda}=\text{\tiny\Yvcentermath1$\young(567,89,\dten)$}\,
&&\text{and}&&
\mathfrak{t}^{\mu^{(p)}}=\text{\tiny\Yvcentermath1$\young(345,67,8,9)$}\,,
\end{align*}
so that $x_{\mu^{(p)}}=(1+qT_3)(1+qT_4+q^2T_4T_3)(1+qT_6)$. Using
the braid relation
$T_j^{-1}T_{j+1}^{-1}T_j=T_{j+1}T_j^{-1}T_{j+1}^{-1}$, it is
verified that
\begin{align*}
E_{3}T_{w_p}^{-1}m_{\mu^{(p)}}= m_\lambda T_{w_p}^{-1}.
\end{align*}

(\emph{b}) Let $\mu=(4,3)$ and
$\mathfrak{s}=\text{\tiny\Yvcentermath1$\young(3459,678)$}$\, so
$d(\mathfrak{s})=s_6s_7s_8$. Then
\begin{align*}
\hat{\mathfrak{s}}=\text{\tiny\Yvcentermath1$\young(1237,456)$}&&
\text{and}&&\mu^{(p)}(\hat{\mathfrak{s}})=\text{\tiny\Yvcentermath1$\young(1114,223)$}\,,
\end{align*}
as shown in Example~\ref{sstabex:2}. Now,
\begin{align*}
6=\min\{i\,|\,\text{$2f-1\le i\le n-2$ and
$(n-1)d(\mathfrak{s})^{-1}\le(i)d(\mathfrak{s})^{-1}$}\},
\end{align*}
hence, writing
$u=s_5s_4s_3s_9s_8s_7s_6s_5s_4$, one obtains
\begin{align*}
E_3T_{w_p}^{-1}T_{d(\mathfrak{s})}^*m_\mu=E_3T^{-1}_{u}m_\mu=E_3T^{-1}_{u}E_1x_\mu
\end{align*}
where
\begin{multline*}
x_\mu=(1+qT_3)(1+qT_4+q^2T_4T_3)(1+qT_5+q^2T_5T_4+q^3T_5T_4T_3)\\
\times(1+qT_7)(1+qT_8+q^2T_8T_7).
\end{multline*}
Using the braid relation,
\begin{align*}
T_{u}^{-1}x_\mu=x_\tau T^{-1}_{u}(1+qT_5+q^2T_5T_4+q^3T_5T_4T_3),
\end{align*}
where $\mathfrak{t}^{\tau}=
\text{\tiny\Yvcentermath1$\young(567,89\dten)$}\,\,$ and
\begin{align*}
x_\tau=(1+qT_5)(1+qT_6+q^2T_6T_5)(1+qT_8)(1+qT_9+q^2T_9T_8).
\end{align*}
As $\tau\rhd\lambda$, it follows that
\begin{align*}
&E_3T_{w_p}^{-1}T_{d(\mathfrak{s})}^*m_\mu=E_1E_3x_\tau
T_{u}^{-1}(1+qT_5+q^2T_5T_4+q^3T_5T_4T_3)\\
&\quad=m_\tau
T_{u}^{-1}(1+qT_5+q^2T_5T_4+q^3T_5T_4T_3)\equiv0\mod{\check{B}^\lambda_n}.
\end{align*}
\end{example}
\begin{corollary}\label{quant:0}
Let $f$ be an integer $0<f\le[n/2]$ and $\lambda$ be a partition of $n-2f$ with $(p-t)$ addable nodes. Suppose that $\mu^{(1)}\unrhd\mu^{(2)}\unrhd\cdots\unrhd\mu^{(p)}$ is the ordering of $\{\mu :\mu\to \lambda\}$ by the dominance order on partitions.
If $\mu$ is a partition of $n-2f+1$ such that $\mu\rhd\mu^{(t+1)}$, and $\mathsf{S}\in\mathcal{T}_0(\mu,\mu^{(p)})$, then 
\begin{align*}
E_{2f-1}T_{w_p}^{-1}m_{\mathsf{S}\mathfrak{t}}\equiv 0 \mod{B_n^\lambda},&&\text{for all $\mathfrak{t}\in\STD_{n-1}(\mu)$.}
\end{align*}
\end{corollary}
\begin{proof}
There are $p-t$ standard tableaux $\mathfrak{s}$ labelled by the integers $\{2f-1,2f,\dots,n-1\}$ which satisfy the conditions (i) $\SHAPE(\mathfrak{s}|_{n-2})=\lambda$, and (ii) $\mu^{(p)}(\mathfrak{s})\in\mathcal{T}_0(\nu,\mu^{(p)})$, for some partition $\nu$ of $n-2f+1$; each such tableau $\mathfrak{s}$ additionally satisfies the condition that $\SHAPE(\mathfrak{s})=\mu^{(i)}$ for some $i$ with $t<i\le p$ (the precise form that any such $d(\mathfrak{s})$ must take is given in~\eqref{skdef} below). Thus if $\mu$ is as given in the statement of the corollary and $\mathfrak{s}\in\STD_{n-1}(\mu)$ satisfies $\mu^{(p)}(\hat{\mathfrak{s}})\in\mathcal{T}_0(\mu,\mu^{(p)})$, then $\tau=\SHAPE(\mathfrak{s}|_{n-2})\rhd\lambda$, so by Lemma~\ref{big:0}, 
\begin{align*}
E_{2f-1}T_{w_p}^{-1}T_{d(\mathfrak{s})}^* m_\mu\equiv 0&&\mod B^\lambda_{n}.
\end{align*}
Using the definition of $m_{\mathsf{S}\mathfrak{t}}$, the result now follows.
\end{proof}

\begin{lemma}\label{big:2}
Let $f$ be an integer, $0< f\le[n/2]$, and $\lambda\vdash n-2f$, $\mu\vdash n-2f+1$ be partitions
such that $\mu\to\lambda$. If $\mu^{(p)}$ is minimal with respect to dominance order among
$\{\nu:\text{$\nu\to\lambda$ and $|\nu|>|\lambda|$}\}$, and
$\mathfrak{s}\in\STD_{n-1}(\mu)$ is a tableau such that
$\mu^{(p)}(\hat{\mathfrak{s}})\in\mathcal{T}_0(\mu,\mu^{(p)})$, then there exist
$a_{(\mathfrak{t},w)}\in R$, for
$(\mathfrak{t},w)\in\mathcal{I}_{n}(\lambda)$, such that
\begin{align*}
E_{2f-1}T_{w_p}^{-1}T_{d(\mathfrak{s})}^*
m_{\mu}\equiv\sum_{(\mathfrak{t},w)\in\mathcal{I}_{n}(\lambda)}a_{(\mathfrak{t},w)}m_\lambda
T_{d(\mathfrak{t})}T_w\mod{\check{B}_n^\lambda}.
\end{align*}
\end{lemma}
\begin{proof}
There is a unique tableau $\mathfrak{s}\in\STD_{n-1}(\mu)$ satisfying the hypotheses of the lemma, namely the tableau with $\mathfrak{s}|_{n-1}=\mathfrak{t}^\lambda\in\STD_{n-2}(\lambda)$. Furthermore,
\begin{align*}
d(\mathfrak{s})=s_ks_{k+1}\cdots
s_{n-2}&&\text{where}&&k=(n-1)d(\mathfrak{s})^{-1}.
\end{align*}
Suppose that $k$ appears as an entry in the row $j$ of
$\mathfrak{s}$. As in the proof of Lemma~\ref{big:0}, we
associate to row $j$ of $\mathfrak{t}^\mu$ the subgroup
\begin{align*}
\mathfrak{R}_{\mathfrak{t}^\mu,j}=\langle s_{i}:\text{$i,i+1$
appear in row $j$ of $\mathfrak{t}^\mu$} \rangle
\end{align*}
and take the parabolic subgroup
$\mathfrak{P}_{\mathfrak{t}^\mu,j}=\langle
w\in\mathfrak{R}_{\mathfrak{t}^\mu,j}:(k)w=k\rangle\subseteq\mathfrak{R}_{\mathfrak{t}^\mu,j}$.
The set of distinguished right coset representatives for
$\mathfrak{P}_{\mathfrak{t}^\mu,j}$ in
$\mathfrak{R}_{\mathfrak{t}^\mu,j}$ is
\begin{align*}
\mathscr{D}=\{v_i:\text{$v_0=1$ and $v_i=v_{i-1}s_{k-i}$ for
$0<i\le\lambda_j$}\}.
\end{align*}
As in the proof of Lemma~\ref{big:0}, the coset representatives
$\mathscr{D}$ enable us to write
\begin{align}
E_{2f-1}T_{w_p}^{-1}
T_{d(\mathfrak{s})}^*m_{\mu}=m_{\lambda}T_{u}^{-1}\sum_{v\in\mathscr{D}}
q^{\ell(v)}T_v,\label{gob:2}
\end{align}
where $u=s_{k}s_{k+1}\cdots s_{n-2}w_p=s_{k-1}s_{k-2}\cdots
s_{2f-1}s_{n-1}s_{n-2}\cdots s_{2f}$. 
\end{proof}
Let $f$ be an integer, $0\le f\le [n/2]$, and $\lambda$ be a
partition of $n-2f$ with $t$ removable and $p-t$ addable nodes. Take
$\mu^{(t+1)}\rhd\mu^{(t+2)}\rhd\cdots\rhd\mu^{(p)}$ as the ordering
of the set $\{\mu:\text{$\mu\to\lambda$ and
$|\mu|>|\lambda|$}\}$ by dominance order on partitions and, for
$t<k\le p$, suppose that $[\lambda]$ is the diagram obtained by
deleting a node from the row $j_k$ of $[\mu^{(k)}]$. There exists for each $\mu^{(k)}$ with
$\mu^{(k)}\to\lambda$ and $|\mu^{(k)}|>|\lambda|$, a unique tableau
$\mathfrak{s}_k\in\STD_{n-1}(\mu^{(k)})$ such that
$\mu^{(p)}(\mathfrak{s}_k)\in\mathcal{T}_0(\mu^{(k)},\mu^{(p)})$ and
$\SHAPE(\mathfrak{s}_k|_{n-2})=\lambda$. To wit, $\mathfrak{s}_k$ is
determined by
\begin{align}\label{skdef}
d(\mathfrak{s}_k)=s_{a_k}s_{a_k+1}\cdots
s_{n-2}&&\text{where} &&{a_k=2(f-1)+\sum_{i=1}^{j_k}\mu_i^{(k)}.}
\end{align}
Thus we let
\begin{align}\label{wkdef}
w_k=d(\mathfrak{s}_k)^{-1}w_p=s_{a_k-1}s_{a_k-2}\cdots
s_{2f-1}s_{n-1}s_{n-2}\cdots s_{2f},
\end{align}
and write
\begin{align}\label{ydef:2}
y_{\mu^{(k)}}^\lambda=E_{2f-1}T_{w_k}^{-1}m_{\mu^{(k)}}+\check{B}^\lambda_n.
\end{align}
By Lemma~\ref{big:2}, we note that $y_{\mu^{(k)}}^\lambda$ is a
well defined element in the $B_n(q,r)$--module
$S^\lambda$. We define $N^{\mu^{(k)}}$, for $t<k\le p$, to be
the $B_{n-1}(q,r)$--submodule of $S^\lambda$ generated by
$y_{\mu^{(k)}}^\lambda$.
\begin{example}\label{smallex:0}
Let $n=4$, $f=1$. If $\lambda=(1,1)$, and $\mu=(2,1)$, then $\mathfrak{s}=\text{\tiny\Yvcentermath1$\young(13,2)$}$ is the unique tableau with $\mathfrak{s}|_{n-1}=\mathfrak{t}^\lambda\in\STD_{n-2}(\lambda)$. Thus $y_\mu^\lambda=E_1T_{w_p}^{-1}T_{d(\mathfrak{s})}^*m_\mu+\check{B}_4^\lambda=E_1T_{2}^{-1}T_1^{-1}T_3^{-1}(1+qT_1)+\check{B}_4^\lambda$.
\end{example}

Recall that $N^{\mu^{(t)}}\subseteq N^{\mu^{(p)}}$ is a strict inclusion of $B_{n-1}(q,r)$--modules.
\begin{lemma}\label{big:1}
Let $f$ be an integer, $0< f\le[n/2]$, and $\lambda$ be a
partition of $n-2f$ with $t$ removable nodes and $(p-t)$ addable
nodes. Suppose that $
\mu^{(t+1)}\rhd\mu^{(t+2)}\rhd\cdots\rhd\mu^{(p)}$ is the ordering
of $\{\mu:\text{$\mu\to\lambda$ and $|\mu|>|\lambda|$}\}$ by
dominance order on partitions. Then the right
$B_{n-1}(q,r)$--module $N^{\mu^{(p)}}/N^{\mu^{(t)}}$ is generated
as an $R$--module by
\begin{align*}
\left\{y_{\mu^{(k)}}^\lambda T_{d(\mathfrak{t})}T_w +N^{\mu^{(t)}}
:\text{$(\mathfrak{t},w)\in\mathcal{I}_{n-1}(\mu^{(k)})$
and $t< k\le p$} \right\}.
\end{align*}
\end{lemma}
\begin{proof}
From the expression~\eqref{ydef:2}, observe that the $B_{n-1}(q,r)$--module $N^{\mu^{(p)}}$ is generated as an $R$--module by elements of the form
\begin{align*}
y_{\mu^{(p)}}^\lambda b=m_\lambda
T_{w_p}^{-1}b+\check{B}^\lambda_n=E_{2f-1}T_{w_p}^{-1}m_{\mu^{(p)}}b+\check{B}_{n}^\lambda,&&\text{for $b\in B_{n-1}(q,r)$.}
\end{align*}

Let $b\in B_{n-1}(q,r)$. Then, by Lemma~\ref{permod:2}, there exist $\mathsf{S}\in\mathcal{T}_0(\mu,\mu^{(p)})$, for $\mu\unrhd\mu^{(p)}$ and $|\mu|=|\mu^{(p)}|$, together and $a_{\mathsf{S},\mathfrak{t},w}$, for $(\mathfrak{t},w)\in\mathcal{I}_{n-1}(\mu)$, such that 
\begin{align}\label{pest:0}
m_{\mu^{(p)}}b=
\sum_{\substack{\mu\unrhd\mu^{(p)}\\ (\mathfrak{t},w)\in\mathcal{I}_{n-1}(\mu) \\ \mathsf{S}\in\mathcal{T}_0(\mu,\mu^{(p)}) }} \,a_{\mathsf{S},\mathfrak{t},w}\,m_{\mathsf{S}\mathfrak{t}}T_w+b',
\end{align}
where $b'\in B^{f}_{n-1}$. Since the process of rewriting a product 
\begin{align*}
E_1E_3\cdots E_{2f-3}T_uT_vb, &&\text{for $u\in\langle s_i: 2f-2< i <n-1\rangle$, $v\in\mathscr{D}_{f-1,n-1}$,}
\end{align*}
in terms of the basis~\eqref{bcell:1} depends only on~\eqref{saru:prop3.8}, Proposition~3.7 of~\cite{saru} and operations in the subalgebra $\langle T_i: 2f-2< i <n-1\rangle\subseteq B_{n-1}(q,r)$, we note that the term $b'$ in~\eqref{pest:0} satisfies
\begin{align*}
b'\in (E_1E_{3}\cdots E_{2f-3})B_{n-1}(q,r)\cap B_{n-1}^{f}.
\end{align*}
By decomposing the set $\{\mu:\text{$|\mu|=n-2f+1$ and $\mu\unrhd\mu^{(p)}$}\}$ and using Lemma~\ref{big:0}, we obtain, for each $w\in\mathscr{D}_{f-1,n-1}$, an expression:
\begin{align}\label{pest:1}
\sum_{\substack{\mu\unrhd\mu^{(p)}\\ \mathfrak{t}\in\STD_{n-1}(\mu)\\ \mathsf{S}\in\mathcal{T}_0(\mu,\mu^{(p)})}}\, 
a_{\mathsf{S},\mathfrak{t},w}\, m_{\mathsf{S}\mathfrak{t}} T_w= 
\sum_{\substack{t<k\le p\\ \mathfrak{t}\in\STD_{n-1}(\mu^{(k)})\\ \mathsf{S}\in\mathcal{T}_0(\mu^{(k)},\mu^{(p)})}} a_{\mathsf{S},\mathfrak{t},w}\,m_{\mathsf{S}\mathfrak{t}}T_w
+\sum_{\substack{\mu\rhd\mu^{(t+1)}\\ \mathfrak{t}\in\STD_{n-1}(\mu)\\ \mathsf{S}\in\mathcal{T}_0(\mu,\mu^{(p)})}} 
a_{\mathsf{S},\mathfrak{t},w}\,m_{\mathsf{S}\mathfrak{t}}T_w.
\end{align}
Hence, multiplying both sides of~\eqref{pest:0} by $E_{2f-1}T_{w_p}^{-1}$ on the left, and using~\eqref{pest:1} together with Corollary~\ref{quant:0}, we obtain: 
\begin{align*}
E_{2f-1}T_{w_p}^{-1}m_{\mu^{(p)}}b+\check{B}_n^\lambda
=E_{2f-1}T_{w_p}^{-1}
\sum_{\substack{t<k\le p\\ (\mathfrak{t},w)\in\mathcal{I}_{n-1}(\mu^{(k)})\\ \mathsf{S}\in\mathcal{T}_0(\mu^{(k)},\mu^{(p)})}} a_{\mathsf{S},\mathfrak{t},w}\,m_{\mathsf{S}\mathfrak{t}}T_w
+ E_{2f-1}T_{w_p}^{-1}b'+\check{B}_n^\lambda.
\end{align*}
We recall the definition of the tableaux $\mathfrak{s}_k\in\STD_{n-1}(\mu^{(k)})$, for $t<k\le p$, in~\eqref{skdef}, and also that the $w_k$ defined, for $t<k\le p$, by~\eqref{wkdef}, are chosen so that $T_{w_p}^{-1}T_{d(\mathfrak{s}_k)}^*=T_{w_k}^{-1}$. Thus 
\begin{align*}
E_{2f-1}T_{w_p}^{-1}m_{\mu^{(p)}}b+\check{B}_n^\lambda
=
\sum_{\substack{t<k\le p\\ (\mathfrak{t},w)\in\mathcal{I}_{n-1}(\mu^{(k)})}} a_{k,\mathfrak{t},w}\,E_{2f-1}T_{w_k}^{-1}m_{\mu^{(k)}} T_{d(\mathfrak{t})}T_w
+ E_{2f-1}T_{w_p}^{-1}b'+\check{B}_n^\lambda,
\end{align*}
where $a_{k,\mathfrak{t},w}=q^{\ell(d(\mathfrak{s}_k))}a_{\mathsf{S},\mathfrak{t},w}$ whenever $\mu^{(p)}(\hat{\mathfrak{s}}_k)=\mathsf{S}$. Thus we have shown that
\begin{align}\label{filter:a}
E_{2f-1}T_{w_p}^{-1}m_{\mu^{(p)}}b+\check{B}_n^\lambda
=
\sum_{\substack{t<k\le p\\ (\mathfrak{t},w)\in\mathcal{I}_{n-1}(\mu^{(k)})}} a_{k,\mathfrak{t},w}\, y_{\mu^{(k)}}^\lambda T_{d(\mathfrak{t})}T_w
+ E_{2f-1}T_{w_p}^{-1}b'+\check{B}_n^\lambda.
\end{align}
It now remains to show that $E_{2f-1}T_{w_p}^{-1}b'+\check{B}_n^\lambda\in N^{\mu^{(t)}}$. Noting the characterisation of the $B_{n-1}(q,r)$--module $N^{\mu^{(t)}}$ given in Corollary~\ref{resbcor1}, to complete the proof of the lemma, it suffices to demonstrate the statement following.
\begin{claim}\label{claim:0}
If $b\in  (E_1E_3\cdots E_{2f-3})B_{n-1}(q,r)\cap B_{n-1}^f$ then there exist $a_{\mathfrak{s},\mathfrak{t},w}\in R$, for $\mathfrak{s},\mathfrak{t}\in\STD_n(\nu)$, $w\in\mathscr{D}_{f,n-1}$ and $\nu\vdash n-2f$, such that 
\begin{align}\label{pest:2}
E_{2f-1}T_{w_p}^{-1}b\equiv\sum_{\substack{\nu\vdash n-2f\\ \mathfrak{s},\mathfrak{t}\in\STD_{n}(\nu)\\ w\in\mathscr{D}_{f,n-1}}}a_{\mathfrak{s},\mathfrak{t},w}T_{d(\mathfrak{s})}^*m_\nu T_{d(\mathfrak{t})}T_w &&\mod{B_{n}^{f+1}}.
\end{align}
\end{claim}
We now prove the claim. Let $b\in(E_1E_3\cdots E_{2f-3})B_{n-1}(q,r)\cap B_{n-1}^f$.
As in the proof of Lemma~\ref{ecor}, we may write $b$, modulo $B_{n-1}^{f+1}\subset {B}_{n}^{f+1}$, as an $R$-linear combination of elements of the form 
\begin{align*}
\left\{T_v^* E_{1}E_{3}\cdots E_{2f-1}T_uT_{w}\bigg|
\begin{matrix}
\text{$v,w\in\mathscr{D}_{f,n-1}$, $u\in\langle s_i:2f<i<n-1\rangle$}\\
\text{and $v\in \langle s_i : 2f-2<i<n-1\rangle$}
\end{matrix}
\right\}.
\end{align*}
Multiplying an element of the above set on
the left by $E_{2f-1}T_{w_p}^{-1}$, we obtain:
\begin{align}\label{obt}
E_{1}E_{3}\cdots E_{2f-3}E_{2f-1}T_{w_p}^{-1}
T_v^*E_{2f-1}T_uT_{w}.
\end{align}
There are two cases following. In the first case, suppose
that $v$ has a reduced expression $v=s_{i_1}s_{i_2}\cdots s_{i_l}$
in $\langle s_i : 2f-2<i<n-2\rangle$. Applying the relations
\begin{align*}
T_{i}^{-1}T_{i+1}^{-1}T_i&=T_{i+1}T_{i}^{-1}T_{i+1}^{-1}&&\text{and}
&&T_{i}^{-1}T_{i+1}^{-1}E_i=E_{i+1}T_{i}^{-1}T_{i+1}^{-1}, 
\end{align*}
we obtain $T_{w_p}^{-1}T_{v}^*E_{2f-1}=T^*_{v''}E_{2f+1}T_{w_p}^{-1}$,
where $v''=s_{i_1+2}s_{i_2+2}\cdots s_{i_l+2}$. As $T_{v''}^*$
commutes with $E_{1}E_3\cdots E_{2f-1}$, substitution
into~\eqref{obt} yields:
\begin{align*}
E_{1}E_{3}\cdots E_{2f-1}T_{w_p}^{-1} T_v^*E_{2f-1}T_uT_{w}=
T_{v''}^*E_{1}E_{3}\cdots E_{2f+1}T_{w_p}^{-1}T_uT_{w}
\end{align*}
which is visibly a term in
$B_n^{f+1}$. 

In the second case, suppose that $v$ does not have a reduced expression in $\langle s_i: 2f-2<i<n-2\rangle$. To obtain an explicit expression for such $v$, we first enumerate the elements of 
\begin{align}\label{blob:0}
\mathscr{D}_{f,n-1}\cap\langle s_{i}: 2f-2<i<n-1\rangle.
\end{align}
As in Example~\ref{cellex:0}, the elements of the set~\eqref{blob:0} take the form
\begin{align*}
v_{i,j}=s_{2f}s_{2f+1}\cdots s_{j-1}s_{2f-1}s_{2f}\cdots s_{i-1},&&\text{for $2f-2< i <j<n$.}
\end{align*}
Now, $v_{i,j}$ does not have a reduced expression in $\langle s_i: 2f-2<i<n-2\rangle$ if and only if $v_{i,j}$ does not stabilise $n-1$; thus $v_{i,j}=v_{i,n-1}$, for some $2f-2< i< n-1$. Define 
\begin{align*}
v_i=v_{i,n-1}=s_{2f}s_{2f+1}\cdots s_{n-2}s_{2f-1}s_{2f}\cdots s_{i-1},&&\text{for $2f-2< i< n-1$,}
\end{align*}
so the elements of the set~\eqref{blob:0} which do not stabilise $n-1$ are precisely 
\begin{align*}
\{v_i:2f-1\le i\le n-2\}.
\end{align*}
Let $j$ be an integer, $2f-1\le j\le n-2$, and calculate $E_{2f-1}T_{w_p}^{-1}T_{v_j}^*E_{2f-1}$ explicitly, beginning with:
\begin{multline*}
E_{2f-1}T_{w_p}^{-1}T_{v_j}^*E_{2f-1}=E_{2f-1}T_{w_p}^{-1}(T_{j-1}T_{j-2}\cdots T_{2f-1}){(T_{n-2}T_{n-3}\cdots T_{2f})E_{2f-1}}\\
=E_{2f-1}T_{w_p}^{-1}(T_{n-2}T_{n-3}\cdots T_{j+1})(T_{j-1}T_{j-2}\cdots T_{2f-1}){(T_{j}T_{j-1}\cdots T_{2f})E_{2f-1}}\\
=E_{2f-1}(T_{2f}^{-1}T_{2f+1}^{-1}\cdots T_{n-1}^{-1})(T_{2f-1}^{-1}T_{2f}^{-1}\cdots T_{j}^{-1})(T_{j-1}T_{j-2}\cdots T_{2f-1}){(T_{j}T_{j-1}\cdots T_{2f})E_{2f-1}}\\
=E_{2f-1}(T_{2f}^{-1}T_{2f+1}^{-1}\cdots T_{j+1}^{-1})
(T_{2f-1}^{-1}T_{2f}^{-1}\cdots T_{j}^{-1})
(T_{j-1}T_{j-2}\cdots T_{2f-1})\\
\times(T_{j}T_{j-1}\cdots T_{2f})E_{2f-1}
(T_{j+2}^{-1}T_{j+3}^{-1}\cdots T_{n-1}^{-1}).
\end{multline*}
Using the relations 
\begin{align*}
E_{2f-1}(T_{2f}^{-1}T_{2f+1}^{-1}\cdots T_{j+1}^{-1})(T_{2f-1}^{-1}T_{2f}^{-1}\cdots T_{j}^{-1})&=E_{2f-1}E_{2f}\cdots E_{j+1}\intertext{and}
(T_{j-1}T_{j-2}\cdots T_{2f-1})(T_jT_{j-1}\cdots T_{2f})E_{2f-1}&=E_jE_{j-1}\cdots E_{2f-1},
\end{align*}
we now obtain:
\begin{align*}
E_{2f-1}T_{w_p}^{-1}T_{v_j}^*E_{2f-1}=(E_{2f-1}E_{2f}\cdots
E_jE_{j+1})(E_jE_{j-1}\cdots E_{2f-1})(T_{j+2}^{-1}T_{j+3}^{-1}\cdots T_{n-1}^{-1}).
\end{align*}
Further applying relations like
$E_i(E_{i+1}E_{i+2}E_{i+1})E_i=E_iE_{i+1}E_i=E_i$ in the right hand side of the above expression gives:
\begin{align}\label{obt:2}
E_{2f-1}T_{w_p}^{-1}T_{v_j}^*E_{2f-1}=E_{2f-1}(T_{j+2}^{-1}T_{j+3}^{-1}\cdots
T_{n-1}^{-1}).
\end{align}
Multiplying both sides of~\eqref{obt:2} by $E_1E_3\cdots E_{2f-3}$ on the left and by $T_uT_{w}$ on the right, the term~\eqref{obt}, with $v_j$ substituted for $v$, becomes
\begin{align*}
E_1E_3\cdots E_{2f-1}T_{w_p}^{-1}T_{v_j}^*E_{2f-1}T_uT_{w}
=E_1E_3\cdots E_{2f-1}(T_{j+2}^{-1}T_{j+3}^{-1}\cdots T_{n-1}^{-1})T_uT_{w}.
\end{align*}
Now $(T_{j+2}^{-1}T_{j+3}^{-1}\cdots T_{n-1}^{-1})T_u$ lies in
$\langle T_{2f+1},T_{2f+2},\dots,T_{n-1}\rangle\subseteq B_n(q,r)$
and consequently, using Theorem~\ref{saruthm}, can be expressed as
an $R$--linear sum of elements from the set $\{T_{u'}:u'\in
\langle s_i:2f<i<n\rangle\}$ together with an element $b'$ from
the two--sided ideal of $\langle
T_{2f+1},T_{2f+2},\dots,T_{n-1}\rangle$ generated by $E_{2f+1}$.
By Lemma~\ref{gomi}, the element labelled $b'$
immediately preceding satisfies
\begin{align*}
E_1E_3\cdots E_{2f-1}E_{2f-1}b'T_{w}\in
B_n^{f+1},
\end{align*}
and can be safely ignored in any calculation modulo $\check{B}_n^\lambda$. If $w\in\mathscr{D}_{f,n-1}$, then straightening a term
\begin{align}\label{hooroo}
E_1E_3\cdots E_{2f-1}T_{u'}T_{w}, &&\text{for $u'\in \langle s_i:2f<i<n\rangle$},
\end{align}
into linear combinations of the basis elements given in Theorem~\ref{saruthm}, is achieved using relations in $\mathscr{H}_{n-2f}(q^2)$, via the map $\vartheta_f$, and does not involve any transformation of $T_w$; it follows that there exist 
$a_{\mathfrak{u},\mathfrak{v},w}$, for $\mathfrak{u},\mathfrak{v}\in\STD_n(\nu)$ and $\nu\vdash n-2f$, such that the term~\eqref{hooroo} can be expressed as
\begin{align*}
E_1E_3\cdots E_{2f-1}T_{u'}T_w\equiv\sum_{\substack{\nu\vdash n-2f\\\mathfrak{u},\mathfrak{v}\in\STD_{n}(\nu)}}a_{\mathfrak{u},\mathfrak{v},w}T_{d(\mathfrak{u})}^*m_\nu T_{d(\mathfrak{v})}T_w&&\mod{B}_n^{f+1}.
\end{align*}
This completes the proof of the claim.
\end{proof}
We continue to use the notation established in the statement of
Lemma~\ref{big:1}.

If $t<k\le p$, then by Lemma~\ref{big:1}, there is a proper
inclusion of $B_{n-1}(q,r)$--modules $N^{\mu^{(t)}}\subseteq
N^{\mu^{(k)}}$.
\begin{corollary}\label{big:5}
Let $f$ be an integer, $0< f\le[n/2]$, and $\lambda$ be a
partition of $n-2f$ with $t$ removable nodes and $(p-t)$ addable
nodes. Suppose that
$\mu^{(1)}\rhd\mu^{(2)}\rhd\cdots\rhd\mu^{(p)}$ is the ordering of
$\{\mu:\mu\to\lambda\}$ by dominance order on partitions. Then
\begin{align*}
(0)=N^{\mu^{(0)}}\subseteq N^{\mu^{(1)}}\subseteq\cdots\subseteq
N^{\mu^{(p)}}=\RES(S^\lambda)
\end{align*}
is a filtration of $\RES(S^\lambda)$ by $B_{n-1}(q,r)$--modules,
wherein each quotient $N^{\mu^{(k)}}/N^{\mu^{(k-1)}}$, for $1\le k \le p$, is isomorphic to the cell
module $S^{\mu^{(k)}}$ via
\begin{align}\label{filter:map}
y^\lambda_{\mu^{(k)}}T_{d(\mathfrak{t})}T_w+N^{\mu^{(k-1)}}\mapsto
m_{\mu^{(k)}}T_{d(\mathfrak{t})}T_w+\check{B}_{n-1}^{\mu^{(k)}},
\end{align}
for $(\mathfrak{t},w)\in\STD_{n-1}(\mu^{(k)})$.
\end{corollary}
\begin{proof}
It has been shown in Lemma~\ref{bres1} that the map~\eqref{filter:map} is an isomorphism $N^{\mu^{(k)}}/N^{\mu^{(k-1)}}\cong S^{\mu^{(k)}}$, for $1\le k \le t$.

For each $k$ with $t<k\le p$, let $\mathsf{S}_k=\mu^{(p)}(\mathfrak{s}_k)$, where $\mathfrak{s}_k$ is the tableau defined by~\eqref{skdef}. 
If $\mathfrak{v}\in\STD_{n-1}(\mu^{(k)})$ and $b\in B_{n-1}(q,r)$, then using Lemmas~\ref{permod:2} and~\ref{big:0}, there exist $a_{j,\mathfrak{t},w}\in R$, for $(\mathfrak{t},w)\in\mathcal{I}_{n-1}(\mu^{(j)})$, and $t<j\le k$, such that  
\begin{align}\label{junk}
m_{\mathsf{S}_k \mathfrak{v}}b
=\sum_{\substack{t<j\le k\\ (\mathfrak{t},w)\in\mathcal{I}_{n-1}(\mu^{(j)})}} a_{j,\mathfrak{t},w}\, m_{\mathsf{S}_j\mathfrak{t}}T_w
+\sum_{\substack{\mu\rhd\mu^{(t+1)}\\ \mathsf{S}\in\mathcal{T}_{0}(\mu,\mu^{(p)})\\ (\mathfrak{u},v)\in\mathcal{I}_{n-1}(\mu) }}a_{\mathsf{S},\mathfrak{u},v}m_{\mathsf{S}\mathfrak{u}}T_v
+b',
\end{align}
where $\mu$ runs over partitions of $n-2f+1$ and 
\begin{align*}
b'\in E_1E_3\cdots E_{2f-3}B_{n-1}(q,r)\cap B_{n-1}^{f}.
\end{align*}
Multiplying both sides of the expression~\eqref{junk} by $E_{2f-1}T_{w_p}^{-1}$ and using Lemma~\ref{big:0}, we obtain
\begin{align*}
q^{\ell(d(\mathfrak{s}_k))}y_{\mu^{(k)}}^\lambda T_{d(\mathfrak{v})} b&=E_{2f-1}T_{w_k}^{-1}m_{\mu^{(k)}}b+\check{B}_n^\lambda\\
&=\sum_{\substack{t<j\le k\\ (\mathfrak{t},w)\in\mathcal{I}_{n-1}(\mu^{(j)})}} a_{j,\mathfrak{t},w}\, q^{\ell(d(\mathfrak{s}_j))}y_{\mu^{(j)}}^\lambda T_{d(\mathfrak{t})}T_w
+ E_{2f-1}T_{w_p}^{-1}b'+\check{B}_n^\lambda,
\end{align*}
where $E_{2f-1}T_{w_p}^{-1}b'+\check{B}_n^\lambda\in N^{\mu^{(t)}}$ by Claim~\ref{claim:0}. Thus 
\begin{align*}
q^{\ell(d(\mathfrak{s}_k))}y_{\mu^{(k)}}^\lambda T_{d(\mathfrak{v})} b\equiv\sum_{\substack{t<j\le k\\ (\mathfrak{t},w)\in\mathcal{I}_{n-1}(\mu^{(j)})}} a_{j,\mathfrak{t},w}\,q^{\ell(d(\mathfrak{s}_j))} y_{\mu^{(j)}}^\lambda T_{d(\mathfrak{t})}T_w \mod N^{\mu^{(t)}}
\end{align*}
and 
\begin{multline}\label{junk:2}
q^{\ell(d(\mathfrak{s}_k))}y_{\mu^{(k)}}^\lambda T_{d(\mathfrak{v})} b\equiv\sum_{\substack{(\mathfrak{t},w)\in\mathcal{I}_{n-1}(\mu^{(k)})}} a_{k,\mathfrak{t},w}\,q^{\ell(d(\mathfrak{s}_k))} y_{\mu^{(k)}}^\lambda T_{d(\mathfrak{t})}T_w \\
+\sum_{\substack{t<j< k\\ (\mathfrak{t},w)\in\mathcal{I}_{n-1}(\mu^{(j)})}} a_{j,\mathfrak{t},w}\,q^{\ell(d(\mathfrak{s}_j))} y_{\mu^{(j)}}^\lambda T_{d(\mathfrak{t})}T_w \mod N^{\mu^{(t)}}.
\end{multline}
From~\eqref{pest:10} and~\eqref{junk}, the $\{a_{k,\mathfrak{t},w}\in R:(\mathfrak{t},w)\in\mathcal{I}_{n-1}(\mu^{(k)})\}$ appearing in~\eqref{junk:2} satisfy $a_{k,\mathfrak{t},w}=a_{\mathfrak{t},w}$, where 
\begin{align*}
m_{\mu^{(k)}}T_{d(\mathfrak{v})}b\equiv \sum_{(\mathfrak{t},w)\in\mathcal{I}_{n-1}(\mu^{(k)})} a_{\mathfrak{t},w}m_{\mu^{(k)}}T_{d(\mathfrak{t})}T_w\mod\check{B}_{n-1}^{\mu^{(k)}},
\end{align*}
thus demonstrating that~\eqref{filter:map} determines a $B_{n-1}(q,r)$--module isomorphism whenever $t<k\le p$.

It remains to observe that $N^{\mu^{(p)}}=\RES(S^\lambda)$. To this end,
\begin{align*}
\dim_R(N^{\mu^{(p)}})&=\sum_{i=1}^p\dim_R(N^{\mu^{(i)}}/N^{\mu^{(i-1)}})
=\sum_{\mu\to\lambda}\dim_R(S^\mu)=\dim_R(S^\lambda)
\end{align*}
where the last equality follows, for instance, from the semisimple branching law given in Theorem~2.3 of~\cite{wenzlqg}.
\end{proof}
The statement below follows from Corollary~\ref{big:5}.
\begin{theorem}\label{big:4}
Let $f$ be an integer, $0\le f\le [n/2]$, and $\lambda$ be a
partition of $n-2f$. Suppose that for each partition $\mu$ with
$\mu\to\lambda$ there exists an index set
$\mathfrak{T}_{n-1}(\mu)$ together with 
\begin{align*}
\{b_\mathfrak{u}\in
B_{n-1}(q,r):\mathfrak{u}\in\mathfrak{T}_{n-1}(\mu)\}
\end{align*}
such that
\begin{align*}
\{m_\mathfrak{u}=m_\mu
b_\mathfrak{u}+\check{B}^\mu_{n-1}:\mathfrak{u}\in
\mathfrak{T}_{n-1}(\mu)\}
\end{align*}
freely generates $S^\mu$ as an $R$--module. Then
\begin{align*}
\{y^\lambda_\mu b_\mathfrak{u}: \mathfrak{u}
\in\mathfrak{T}_{n-1}(\mu)\text{ for }\mu\to \lambda \}
\end{align*}
is a free $R$--basis for $S^\lambda$. Moreover, if $\check{N}^\mu$
denotes the $B_{n-1}(q,r)$--submodule of $S^\lambda$ generated by
\begin{align*}
\{y^\lambda_\nu b_\mathfrak{t}: \text{$\mathfrak{t}
\in\mathfrak{T}_{n-1}(\nu)$ for $\nu\to \lambda$ and
$\nu\rhd\mu$}\},
\end{align*}
then
\begin{align*}
y^\lambda_\mu b_\mathfrak{u}+\check{N}^\mu\mapsto m_\mu
b_\mathfrak{u}+\check{B}^\mu_{n-1}&&\text{for
$\mathfrak{u}\in\mathfrak{T}_{n-1}(\mu)$ with $\mu\to\lambda$,}
\end{align*}
determines an isomorphism $N^\mu/\check{N}^\mu\cong S^\mu$ of
$B_{n-1}(q,r)$--modules.
\end{theorem}

\section{New Bases for the B-M-W Algebras}\label{newbasis}
If $f$ is an integer, $0\le f\le [n/2]$, and $\lambda$ is a
partition of $n-2f$ then, appropriating the definition given
in~\cite{ramleduc:rh}, we define a \emph{path} of shape
$\lambda$ in the Bratteli diagram associated with $B_n(q,r)$ to be
a sequence of partitions
\begin{align*}
\mathfrak{t}=\left(\lambda^{(0)},\lambda^{(1)},\dots,\lambda^{(n)}\right)
\end{align*}
where $\lambda^{(0)}=\varnothing$ is the empty partition,
$\lambda^{(n)}=\lambda$, and $\lambda^{(i-1)}\to\lambda^{(i)}$,
whenever $1\le i\le n$. Let $\mathfrak{T}_{n}(\lambda)$ denote the
set of paths of shape $\lambda$ in the Bratteli diagram of
$B_{n}(q,r)$. If
$\mathfrak{t}=(\lambda^{(0)},\lambda^{(1)},\dots,\lambda^{(n)})$
is in $\mathfrak{T}_{n}(\lambda)$, and $i$ is an integer, $0\le
i\le n$, define
\begin{align*}
\mathfrak{t}|_i=\left(\lambda^{(0)},\lambda^{(1)},\dots,\lambda^{(i)}\right).
\end{align*}
The set $\mathfrak{T}_n(\lambda)$ is equipped with a dominance order
$\unrhd$ defined as follows: given paths
\begin{align*}
\mathfrak{t}=\left(\lambda^{(0)},\lambda^{(1)},\dots,\lambda^{(n)}\right)
&&\text{and}&&\mathfrak{u}=\left(\mu^{(0)},\mu^{(1)},\dots,\mu^{(n)}\right)
\end{align*}
in $\mathfrak{T}_n(\lambda)$, write
$\mathfrak{t}\unrhd\mathfrak{u}$ if $\lambda^{(k)}\unrhd\mu^{(k)}$
for $k=1,2,\dots,n$. As usual, we write
$\mathfrak{t}\rhd\mathfrak{u}$ to mean that
$\mathfrak{t}\unrhd\mathfrak{u}$ and
$\mathfrak{t}\ne\mathfrak{u}$. There is a unique path in
$\mathfrak{T}_n(\lambda)$ which is maximal with respect to the
order $\unrhd$.  Denote by $\mathfrak{t}^\lambda$ the maximal
element in $\mathfrak{T}_n(\lambda)$.
\begin{example}
Let $n=10$, $f=2$ and $\lambda=(3,2,1)$. Then
\begin{align*}
\mathfrak{t}^\lambda=\left(\varnothing,\text{\tiny$\begin{matrix}\yng(1)\end{matrix}$}\,,
\varnothing,\text{\tiny$\begin{matrix}\yng(1)\end{matrix}$}\,,\varnothing,
\text{\tiny$\begin{matrix}\yng(1)\end{matrix}$}\,,
\text{\tiny$\begin{matrix}\yng(2)\end{matrix}$}\,,
\text{\tiny$\begin{matrix}\yng(3)\end{matrix}$}\,,
\text{\tiny$\begin{matrix}\yng(3,1)\end{matrix}$}\,,
\text{\tiny$\begin{matrix}\yng(3,2)\end{matrix}$}\,,
\text{\tiny$\begin{matrix}\yng(3,2,1)\end{matrix}$}\,\right)
\end{align*}
is the maximal element in $\mathfrak{T}_n(\lambda)$ with respect
to the order $\unrhd$.
\end{example}

Let $\lambda$ be a partition of $n-2f$, for $0\le f\le [n/2]$. Theorem~\ref{big:4} will now be applied iteratively to give the $B_n(q,r)$--module $S^\lambda$ a generic basis indexed by the set $\mathfrak{T}_{n}(\lambda)$.

Assume that for each partition $\mu$ with $\mu\to\lambda$, we have defined a set
\begin{align}\label{goup:1}
\{m_\mathfrak{u}=m_\mu
b_{\mathfrak{u}}+\check{B}^\mu_{n-1}:\mathfrak{u}
\in\mathfrak{T}_{n-1}(\mu)\}
\end{align}
which freely generates $S^\mu$ as an $R$--module. 
To define $\{b_\mathfrak{t}:\mathfrak{t}\in\mathfrak{T}_n(\lambda)\}$, we refer to the definition of $y^\lambda_\mu$ given in~\eqref{ydef:1} and~\eqref{ydef:2}, and write
\begin{align}\label{goup:2}
m_{\mathfrak{t}}=y^\lambda_{\mu}b_{\mathfrak{u}}&&\text{whenever
$\mathfrak{u}\in\mathfrak{T}_{n-1}(\mu)$ and
$\mathfrak{t}|_{n-1}=\mathfrak{u}$}.
\end{align}
By Theorem~\ref{saruthm} there exist $a_w$, for
$w\in\mathfrak{S}_{n}$, depending only on $b_{\mathfrak{u}}$,
such that the term $y^\lambda_\mu b_{\mathfrak{u}}$ on the right
hand side of the expression~\eqref{goup:2} can be expressed in terms
of the basis~\eqref{bcell:1} as
\begin{align}\label{goup:3}
m_{\mathfrak{t}}=y^\lambda_{\mu}b_{\mathfrak{u}}=\sum_{w\in\mathfrak{S}_{n}}a_w
m_\lambda T_w+\check{B}^\lambda_n.
\end{align}
Thus, given $\mathfrak{t}\in\mathfrak{T}_n(\lambda)$ and
$\mathfrak{u}\in\mathfrak{T}_{n-1}(\mu)$ with
$\mathfrak{t}|_{n-1}=\mathfrak{u}$, define
\begin{align}\label{btdef}
b_\mathfrak{t}=\sum_{w\in\mathfrak{S}_{n}}a_w T_w
\end{align}
where the elements $a_w\in R$, for $w\in\mathfrak{S}_{n}$, are
determined uniquely by the basis~\eqref{bcell:1} and the
expression~\eqref{goup:3}. 

From Theorem~\ref{big:4} it follows that set
\begin{align}\label{b-murphy}
\{m_{\mathfrak{t}}=m_\lambda
b_\mathfrak{t}+\check{B}_{n}^\lambda: \mathfrak{t}
\in\mathfrak{T}_{n}(\lambda) \}
\end{align}
constructed by the above procedure is a basis for $S^\lambda$ over $R$ and that, for $1\le i\le n$, the basis~\eqref{b-murphy}
admits natural filtrations by $B_{i}(q,r)$--modules, which is
analogous to the property of the Murphy basis for
$\mathscr{H}_n(q^2)$ given in Lemma~\ref{hres}.

With little further ado, the above construction allows us to write the following.
\begin{theorem}
The algebra $B_n(q,r)$ is freely generated as an $R$ module by the collection
\begin{align*}
\mathcal{M}=\big\{ m_{\mathfrak{s}\mathfrak{t}}=b_{\mathfrak{s}}^*m_\lambda b_\mathfrak{t} :\text{$\mathfrak{s},\mathfrak{t}\in\mathfrak{T}_n(\lambda)$, $\lambda\vdash n-2f$, and  $0\le f \le [n/2]$}\big\} 
\end{align*}
Moreover the following statements hold:
\begin{enumerate}
\item The algebra anti--involution $*$ satisfies $*:m_{\mathfrak{s}\mathfrak{t}}\mapsto m_{\mathfrak{t}\mathfrak{s}}$, for all $m_{\mathfrak{s}\mathfrak{t}}\in\mathcal{M}$;
\item Suppose that $b\in B_n(q,r)$ and let $f$ be an integer $0\le f\le [n/2]$. If $\lambda$ is a partition of $n-2f$ and $\mathfrak{t}\in\mathfrak{T}_n(\lambda)$, then there exist $a_\mathfrak{v}\in R$, for $\mathfrak{v}\in\mathfrak{T}_n(\lambda)$, such that, for all $\mathfrak{s}\in\mathfrak{T}_n(\lambda)$, 
\begin{align*}
m_{\mathfrak{s}\mathfrak{t}}b\equiv \sum_{\mathfrak{v}\in\mathfrak{T}_n(\lambda)} a_\mathfrak{v} m_{\mathfrak{s}\mathfrak{v}}\mod{\check{B}_n^\lambda}.
\end{align*}
\end{enumerate}
\end{theorem}
\begin{example}\label{bas-ex:1}
We explicitly compute a basis of the form displayed
in~\eqref{b-murphy} for the $B_4(q,r)$--modules $S^\lambda$ and $S^{\lambda'}$ where $\lambda=(2)$ and $\lambda'=(1,1)$. Our iterative construction the basis for $S^\lambda$ entails explicit computation of $b_{\mathfrak{u}}$, for all
$\mathfrak{u}\in\mathfrak{T}_i(\lambda^{(i)})$ for which
\begin{align*}
(\varnothing,\dots,\lambda^{(i-1)},\lambda^{(i)},\dots,\lambda)\in\mathfrak{T}_4(\lambda),
\end{align*}
with similar requirements for computing the basis for $S^{\lambda'}$. 

(\emph{a}) The algebra $B_2(q,r)$ has three one dimensional cell
modules; if $\mu$ is one of the partitions $\varnothing$, $(2)$ or $(1,1)$,
associate to the path in $\mathfrak{T}_2(\mu)$ an element of $S^\mu$
as
\begin{align*}
(\varnothing,\text{\tiny$\begin{matrix}\yng(1)\end{matrix}$}\,,\varnothing)\hspace{0.5em}&\mapsto\hspace{0.5em}E_1;\\
(\varnothing,\text{\tiny$\begin{matrix}\yng(1)\end{matrix}$}\,,\text{\tiny$\begin{matrix}\yng(2)\end{matrix}$}\,)\hspace{0.5em}&\mapsto\hspace{0.5em}(1+qT_1)+\check{B}_2^{(2)}\\
(\varnothing,\text{\tiny$\yng(1)$}\,,\text{\tiny$\begin{matrix}\yng(1,1)\end{matrix}$}\,)\hspace{0.5em}
&\mapsto\hspace{0.5em} 1+B_2^{(1,1)},
\end{align*}
to obtain a cellular basis for $B_2(q,r)$ which is compatible with
the ordering of partitions $\varnothing\rhd(2)\rhd(1,1)$.

(\emph{b}) The algebra $B_3(q,r)$ has four cell modules, one
corresponding to each of the partitions,
$(1)\rhd(3)\rhd(2,1)\rhd(1^3)$.

(i) If $\mu=(1)$ then $\check{B}^\mu_3=0$ and $m_\mu=E_1$; since $\nu\to\mu$ precisely if
$\nu$ is one of $\varnothing\rhd(2)\rhd(1,1)$, using part~(\emph{a}) above, we associate to each path in $\mathfrak{T}_3(\mu)$ an element of $S^\mu$ as
\begin{align*}
(\varnothing,\text{\tiny$\begin{matrix}\yng(1)\end{matrix}$}\,,\varnothing,\text{\tiny$\begin{matrix}\yng(1)\end{matrix}$}\,)\hspace{0.5em}&\mapsto\hspace{0.5em}m_{\mathfrak{t}^\mu}=E_1;\\
(\varnothing,\text{\tiny$\begin{matrix}\yng(1)\end{matrix}$}\,,\text{\tiny$\begin{matrix}\yng(2)\end{matrix}$}\,,\text{\tiny$\yng(1)$}\,)\hspace{0.5em}&\mapsto\hspace{0.5em} m_{\mathfrak{t}^\mu} T_2^{-1}T_1^{-1}(1+qT_1);\\
(\varnothing,\text{\tiny$\begin{matrix}\yng(1)\end{matrix}$}\,,\text{\tiny$\begin{matrix}\yng(1,1)\end{matrix}$}\,,\text{\tiny$\begin{matrix}\yng(1)\end{matrix}$}\,) \hspace{0.5em}&\mapsto  \hspace{0.5em}m_{\mathfrak{t}^\mu}T_2^{-1}T_1^{-1}=m_{\mathfrak{t}^\mu}T_2T_1.
\end{align*}
The transition matrix from the basis $\{m_\mathfrak{t}=m_\lambda b_\mathfrak{t}+\check{B}_3^\mu:\mathfrak{t}\in\mathfrak{T}_3(\mu)\}$ for $S^\mu$ given in~\eqref{b-murphy} and ordered by dominance as above, to the ordered basis 
\begin{align*}
\{\mathbf{v}_i=m_\mu T_{v_i}:v_1=1,v_2=s_2,v_3=s_2s_1 \} 
\end{align*}
for $S^\mu$ given in~\eqref{bcell:1} is:
\begin{align}\label{t-matrix:1}
\begin{bmatrix}
1 & 1-q^2&0\\
0 & q & 0\\
0 & q^2 & 1
\end{bmatrix}.
\end{align}
The elements $\{ b_\mathfrak{t}:\mathfrak{t}\in\mathfrak{T}_3(\mu)\}$ of~\eqref{b-murphy} are made explicit by the above transition matrix.

(ii) If $\mu=(3)$, then $S^\mu$ is one--dimensional and
\begin{align*}
(\varnothing,\text{\tiny$\begin{matrix}\yng(1)\end{matrix}$}\,,\text{\tiny$\begin{matrix}\yng(2)\end{matrix}$}\,,\text{\tiny$\begin{matrix}\yng(3)\end{matrix}$})\hspace{.5em}\mapsto\hspace{0.5em}m_{\mathfrak{t}^\mu}=(1+qT_1)(1+qT_2+q^2T_2T_1)+\check{B}^{(3)}_3.
\end{align*}

(iii) If $\mu=(2,1)$, then $m_\mu=(1+qT_1)$ and a basis for $S^\mu$ is
obtained by associating to each path in $\mathfrak{T}_3(\mu)$ an
element as
\begin{align*}
\left(\varnothing,\text{\tiny$\begin{matrix}\yng(1)\end{matrix}$}\,,\text{\tiny$\begin{matrix}\yng(2)\end{matrix}$}\,,\text{\tiny$\begin{matrix}\yng(2,1)\end{matrix}$}\,\right)\hspace{0.5em}&\mapsto\hspace{0.5em}m_{\mathfrak{t}^\mu}=(1+qT_1)+\check{B}^{(2,1)}_3;\\
\left(\varnothing,\text{\tiny$\begin{matrix}\yng(1)\end{matrix}$}\,,\text{\tiny$\begin{matrix}\yng(1,1)\end{matrix}$}\,,\text{\tiny$\begin{matrix}\yng(2,1)\end{matrix}$}\,\right)\hspace{0.5em}&\mapsto\hspace{0.5em}m_{\mathfrak{t}^\mu}T_2.
\end{align*}

(iv) Finally, if $\mu=(1,1,1)$, then $S^\mu$ is the right $B_3(q,r)$--module generated by
$1+\check{B}^{(1,1,1)}_3$.

(\emph{c}) Let $n=4$ and $\lambda=(2)$. Then $m_\lambda=E_1(1+qT_3)$
and $\mu\to\lambda$ if $\mu$ is one of the partitions
$\mu^{(1)}=(1)\rhd\mu^{(2)}=(3)\rhd\mu^{(3)}=(2,1)$. Thus, based on~(\emph{b}) above, we associate
to each path $\mathfrak{t}\in\mathfrak{T}_4(\lambda)$ a basis
element of the cell module $S^\lambda$ as follows:
\begin{align*}
\begin{split}
(\varnothing,\text{\tiny$\begin{matrix}\yng(1)\end{matrix}$}\,,\varnothing,
\text{\tiny$\begin{matrix}\yng(1)\end{matrix}$}\,,
\text{\tiny$\begin{matrix}\yng(2)\end{matrix}$}\,)\hspace{0.5em}\mapsto\hspace{0.5em}
&y_{\mu^{(1)}}^\lambda=m_{\mathfrak{t}^\lambda}=E_1(1+qT_3) +\check{B}^{\lambda}_4;\\
(\varnothing,\text{\tiny$\begin{matrix}\yng(1)\end{matrix}$}\,,
\text{\tiny$\begin{matrix}\yng(2)\end{matrix}$}\,,
\text{\tiny$\begin{matrix}\yng(1)\end{matrix}$}\,,
\text{\tiny$\begin{matrix}\yng(2)\end{matrix}$}\,)\hspace{0.5em}\mapsto\hspace{0.5em}
&y_{\mu^{(1)}}^{\lambda}T_2^{-1}T_1^{-1}(1+qT_1);\\
(\varnothing,\text{\tiny$\begin{matrix}\yng(1)\end{matrix}$}\,,
\text{\tiny$\begin{matrix}\yng(1,1)\end{matrix}$}\,,
\text{\tiny$\begin{matrix}\yng(1)\end{matrix}$}\,,
\text{\tiny$\begin{matrix}\yng(2)\end{matrix}$}\,)\hspace{0.5em}\mapsto\hspace{0.5em}
&y_{\mu^{(1)}}^{\lambda} T_2^{-1}T_1^{-1}=m_{\mathfrak{t}^\lambda} T_2T_1;\\
(\varnothing,\text{\tiny$\begin{matrix}\yng(1)\end{matrix}$}\,,
\text{\tiny$\begin{matrix}\yng(2)\end{matrix}$}\,,
\text{\tiny$\begin{matrix}\yng(3)\end{matrix}$}\,,
\text{\tiny$\begin{matrix}\yng(2)\end{matrix}$}\,)\hspace{0.5em}\mapsto\hspace{0.5em}
&y_{\mu^{(2)}}^\lambda =E_1(T_2T_1T_3T_2)^{-1}m_{\mu^{(2)}}+\check{B}_4^\lambda\\
&= E_1 (1+qT_3)(T_2T_1T_3T_2)^{-1}\\
&\quad\quad\times(1+qT_2+q^2T_2T_1)+\check{B}_4^\lambda\\
&= m_{\mathfrak{t}^\lambda} (T_2T_1T_3T_2)^{-1}(1+qT_2+q^2T_2T_1);\\
(\varnothing,\text{\tiny$\begin{matrix}\yng(1)\end{matrix}$}\,,
\text{\tiny$\begin{matrix}\yng(2)\end{matrix}$}\,,
\text{\tiny$\begin{matrix}\yng(2,1)\end{matrix}$}\,,
\text{\tiny$\begin{matrix}\yng(2)\end{matrix}$}\,)\hspace{0.5em}\mapsto\hspace{0.5em}
&y_{\mu^{(3)}}^\lambda= E_1(T_2T_1T_3T_2)^{-1}m_{\mu^{(3)}} +\check{B}_4^\lambda\\ 
&=m_{\mathfrak{t}^\lambda}
(T_2T_1T_3T_2)^{-1}=m_{\mathfrak{t}^\lambda} T_2T_3T_1T_2;\\
(\varnothing,\text{\tiny$\begin{matrix}\yng(1)\end{matrix}$}\,,
\text{\tiny$\begin{matrix}\yng(1,1)\end{matrix}$}\,,
\text{\tiny$\begin{matrix}\yng(2,1)\end{matrix}$}\,,
\text{\tiny$\begin{matrix}\yng(2)\end{matrix}$}\,)\hspace{0.5em}\mapsto\hspace{0.5em}
&y_{\mu^{(3)}}^\lambda T_2= m_{\mathfrak{t}^\lambda}
(T_2T_1T_3T_2)^{-1}T_2\\
&=m_{\mathfrak{t}^\lambda} T_2^{-1}T_3^{-1}T_1^{-1}.
\end{split}
\end{align*}
Expanding the terms on the right hand side above using results from Section~3 of~\cite{saru}, we obtain the transition matrix from the basis $\{m_\mathfrak{t}=m_\lambda b_\mathfrak{t}+\check{B}_4^\lambda:\mathfrak{t}\in\mathfrak{T}_4(\lambda)\}$ for $S^\lambda$ given in~\eqref{b-murphy} and ordered by dominance as above, to basis 
\begin{align*}
\{\mathbf{v}_{i,j}=m_\lambda T_{v_{i,j}}+\check{B}_4^\lambda: v_{i,j}=s_2s_3\cdots s_{j-1}s_1s_2\cdots s_{i-1} \}
\end{align*}
for $S^\lambda$ given in~\eqref{bcell:1}, ordered lexicographically, as:
\begin{align*}
\begin{bmatrix}
1 & 1-q^2	&0	& 1-q^2 	& 0	& 0	\\
0 & q 		& 0	& q(1-q^2) 	& 0	& 0	\\
0 & 0		& 0	& q^2		& 0	& 0	\\
0 & q^2 	& 1	& q^2(1-q^2) 	& 0	& \frac{\displaystyle 1-q^2}{\displaystyle q} \\
0 & 0	 	& 0 	& q^3		& 0	& 1	\\
0 & 0		& 0	& q^4		& 1	& \frac{\displaystyle q^2-1}{\displaystyle q} 
\end{bmatrix}.
\end{align*}
It may be observed that the elements $\{b_\mathfrak{t}:\mathfrak{t}\in\mathfrak{T}_4(\lambda)\}$, given by the above matrix, are consistent with~\eqref{t-matrix:1} above and reflect the existence of an embedding $S^{\mu^{(1)}}\hookrightarrow S^\lambda$ of $B_3(q,r)$--modules, as $N^{\mu^{(1)}}/\check{N}^{\mu^{(1)}}\cong S^{\mu^{(1)}}$, where $\check{N}^{\mu^{(1)}}=0$.

(\emph{d}) Now consider the partition $\lambda'=(1,1)$; here $m_{\lambda'}=E_1$ and $\mu\to\lambda'$ if $\mu$ is one of the partitions $\mu^{(1)}=(1)\rhd\mu^{(2)}=(2,1)\rhd\mu^{(3)}=(1,1,1)$; thus, based on Example~\ref{smallex:0} and the calculations (\emph{b}) above, we associate to each path $\mathfrak{t}\in\mathfrak{T}_4(\lambda')$ a basis element in the cell module $S^{\lambda'}$ as follows:
\begin{align*}
\begin{split}
(\varnothing,\text{\tiny$\begin{matrix}\yng(1)\end{matrix}$}\,,\varnothing,
\text{\tiny$\begin{matrix}\yng(1)\end{matrix}$}\,,
\text{\tiny$\begin{matrix}\yng(1,1)\end{matrix}$}\,)\hspace{0.5em}\mapsto\hspace{0.5em}
&y_{\mu^{(1)}}^{\lambda'}=m_{\mathfrak{t}^{\lambda'}}=E_1 +\check{B}^{\lambda'}_4;\\
(\varnothing,\text{\tiny$\begin{matrix}\yng(1)\end{matrix}$}\,,
\text{\tiny$\begin{matrix}\yng(2)\end{matrix}$}\,,
\text{\tiny$\begin{matrix}\yng(1)\end{matrix}$}\,,
\text{\tiny$\begin{matrix}\yng(1,1)\end{matrix}$}\,)\hspace{0.5em}\mapsto\hspace{0.5em}
&y_{\mu^{(1)}}^{\lambda'}T_2^{-1}T_1^{-1}(1+qT_1);\\
(\varnothing,\text{\tiny$\begin{matrix}\yng(1)\end{matrix}$}\,,
\text{\tiny$\begin{matrix}\yng(1,1)\end{matrix}$}\,,
\text{\tiny$\begin{matrix}\yng(1)\end{matrix}$}\,,
\text{\tiny$\begin{matrix}\yng(1,1)\end{matrix}$}\,)\hspace{0.5em}\mapsto\hspace{0.5em}
&y_{\mu^{(1)}}^{\lambda'}T_2^{-1}T_1^{-1}=m_{\mathfrak{t}^\lambda} T_2T_1;\\
(\varnothing,\text{\tiny$\begin{matrix}\yng(1)\end{matrix}$}\,,
\text{\tiny$\begin{matrix}\yng(2)\end{matrix}$}\,,
\text{\tiny$\begin{matrix}\yng(2,1)\end{matrix}$}\,,
\text{\tiny$\begin{matrix}\yng(1,1)\end{matrix}$}\,)\hspace{0.5em}\mapsto\hspace{0.5em}
&y_{\mu^{(2)}}^{\lambda'}=E_{1}(T_{2}T_{1}T_{3}T_{2})^{-1}T_2 \,m_{\mu^{(1)}}+\check{B}^{\lambda'}_4\\
&=m_{\mathfrak{t}^{\lambda'}} T_2^{-1}T_3^{-1}T_1^{-1}(1+qT_1);\\
(\varnothing,\text{\tiny$\begin{matrix}\yng(1)\end{matrix}$}\,,
\text{\tiny$\begin{matrix}\yng(1,1)\end{matrix}$}\,,
\text{\tiny$\begin{matrix}\yng(2,1)\end{matrix}$}\,,
\text{\tiny$\begin{matrix}\yng(1,1)\end{matrix}$}\,)\hspace{0.5em}\mapsto\hspace{0.5em}
&y_{\mu^{(2)}}^{\lambda'}T_2=m_{\mathfrak{t}^{\lambda'}}
T_2^{-1}T_3^{-1}T_1^{-1}(1+qT_1)T_2;\\
(\varnothing,\text{\tiny$\begin{matrix}\yng(1)\end{matrix}$}\,,
\text{\tiny$\begin{matrix}\yng(1,1)\end{matrix}$}\,,
\text{\tiny$\begin{matrix}\yng(1,1,1)\end{matrix}$}\,,
\text{\tiny$\begin{matrix}\yng(1,1)\end{matrix}$}\,)\hspace{0.5em}\mapsto\hspace{0.5em}
&y_{\mu^{(3)}}^{\lambda'}=E_1(T_2T_1T_3T_2)^{-1} m_{\mu^{(2)}}+\check{B}^{\lambda'}_4\\
&=m_{\mathfrak{t}^{\lambda'}}
(T_2T_1T_3T_2)^{-1}=m_{\mathfrak{t}^{\lambda'}} T_2T_3T_1T_2.
\end{split}
\end{align*}
The transition matrix from the basis $\{m_\mathfrak{t}=m_{\lambda'} b_\mathfrak{t}+\check{B}_4^{\lambda'}:\mathfrak{t}\in\mathfrak{T}_4(\lambda')\}$ for $S^{\lambda'}$ given in~\eqref{b-murphy} and ordered by dominance, to the basis 
\begin{align*}
\{\mathbf{v}_{i,j}=m_{\lambda'} T_{v_{i,j}} +\check{B}_4^{\lambda '}: v_{i,j}=s_2s_3\cdots s_{j-1}s_1s_2\cdots s_{i-1} \}
\end{align*}
for $S^{\lambda'}$ given in~\eqref{bcell:1} and ordered lexicographically, is:
\begin{align}\label{t-matrix:3}
\begin{bmatrix}
1 & 1-q^2	& 0	& q(q^2-1)	&  1-q^2	& 0	\\
0 & q^2 	& 0	& 1-q^2 	& \frac{\displaystyle q^2-1}{\displaystyle q}	& 0	\\
0 & 0		& 0	& q		& -1	& 0	\\
0 & q^3 	& 1	& q(1-q^2) 	& \frac{\displaystyle 1-q^2}{\displaystyle qr} 	& 0	 \\
0 & 0	 	& 0 	& q^2		& 0	& 0	\\
0 & 0		& 0	& 0		& q^2	& 1
\end{bmatrix}.
\end{align}
The elements $\{b_\mathfrak{t}:\mathfrak{t}\in\mathfrak{T}_n(\lambda')\}$ are made explicit by the above transition matrix.\end{example}
\begin{example}\label{b-m-w-murphex:2}
Let $n=5$ and $\lambda=(2,1)$. Then $\mu\to\lambda$ if $\mu$ is one of the partitions 
\begin{align*}
\mu^{(1)}=(2)\rhd\mu^{(2)}=(1,1)\rhd \mu^{(3)}=(3,1)\rhd\mu^{(4)}=(2,2)\rhd \mu^{(5)}=(2,1,1).
\end{align*}
By considering a suitable basis for $N^{\mu^{(2)}}/N^{\mu^{(1)}}$, we make explicit the elements $b_\mathfrak{t}$, for $\mathfrak{t}\in\mathfrak{T}_n(\lambda)$, defined by~\eqref{goup:3} and 
\begin{align*}
\big\{ y_{\mu^{(2)}}^\lambda b_{\mathfrak{u}}=m_{\mathfrak{t}^\lambda} \,b_\mathfrak{t}: \text{$\mathfrak{t}\in\mathfrak{T}_n(\lambda)$ and $\mathfrak{t}|_{n-1}=\mathfrak{u}\in\mathfrak{T}_{n-1}(\mu^{(2)})$}\big\}.
\end{align*}
For brevity, write $\mu=\mu^{(2)}$. Since $\mathfrak{s}=\text{\tiny\Yvcentermath1$\young(35,4)$}$ satisfies $\mathfrak{s}|_{n-1}=\mathfrak{t}^\mu$, we have $y_\mu^\lambda=m_{\mathfrak{t}^\lambda} T_{d(\mathfrak{s})}=m_{\mathfrak{t}^\lambda }T_{4}$, where where $m_{\mathfrak{t}^\lambda}=E_1(1+qT_3)+\check{B}_n^\lambda$. The transition matrix from the basis
\begin{align*}
\big\{y_\mu^\lambda\, b_{\mathfrak{u}}+\check{N}^\mu=m_{\mathfrak{t}^\lambda} b_\mathfrak{t}+\check{N}^{\mu}:\text{$\mathfrak{t}\in\mathfrak{T}_n(\lambda)$ and $\mathfrak{t}|_{n-1}=\mathfrak{u}\in\mathfrak{T}_{n-1}(\mu)$}\big\},
\end{align*}
which is ordered by dominance, to the basis
\begin{align*}
\big\{\mathbf{v}_{i,j}=m_{\mathfrak{t}^\lambda}T_{d(\mathfrak{s})} T_{v_{i,j}}+\check{N}^\mu: v_{i,j}=s_2s_3\cdots s_{j-1}s_1s_2\cdots s_{i-1} \big\},
\end{align*}
which we order lexicographically, is given by~\eqref{t-matrix:3} above. Observe that though $N^{\mu}/\check{N}^{\mu}\cong S^{\mu^{(2)}}$ as $B_{n-1}(q,r)$--modules, the construction does not give an embedding $S^{\mu^{(2)}}\hookrightarrow S^\lambda$ of $B_{n-1}(q,r)$--modules.
\end{example}

\section{Jucys--Murphy Operators}\label{jmops}
Define the operators $L_i\in B_n(q,r)$, for $i=1,2,\dots,n$, by
$L_1=1$ and $L_{i}=T_{i-1}L_{i-1}T_{i-1}$ when $i=2,\dots,n$. Let
$\mathscr{L}=\mathscr{L}_n$ denote the subalgebra of $B_n(q,r)$
generated by $L_1,\dots,L_n$. The next statement, which is the
analogue to Proposition~\ref{murphyop:1}, is easily obtained from
the braid relation $T_{i}T_{i+1}T_{i}=T_{i+1}T_{i}T_{i+1}$.
\begin{proposition}\label{murphyop:2}
Let $i$ and $k$ be integers, $1\le i<n$ and $1\le k\le n$. Then the
following statements hold.
\begin{enumerate}
\item $T_i$ and $L_k$ commute if $i\ne k-1,k$. \item $L_i$ and
$L_k$ commute.\item $T_i$ commutes with $L_iL_{i+1}$. \item
$L_2\cdots L_n$ belongs to the centre of $B_n(q,r)$.
\end{enumerate}
\end{proposition}
\begin{remark}
(i) The elements $L_i$ are a special case of certain operators defined in Corollary~1.6 of~\cite{ramleduc:rh} in a context of semisimple path algebras.

(ii) The elements $L_i$ bear an analogy with the Jucys--Murphy operators $D_i$ defined in Section~\ref{ihsec}; we therefore refer to the $L_i$ as ``Jucys--Murphy operators'' in $B_n(q,r)$.
\end{remark}
For integers $j,k$, with $1\le j,k \le n$, define the elements
$L_k^{(j)}$ by $L_1^{(j)}=1$ and
\begin{align*}
L_k^{(j)}=T_{j+k-2}L_{k-1}^{(j)}T_{j+k-2}, &&\text{for $k\ge2$.}
\end{align*}
In particular $L_k^{(1)}$, for $k=1,\dots,n$, are the usual
Jucys--Murphy operators in $B_n(q,r)$. 

The next proposition is a step on the way to showing that the set $\{m_\mathfrak{t}=m_\lambda b_\mathfrak{t}+\check{B}^\lambda_n:\mathfrak{t}\in\mathfrak{T}_n(\lambda)\}$ defined in~\eqref{b-murphy} above is a basis of generalised eigenvectors for the action of Jucys--Murphy operators on the cell module $S^\lambda$.
\begin{proposition}\label{b-m-wshorten}
Let $i,k$ be integers with $1\le i\le n$ and $1<k\le n$. Then
\begin{align*}
E_i L_k^{(i)}=
\begin{cases}
r^{-2}E_i&\text{if $k=2$;}\\
E_i&\text{if $k=3$;}\\
E_iL_{k-2}^{(i+2)}&\text{if $k\ge 4$.}
\end{cases}
\end{align*}
\end{proposition}
\begin{proof}
If $k=2$, then $E_i L_k^{(i)}=E_iT_i^2=r^{-2}E_i$. For $k=3$, we use
the relations $E_iE_{i+1}=E_iT_{i+1}T_i=T_{i+1}T_iE_{i+1}$ and
$E_iE_{i+1}E_i=E_i$ to obtain
\begin{align}\label{groble}
E_iL_3^{(i)}=E_iT_{i+1}T_iT_iT_{i+1}=E_iE_{i+1}T_iT_{i+1}=E_iE_{i+1}E_i=E_i.
\end{align}
If $k\ge 4$, then using~\eqref{groble},
\begin{align*}
E_iL_k^{(i)}&=E_iT_{i+k-2}T_{i+k-3}\cdots T_{i+2}L_3^{(i)}T_{i+2}\cdots T_{i+k-3}T_{i+k-2}\\
&=T_{i+k-2}T_{i+k-3}\cdots T_{i+2}E_iL_3^{(i)}T_{i+2}\cdots T_{i+k-3}T_{i+k-2}\\
&=E_iT_{i+k-2}T_{i+k-3}\cdots T_{i+2}T_{i+2}\cdots T_{i+k-3}T_{i+k-2}=E_iL^{(i+2)}_{k-2}.
\end{align*}
\end{proof}

\begin{corollary}\label{groblecor}
Let $f,k$ be integers, $0<f\le[n/2]$ and $1\le k\le n$. Then
\begin{align*}
E_1E_3\cdots E_{2f-1}L_k
=
\begin{cases}
E_1E_3\cdots E_{2f-1}&\text{if $k$ is odd, $1\le
k\le2f+1$;}\\
r^{-2}E_1E_3\cdots E_{2f-1}&\text{if $k$ is even, $1<
k\le2f$;}\\
E_1E_3\cdots E_{2f-1}L_{k-2f}^{(2f+1)}&\text{if $2f+1<k\le n$.}
\end{cases}
\end{align*}
\end{corollary}
\begin{proof}
If $k$ is odd, $1<k\le 2f+1$, then by the proposition above,
\begin{multline}\begin{split}\label{grobles}
E_1E_3\cdots E_k L_k=E_1E_3\cdots E_kL_k^{(1)}=E_1E_3\cdots
E_kL_{k-2}^{(3)}=\cdots\\
\cdots=E_1E_3\cdots E_kL_{1}^{(k)}=E_1E_3\cdots E_k.\end{split}
\end{multline}
Since $E_{k+2}E_{k+4}\cdots E_{2f-1}$ commutes with $L_k$, the first
statement has been proved. If $k$ is even, $1<k\le 2f$, then use the
relation $E_iT_i=r^{-1}T_i$ and~\eqref{grobles} so that
\begin{align*}
E_1E_3\cdots E_{2f-1}L_k&=E_1E_3\cdots
E_{2f-1}T_{k-1}L_{k-1}T_{k-1}\\
&=r^{-1}E_1E_3\cdots E_{2f-1}L_{k-1}T_{k-1}=r^{-2}E_1E_3\cdots
E_{2f-1},
\end{align*}
as above. The final case where $2f+1<k\le n$ is similar
to~\eqref{grobles} above.
\end{proof}
Let $f$ be an integer, $0\le f\le [n/2]$, and $\lambda$ be a
partition of $n-2f$. Suppose that
$\mathfrak{t}=(\lambda^{(0)},\lambda^{(1)},\dots,\lambda^{(n)})$ is
a path in $\mathfrak{T}_n(\lambda)$, and that $k$ is an integer,
$1\le k\le n$. Then generalise the definition~\eqref{mon:1} by
writing
\begin{align*}
P_\mathfrak{t}(k)=
\begin{cases}
q^{2(j-i)}&\text{if $[\lambda^{(k)}]=[\lambda^{(k-1)}]\cup
\{(i,j)\}$}\\
q^{2(i-j)}r^{-2}&\text{if
$[\lambda^{(k)}]=[\lambda^{(k-1)}]\setminus \{(i,j)\}$}.
\end{cases}
\end{align*}
Since $q$ does not have finite multiplicative order in $R$, the next result which is similar in flavour to Lemma~5.20 of~\cite{ramleduc:rh}, follows immediately from Lemma~3.34 of~\cite{mathas:ih}.
\begin{lemma}\label{ram:dist0}
Let $f$ be an integer $0\le f<[n/2]$ and $\lambda^{(n-1)}$ be a
partition of $n-1-2f$. If
$\mathfrak{s}=(\lambda^{(0)},\lambda^{(1)},\dots,\lambda^{(n-1)})$
is a path in $\mathfrak{T}_{n-1}(\lambda^{(n-1)})$, then the terms
$\big(P_{\mathfrak{t}}(n)\,:\,\mathfrak{t}|_{n-1}=\mathfrak{s}\big)$
are all distinct.
\end{lemma}
The next proposition is essentially a restatement of
Theorem~\ref{utrangular:1}.
Recall that if $f$ is an integer, $0\le f\le [n/2]$, and $\lambda$
is a partition of $n-2f$, then $\mathfrak{t}^\lambda$ is the element in
$\mathfrak{T}_n(\lambda)$ which is maximal under the dominance order.
\begin{proposition}\label{utrangular:3}
If $\lambda$ is a partition of $n$ and $k$ is an integer $1\le
k\le n$, then
$m_{\mathfrak{t}^\lambda}L_k=P_{\mathfrak{t}^\lambda}(k)m_{\mathfrak{t}^\lambda}$.
\end{proposition}
\begin{proof}
By definition,
$m_{\mathfrak{t}^\lambda}=m_\lambda+\check{B}^\lambda_n$ so, using
the property~\eqref{nearh},
\begin{align*}
m_\lambda L_k+B_n^1=\vartheta_0(c_\lambda D_{k})=
P_{\mathfrak{t}^\lambda}(k)\vartheta_0(c_\lambda)
\end{align*}
where the last equality follows from Theorem~\ref{utrangular:1}. Now, given that $B^1_n\subseteq \check{B}^\lambda_n$ whenever $\lambda$ is a
partition of $n$, the result follows.
\end{proof}
\begin{proposition}\label{danaming}
Let $f$ be an integer, $0< f\le [n/2]$, and $\lambda$ be a partition
of $n-2f$. Then $m_{\mathfrak{t}^\lambda}
L_k=P_{\mathfrak{t}^\lambda}(k)m_{\mathfrak{t}^\lambda}$.
\end{proposition}
\begin{proof}
If $k$ is an integer, $1\le k\le 2f+1$, the statement follows from
Corollary~\ref{groblecor}; otherwise, using the corollary and
property~\eqref{nearh},
\begin{align*}
m_{\lambda}L_k+B_n^{f+1}&=x_\lambda E_1E_3\cdots E_{2f-1} L_k+B_n^{f+1}\\
&=x_\lambda E_1E_3\cdots E_{2f-1}L_{k-2f}^{(2f+1)}+B_n^{f+1}\\
&=\vartheta_f(c_\lambda D_{k-2f})=
P_{\hat{\mathfrak{t}}^\lambda}(k-2f)\vartheta_f(c_\lambda)\\
&=P_{{\mathfrak{t}}^\lambda}(k)m_\lambda+{B_n^{f+1}},
\end{align*}
whence the result follows, since $B_n^{f+1}\subseteq
\check{B}_n^\lambda$ whenever $\lambda$ is a partition of $n-2f$.
\end{proof}
\begin{proposition}\label{central}
Let $f$ be an integer, $0\le f\le[n/2]$ and $\lambda$ be a partition of $n-2f$. Then there exists an invariant $\alpha\in R$ such that $(L_2\cdots L_n)$ acts on $S^\lambda$ as a multiple by $\alpha$ of the
identity.
\end{proposition}
\begin{proof}
Consider an element $\sum_{w\in\mathfrak{S}_n}a_w m_\lambda
T_w +\check{B}_n^\lambda$, for $a_w\in R$. Since $(L_2\cdots L_n)$ is central in
$B_n(q,r)$,
\begin{align*}
\sum_{w\in\mathfrak{S}_n}a_w m_\lambda T_w (L_2\cdots L_n) &=\sum_{w\in\mathfrak{S}_n}a_w m_\lambda (L_2\cdots L_n)T_w,
\end{align*}
so $\alpha=\prod_{k=2}^n P_{\mathfrak{t}^\lambda}(k)$, by the
previous proposition.
\end{proof}
For the proof of Theorem~\ref{utrangular:2} we use the filtration of
the $B_n(q,r)$ module $S^\lambda$ by $B_{n-1}(q,r)$--modules given
in Theorem~\ref{big:4}.
\begin{theorem}\label{utrangular:2}
Let $f$ be an integer, $0\le f\le [n/2]$, and $\lambda$ be a
partition of $n-2f$. If $\mathfrak{t}\in\mathfrak{T}_n(\lambda)$,
then there exist $a_\mathfrak{u}\in R$, for
$\mathfrak{u}\in\mathfrak{T}_n(\lambda)$, such that
\begin{align*}
m_\mathfrak{t} L_k=P_{\mathfrak{t}}(k)m_\mathfrak{t}+
\sum_{\substack{\mathfrak{u}\in\mathfrak{T}_n(\lambda)\\\mathfrak{u}\rhd\mathfrak{t}}}
a_\mathfrak{u}m_\mathfrak{u}.
\end{align*}
\end{theorem}
\begin{proof}
We proceed by induction. Let $\mathfrak{t}$ be in $\mathfrak{T}_n(\lambda)$ and suppose that $\mathfrak{s}=\mathfrak{t}|_{n-1}$ is an element of $\mathfrak{T}_{n-1}(\mu)$. Then $m_\mathfrak{t}+\check{N}^\mu\mapsto m_\mathfrak{s}$ under the isomorphism $N^\mu/\check{N}^\mu\to S^\mu$ of $B_{n-1}(q,r)$--modules given in Theorem~\ref{big:4}. Hence, if $1\le k<n$, there exist $a_\mathfrak{v}\in R$, for
$\mathfrak{v}\in\mathfrak{T}_{n-1}(\mu)$, such that
\begin{align*}
m_\mathfrak{t}L_k+\check{N}^\mu\mapsto P_\mathfrak{s}(k)
m_\mathfrak{s}+\sum_{\substack{\mathfrak{v}\in\mathfrak{T}_{n-1}(\mu)\\
\mathfrak{v}\rhd\mathfrak{s}}}a_\mathfrak{v}m_\mathfrak{v}
\end{align*}
under the $B_{n-1}(q,r)$--module isomorphism $N^\mu/\check{N}^\mu\to S^\mu$. Thus the $a_\mathfrak{v}\in R$, for $\mathfrak{v}\in\mathfrak{T}_{n-1}(\mu)$, satisfy
\begin{align*}
m_\mathfrak{t}L_k\equiv P_\mathfrak{s}(k)
m_\mathfrak{t}+\sum_{\substack{\mathfrak{v}\in\mathfrak{T}_{n-1}(\mu)\\
\mathfrak{v}\rhd\mathfrak{s}}}a_\mathfrak{v}y^\lambda_\mu b_\mathfrak{v} \mod{\check{N}^\mu}.
\end{align*}
If $\mathfrak{v}\in\mathfrak{T}_{n-1}(\mu)$ and $\mathfrak{v}\rhd\mathfrak{s}$, then, using the definition~\eqref{goup:3}, $y_\mu^\lambda b_\mathfrak{v}=m_\mathfrak{u}$, where $\mathfrak{u}|_{n-1}=\mathfrak{v}\rhd\mathfrak{s}=\mathfrak{t}|_{n-1}$, and thus $\mathfrak{u}\rhd\mathfrak{t}$. Since $P_\mathfrak{t}(k)=P_\mathfrak{s}(k)$ whenever $1\le k<n$, the above expression becomes 
\begin{align}\label{becomes}
m_\mathfrak{t}L_k\equiv P_\mathfrak{t}(k)
m_\mathfrak{t}+\sum_{\substack{\mathfrak{u}\in\mathfrak{T}_{n}(\lambda)\\
\mathfrak{u}\rhd\mathfrak{t}}}a_\mathfrak{u}m_\mathfrak{u} \mod{\check{N}^\mu},
\end{align}
where $a_\mathfrak{u}=a_\mathfrak{v}$ whenever $\mathfrak{u}|_{n-1}=\mathfrak{v}$. Now, $\check{N}^\mu$ is the $B_{n-1}(q,r)$--module freely generated by  
\begin{align*}
\big\{ m_\mathfrak{u}=y_\nu^\lambda b_\mathfrak{v}:\text{$\mathfrak{u}\in\mathfrak{T}_{n}(\lambda)$, $\nu\to\lambda$, $\nu\rhd\mu$ and $\mathfrak{u}|_{n-1}=\mathfrak{v}\in\mathfrak{T}_{n-1}(\nu)$} \big\},
\end{align*}
and so it follows that $\check{N}^\mu$ is contained in the $R$--submodule of $S^\lambda$ generated by $\{m_\mathfrak{u}:\text{$\mathfrak{u}\in\mathfrak{T}_n(\lambda)$ and $\mathfrak{u}\rhd\mathfrak{t}$}\}$. Thus~\eqref{becomes} shows that the theorem holds true whenever $1\le k<n$.

That $L_n$ acts triangularly on $S^\lambda$, can now be deduced using
Proposition~\ref{central}:
\begin{align*}
m_\mathfrak{t}L_n=\prod_{k=1}^{n}P_\mathfrak{t}(k)m_\mathfrak{t}(L_2L_3\cdots
L_{n-1})^{-1}.
\end{align*}
Thus the generalised eigenvalue for $L_n$ acting on
$m_\mathfrak{t}$ is $P_\mathfrak{t}(n)$.
\end{proof}
\section{Semisimplicity Criteria for B--M--W
Algebras}\label{b-m-w-sc} Let $\kappa$ be a field and take
$\hat{q},\hat{r},(\hat{q}-\hat{q}^{-1})$ to be units in $\kappa$. In
this section we consider the algebra
$B_n(\hat{q},\hat{r})=B_n(q,r)\otimes_{R}\kappa$. For
$\mathfrak{t}\in\mathfrak{T}_n(\lambda)$ and $k=1,\dots,n$, let
$\hat{P}_\mathfrak{t}(k)$ denote the evaluation of the monomial
$P_\mathfrak{t}(k)$ at $(\hat{q},\hat{r})$,
\begin{align*}
\hat{P}_\mathfrak{t}(k)=
\begin{cases}
\hat{q}^{2(j-i)}&\text{if $[\lambda^{(k)}]=[\lambda^{(k-1)}]\cup
\{(i,j)\}$}\\
\hat{q}^{2(i-j)}\hat{r}^{-2}&\text{if
$[\lambda^{(k)}]=[\lambda^{(k-1)}]\setminus \{(i,j)\}$},
\end{cases}
\end{align*}
and define the ordered $n$-tuple
$\hat{P}(\mathfrak{t})=(\hat{P}_\mathfrak{t}(1),\dots,\hat{P}_\mathfrak{t}(n))$.
The next statement is the counterpart to Proposition~3.37
of~\cite{mathas:ih}.
\begin{proposition}\label{grubs}
Let $f$ be an integer, $0\le f\le [n/2]$, and $\lambda$ be a
partition of $n-2f$.

(i) Let $\rho=(\rho_1,\dots,\rho_n)$ be a sequence of elements of
$\kappa$ such that there exists a path
$\mathfrak{t}\in\mathfrak{T}_n(\lambda)$ with
$\rho=\hat{P}(\mathfrak{t})$. Then there exists a one--dimensional
$\mathscr{L}$--module $\mathscr{L}_\rho=\kappa x_\rho$ such that
\begin{align*}
x_\rho L_k=\rho_k x_\rho&&\text{for $k=1,2,\dots,n$}.
\end{align*}
Moreover, every irreducible $\mathscr{L}$--module has this form.

(ii) Let $f$ be an integer, $0\le f\le [n/2]$, and suppose that
$\lambda$ is a partition of $n-2f$. Fix an ordering
$\mathfrak{t}_1,\dots,\mathfrak{t}_k=\mathfrak{t}^\lambda$ of
$\mathfrak{T}_n(\lambda)$ so that $i>j$ whenever
$\mathfrak{t}_i\rhd\mathfrak{t}_j$. Then $S^\lambda$ has a
$\mathscr{L}$--module composition series
\begin{align*}
S^\lambda=S_1>S_2>\cdots>S_k>S_{k+1}=0
\end{align*}
such that $S_i/S_{i+1}=\mathscr{L}_{\rho^i}$, for each $i$, where
$\rho^i=\hat{P}({\mathfrak{t}_i})$.
\end{proposition}
\begin{proof}
As in~\cite{mathas:ih}, we prove (ii) from which item (i) will
follow. Order the elements of $\mathfrak{T}_n(\lambda)$ as in item
(ii), and for $i=1,\dots,k$, let $S_i$ be the $\kappa$--module
generated by $\{m_{\mathfrak{t}_j}:i\le j\le k\}$. By
Theorem~\ref{utrangular:2}, each $S_i$ is an $\mathscr{L}$--module,
and so $S^\lambda=S_1>\cdots S_k>0$ is an $\mathscr{L}$--module
filtration of $S^\lambda$. Further, by Theorem~\ref{utrangular:2}
again, $S_i/S_{i+1}=\kappa(m_{\mathfrak{t}_i}+S_{i+1})$ is a one
dimensional module isomorphic to $\mathscr{L}_i$.
\end{proof}
\begin{theorem}\label{ss:1}
Suppose that for each pair of partitions $\lambda$ of $n-2f$ and $\mu$ of $n-2f'$, for integers $f,f'$ with $0\le f,f'\le [n/2]$, and that for each pair of paths $\mathfrak{s}\in\mathfrak{T}_n(\lambda)$ and $\mathfrak{t}\in\mathfrak{T}_n(\mu)$, the conditions $\lambda\unrhd\mu$ and $\hat{P}(\mathfrak{s})=\hat{P}(\mathfrak{t})$ together imply that $\lambda=\mu$. Then $B_n(\hat{q},\hat{r})$ is a semisimple algebra over $\kappa$.
\end{theorem}
\begin{proof}
The hypotheses of the theorem imply that given a pair of partitions
$\lambda$ and $\mu$ with $\lambda\rhd\mu$, there are no
$\mathscr{L}$--module composition factors in common between
$S^\lambda$ and $S^\mu$. However, if $B_n(\hat{q},\hat{r})$ is not
semisimple, then using Theorem~\ref{g-lthm:2}, $D^\mu$ is a
$B_n(\hat{q},\hat{r})$--module composition factor of $S^\lambda$ for
some pair of partitions $\lambda$ and $\mu$ for which, by
Proposition~3.6 of~\cite{grahamlehrer}, $\lambda\rhd\mu$; in
particular, by Proposition~\ref{grubs}, there must be
$\mathscr{L}$--module composition factors in common between
$S^\lambda$ and $S^\mu$, which as already noted, is an
impossibility.
\end{proof}

From the next statement (Lemma~5.20 of~\cite{ramleduc:rh}), it will
follow that the Jucys--Murphy operators do in fact distinguish
between cell modules of $B_n(q,r)$.
\begin{lemma}\label{ram:dist}
Let $f$ be an integer $0\le f<[n/2]$ and $\lambda^{(n-1)}$ be a
partition of $n-1-2f$. If
$\mathfrak{s}=(\lambda^{(0)},\lambda^{(1)},\dots,\lambda^{(n-1)})$
is a path in $\mathfrak{T}_{n-1}(\lambda^{(n-1)})$, then the terms
$\big(P_{\mathfrak{t}}(n)\,:\,\mathfrak{t}|_{n-1}=\mathfrak{s}\big)$
are all distinct.
\end{lemma}
For the case where $\kappa=\mathbb{C}(\hat{q},\hat{r})$, a form of
the following statement can be found in Corollary~5.6
of~\cite{wenzlqg}.
\begin{corollary}\label{ss:2}
If $\kappa$ is a field, then a B--M--W algebra
$B_n(\hat{q},\hat{r})$ over $\kappa$ is semisimple for almost all
(all but finitely many) choices of the parameters $\hat{q}$ and
$\hat{r}$. If $B_n(\hat{q},\hat{r})$ is not semisimple then
necessarily $\hat{q}$ is a root of unity or $\hat{r}=\pm\hat{q}^k$
for some integer $k$.
\end{corollary}
Theorem~\ref{d-w:analogue} may be compared with Theorem~\ref{doran-wales-thm:1} below. Theorem~\ref{d-w:analogue} gives a semisimplicity criterion for $B_n(q,r)$.
\begin{theorem}\label{d-w:analogue}
Let $\lambda$ be a partition of $n-2f$ and $\mu$ be a partition of $n-2g$, where $0\le f<g\le [n/2]$. If $\Hom_{B_n(\hat{q},\hat{r})}(S^\lambda,S^\mu)\ne 0$, then
\begin{align*}
\hat{r}^{2(g-f)}\hat{q}^{2\sum_{(i,j)\in [\lambda]} {(j-i)} }= \hat{q}^{2\sum_{(i,j)\in [\mu]} {(j-i)}}.
\end{align*}
\end{theorem}
\begin{proof}
Suppose that  $\mathbf{u}\in S^\lambda$, $\mathbf{v}\in S^\mu$ are non--zero and that $\mathbf{u}\mapsto\mathbf{v}$ under some element in  $\Hom_{B_n(\hat{q},\hat{r})}(S^\lambda,S^\mu)$. Then, using Lemma~\ref{central}, on the one hand $\mathbf{u}(L_2L_3\cdots L_n)=\hat{r}^{-2f}\hat{q}^{2\sum_{(i,j)\in[\lambda]}(j-i)}\mathbf{u}$, while on the other $\mathbf{v}L_2L_3\cdots L_n=\hat{r}^{-2g}\hat{q}^{2\sum_{(i,j)\in[\mu]}(j-i)}\mathbf{v}$. Since $\mathbf{v}$ is the homomorphic image of $\mathbf{u}$, it follows that $\hat{r}^{-2f}\hat{q}^{2\sum_{(i,j)\in[\lambda]}(j-i)}=\hat{r}^{-2g}\hat{q}^{2\sum_{(i,j)\in[\mu]}(j-i)}$; hence the result.
\end{proof}
As the next example shows, Theorem~\ref{ss:1} gives a sufficient
but not the necessary condition for $B_n(\hat{q},\hat{r})$ to be a
semisimple algebra over $\kappa$; it can also be seen from the example that Theorem~\ref{d-w:analogue} gives a necessary but not sufficient condition for $\Hom_{B_n(\hat{q},\hat{r})}(S^\lambda,S^\mu)$ to be non--zero.
\begin{example}\label{b-m-w3:ex}
Let $n=3$, $\lambda=(1)$, $\mu=(3)$, $\kappa=\mathbb{Q}(\hat{q},\hat{r})$, and suppose that $\hat{r}=-\hat{q}^{-3}$, where $\hat{q}$ is not a root of unity. Since $\hat{q}$ is not root of unity, the cell modules for $B_3(\hat{q},\hat{r})$ corresponding to the partitions $(3)$, $(2,1)$ and $(1,1,1)$ are absolutely irreducible (Theorem 3.43 of~\cite{mathas:ih} together with Lemma~\ref{bquot} with $f=0$). On the other hand, if
\begin{align*}
\mathfrak{s}=(\varnothing,\text{\tiny$\begin{matrix}\yng(1)\end{matrix}$}\,,
\text{\tiny$\begin{matrix}\yng(2)\end{matrix}$}\,,
\text{\tiny$\begin{matrix}\yng(1)\end{matrix}$}\,)\in\mathfrak{T}_n(\lambda)&&\text{and}&&
\mathfrak{t}=(\varnothing,\text{\tiny$\begin{matrix}\yng(1)\end{matrix}$}\,,
\text{\tiny$\begin{matrix}\yng(2)\end{matrix}$}\,,
\text{\tiny$\begin{matrix}\yng(3)\end{matrix}$}\,)\in\mathfrak{T}_n(\mu),
\end{align*}
then
$\hat{P}(\mathfrak{s})=(1,\hat{q}^2,\hat{q}^{-2}\hat{r}^{-2})=(1,\hat{q}^2,\hat{q}^4)$
and $\hat{P}(\mathfrak{t})= (1,\hat{q}^2,\hat{q}^{4})$. Since
$\hat{P}(\mathfrak{s})=\hat{P}(\mathfrak{t})$ whilst
$\lambda\rhd\mu$, the pair $\mathfrak{s},\mathfrak{t}$ violates
the hypotheses of Theorem~\ref{ss:1}. But we note by reference to
the determinant of Gram matrix associated to $S^\lambda$ in
Example~\ref{bilinear:ex} that $S^\lambda$ is absolutely
irreducible and hence that $B_3(\hat{q},\hat{r})$ remains
semisimple over $\kappa$ (Theorems~\ref{g-lthm:1} and~\ref{g-lthm:2}).
\end{example}

\section{Brauer Algebras}\label{bralgs}
The foregoing construction for the B--M--W algebras applies with
minor modification to the Brauer algebras over an arbitrary
field. We begin once more by considering Brauer algebras over a
polynomial ring over $\mathbb{Z}$. Take $z$ to be an indeterminate
over $\mathbb{Z}$; we write $R=\mathbb{Z}[z]$ and define the Brauer
algebra $B_n(z)$ over $R$ as the associative unital $R$--algebra
generated by the transpositions $s_1,s_2,\dots,s_{n-1}$, together
with elements $E_1,E_2,\dots,E_{n-1}$, which satisfy the defining relations:
\begin{align*}
&s_i^2=1&&\text{for $1\le i<n$;}\\
&s_is_{i+1}s_i=s_{i+1}s_is_{i+1}&&\text{for $1\le i<n-1$;}\\
&s_is_j=s_js_i&&\text{for $2\le|i-j|$;}\\
&E_i^2=zE_i&&\text{for $1\le i<n$;}\\
&s_iE_j=E_js_i&&\text{for $2\le|i-j|$;}\\
&E_iE_j=E_jE_i&&\text{for $2\le|i-j|$;}\\
&E_is_i=s_iE_i= E_i&&\text{for $1\le i<n$;}\\
&E_is_{i\pm1}s_i=s_{i\pm1}s_iE_{i\pm1}=E_iE_{i\pm1}&&\text{for $1\le i,i\pm1<n$;}\\
&E_is_{i\pm1}E_i=E_iE_{i\pm1}E_{i}=E_i&&\text{for $1\le
i,i\pm1<n$.}
\end{align*}
Regard the group ring $R\mathfrak{S}_n$ as the subring of $B_n( z)$
generated by the transpositions $\{s_i=(i,i+1):\text{for $1\le
i<n$}\}$. If $f$ is an integer, $0\le f\le [n/2]$, and $\lambda$ is
a partition of $n-2f$, define the elements
\begin{align*}
x_\lambda=\sum_{w\in\mathfrak{S}_\lambda}w\hspace{2em}
\text{and}\hspace{2em} m_\lambda=E_1E_3\cdots E_{2f-1}x_\lambda,
\end{align*}
where $\mathfrak{S}_\lambda$ is the row stabiliser in $\langle
s_{2f+1}, s_{2f+2},\dots,s_{n-1}\rangle$ of the superstandard
tableau $\mathfrak{t}^\lambda\in\STD_{n}(\lambda)$. Let
$B^\lambda_n$ be the two sided ideal of $B_n( z)$ generated by
$m_\lambda$ and write
\begin{align*}
\check{B}^\lambda_n=\sum_{\mu\rhd\lambda}B^\mu_n.
\end{align*}
A cellular basis in terms of dangles has been given for the Brauer algebra in~\cite{grahamlehrer}. Replacing cellular bases for $\mathscr{H}_n(q^2)$ with cellular bases for $R\mathfrak{S}_n$, the process used to construct cellular bases the B--M--W algebras in~\cite{saru} will produce also cellular bases for $B_n( z)$ as follows. 

If $f$ is an integer, $0\le f\le [n/2]$, and $\lambda$ a partition of $n-2f$, then $\mathcal{I}_{n}(\lambda)$ retains the meaning assigned in~\eqref{index:1}.
 
\begin{theorem}\label{saruthm:2}
The algebra $B_n( z)$ is freely generated as an $R$--module by the
collection
\begin{align*}
\left\{(d(\mathfrak{s})v)^{-1}m_\lambda d(\mathfrak{t})u\,\bigg|\,
\begin{matrix}
\text {$(\mathfrak{s},v),(\mathfrak{t},u)\in \mathcal{I}_{n}(\lambda)$ for $\lambda$ a partition} \\
\text{of $n-2f$ and $0\le f\le [n/2]$\,}
\end{matrix}
\right\}.
\end{align*}
Moreover, the following statements hold.
\begin{enumerate}
\item The $R$--linear map determined by
\begin{align*}
(d(\mathfrak{s})v)^{-1}m_\lambda d(\mathfrak{t})u\mapsto
(d(\mathfrak{t})u)^{-1}m_\lambda d(\mathfrak{s})v
\end{align*}
is an algebra anti--involution of $B_n( z)$. \item  Suppose that
$b\in B_n( z)$ and let $f$ be an integer, $0\le f\le [n/2]$. If
$\lambda$ is a partition of $n-2f$ and
$(\mathfrak{t},u)\in\mathcal{I}_{n}(\lambda)$, then there exist
$a_{(\mathfrak{u},w)}\in R$, for
$(\mathfrak{u},w)\in\mathcal{I}_{n}(\lambda)$, such that for all
$(\mathfrak{s},v)\in\mathcal{I}_{n}(\lambda)$,
\begin{align}\label{btsu:2}
(d(\mathfrak{s})v)^{-1}m_\lambda d(\mathfrak{t})u b\equiv
\sum_{(\mathfrak{u},w)}a_{(\mathfrak{u},w)}
(d(\mathfrak{s})v)^{-1}m_\lambda d(\mathfrak{u})w \mod
\check{B}^\lambda_n.
\end{align}
\end{enumerate}
\end{theorem}
As a consequence of the above theorem, $\check{B}^\lambda_n$ is the $R$--module freely generated by
\begin{align*}
\big\{(d(\mathfrak{s})v)^{-1}m_\mu d(\mathfrak{t})u
:(\mathfrak{s},v),(\mathfrak{t},u)\in\mathcal{I}_{n}(\mu), \text
{ for }\mu\rhd\lambda\big\}.
\end{align*}

If $f$ is an integer, $0\le f\le [n/2]$, and $\lambda$ is a
partition of $n-2f$, the cell module $S^\lambda$ is defined to be
the $R$--module freely generated by
\begin{align}\label{bcell:2}
\left \{m_\lambda d(\mathfrak{t})u
+\check{B}_n^\lambda\,|\,(\mathfrak{t},u)\in
\mathcal{I}_{n}(\lambda)\right\}
\end{align}
with right $B_n( z)$ action
\begin{align*}
(m_\lambda d(\mathfrak{t})u) b+\check{B}_n^\lambda=
\sum_{(\mathfrak{u},w)}a_{(\mathfrak{u},w)} m_\lambda
d(\mathfrak{u})w +\check{B}_n^\lambda&&\text{for $b\in B_n( z)$,}
\end{align*}
where the coefficients $a_{(\mathfrak{u},w)}\in R$, for
$(\mathfrak{u},w)$ in $\mathcal{I}_{n}(\lambda)$, are determined by
the expression~\eqref{btsu:2}.

The construction of cellular algebras~\cite{grahamlehrer} equips the
$B_n(z)$--module $S^\lambda$ with a symmetric associative bilinear
form (compare~\eqref{formdef:1} above). Following is the counterpart
to Example~\ref{bilinear:ex}, stated for reference in Section~\ref{brauerss}.
\begin{example}\label{bilinear:ex2}
Let $n=3$ and $\lambda=(1)$ so that $\check{B}_n^\lambda=(0)$ and
$m_\lambda=E_1$. We order the basis~\eqref{bcell:2} for the module
$S^\lambda$ as $\mathbf{v}_1=E_1$, $\mathbf{v}_2=E_1s_2$ and
$\mathbf{v}_3=E_1s_2s_1$ and, with respect to this ordered basis,
the Gram matrix $\langle\mathbf{v}_i,\mathbf{v}_j\rangle$ of the
bilinear form on the $B_n(z)$--module $S^\lambda$ is
\begin{align*}
\left[
\begin{matrix}
z &1 &1\\
1 &z&1\\
1 &1& z
\end{matrix}
\right].
\end{align*}
The determinant of the Gram matrix given above is
\begin{align*}
(z-1)^2(z+2).
\end{align*}
\end{example}

By Theorem~2.3 of~\cite{wenzlqg}, the Bratteli diagram associated
with $B_n( z)$ is identical to the Bratteli diagram for $B_n(q,r)$.
Thus $\mu\to\lambda$ retains the meaning assigned in Section~\ref{resf}.

Let $f$ be an integer, $0\le f\le[n/2]$, and $\lambda$ be a
partition of $n-2f$ with $t$ removable nodes and $(p-t)$ addable
nodes. Suppose that
$\mu^{(1)}\rhd\mu^{(2)}\rhd\cdots\rhd\mu^{(p)}$ is the ordering of
$\{\mu:\mu\to\lambda\}$ by dominance order on partitions. If
$1\le k\le t$, define
\begin{align*}
y_{\mu^{(k)}}^\lambda = m_\lambda d(\mathfrak{s})+\check{B}_n^\lambda &&\text{where
$\mathfrak{s}|_{n-1}=\mathfrak{t}^{\mu^{(k)}}\in\STD_{n-1}(\mu^{(k)})$}
\end{align*}
and, if $t<k\le p$ define $w_k$ by~\eqref{wkdef} and, by analogy
with~\ref{ydef:2}, write
\begin{align*}
 y_{\mu^{(k)}}^\lambda=E_{2f-1}w_k^{-1}m_{\mu^{(k)}}+ \check{B}_n^\lambda 
\end{align*}
Given the elements $y_\mu^\lambda$ in $S^\lambda$ for each partition
$\mu\to \lambda$, define $N^\mu$ to be the $B_{n-1}( z)$--submodule
of $S^\lambda$ generated by
\begin{align*}
\{y^\lambda_\nu :\nu\to\lambda\text{ and }\nu\unrhd\mu\}
\end{align*}
and let $\check{N}^\mu$ be the $B_{n-1}( z)$--submodule of
$S^\lambda$ generated by
\begin{align*}
\{y^\lambda_\nu:\text{$\nu\to\lambda$ and $\nu\rhd\mu$}\}.
\end{align*}
Theorem~\ref{big:4} and the construction given for the B--M--W
algebras in Section~\ref{newbasis} have analogues in the context of
$B_n( z)$. Thus the cell module~\eqref{bcell:2} has a basis over
$R$, 
\begin{align*}
\{ m_\mathfrak{t}= m_\lambda b_{\mathfrak{t}}+\check{B}_n^\lambda\,:\,\mathfrak{t}\in\mathfrak{T}_n(\lambda)\}
\end{align*}
indexed by the paths $\mathfrak{T}_n(\lambda)$ of
shape $\lambda$ in the Bratteli diagram associated with $B_n( z)$, and defined in the same manner as the basis~\eqref{b-murphy}.
\section{Jucys--Murphy Operators for the Brauer Algebras}
Define the operators $L_i$, for $i=1,\dots,n$, in $B_n(z)$ by
$L_1=0$ and
\begin{align*}
L_i=s_{i-1}-E_{i-1}+s_{i-1}L_{i-1}s_{i-1}&&\text{for $1<i\le n$}.
\end{align*}
\begin{remark}
The elements $L_i$ as defined above bear an obvious analogy with the elements $\tilde{D}_i$ defined in Section~\ref{ihsec}; thus we refer to the elements $L_i$ as the ``Jucys--Murphy operators'' in $B_n(z)$.
\end{remark}

In~\cite{nazarov}, M.~Nazarov made use of operators $x_i$ with are related to the $L_i$ defined above by $x_i=\frac{z-1}{2}+L_i$. Since the difference of $L_i$ and the $x_i$ of~\cite{nazarov} is a scalar multiple of the identity, we derive the next statement from results in Section~2 of~\cite{nazarov}.
\begin{proposition} Let $i$ and $k$ be integers, $1\le i<n$
and $1\le k\le n$.
\begin{enumerate}
\item $s_i$ and $L_k$ commute if $i\ne k-1,k$. \item $L_i$ and
$L_k$ commute.\item $s_i$ commutes with $L_i+L_{i+1}$. \item
$L_2+L_3+\cdots+ L_n$ belongs to the centre of $B_n(z)$.
\end{enumerate}
\end{proposition}

For integers $j,k$ with $1\le j,k\le n$, we define the elements
$L^{(j)}_k$ by $L^{(j)}_1=0$ and
\begin{align*}
L^{(j)}_{k+1}=s_{j+k-1}-E_{j+k-1}+s_{j+k-1}L^{(j)}_{k}s_{j+k-1},&&\text{for $k\ge 1$.}
\end{align*}
In particular, $L_k^{(1)}=L_k$, for $k=1,\dots,n$, are the
Jucys--Murphy elements for $B_n(z)$. 

The objective now is to show that $m_{\mathfrak{t}^\lambda}$ is a common eigenvector for the action of the Jucys--Murphy elements $L_k$ on the cell module $S^\lambda$.
\begin{proposition}\label{fstep}
Let $i,k$ be integers with $1\le i\le n$ and $1<k\le n$. Then
\begin{align*}
E_iL_{k}^{(i)}=
\begin{cases}
(1-z)E_i &\text{if $k=2$;}\\
0&\text{if $k=3$;}\\
E_iL^{(i+2)}_{k-2}&\text{if $k\ge4$.}
\end{cases}
\end{align*}
\end{proposition}
\begin{proof}
If  $k=2$ then $E_iL^{(i)}_{k}=E_i(s_i-E_i)=(1-z)E_i$. For $k=3$
we have
\begin{align*}
E_iL^{(i)}_3&=E_i(s_{i+1}-E_{i+1}+s_{i+1}s_{i}s_{i+1}-s_{i+1}E_{i}s_{i+1})\\
&=E_i(s_{i+1}-E_{i+1})+E_i(E_{i+1}s_{i+1}-s_{i+1})=0.
\end{align*}
If $k=4$ then,
\begin{align*}
E_iL^{(i)}_4&=E_i(s_{i+2}-E_{i+2})+s_{i+2}E_iL^{(i)}_3s_{i+2}\\
&=E_{i}(s_{i+2}-E_{i+2})=E_iL^{(i+2)}_2,
\end{align*}
and when $k>4$,
\begin{align*}
E_iL^{(i)}_k&=E_i(s_{i+k-2}-E_{i+k-2})+s_{i+k-2}E_iL^{(i)}_{k-1}s_{i+k-2}\\
&=E_i(s_{i+k-2}-E_{i+k-2})+s_{i+k-2}E_iL^{(i+2)}_{k-3}s_{i+k-2}\\
&=E_i(s_{i+k-2}-E_{i+k-2}+s_{i+k-2}L^{(i+2)}_{k-3}s_{i+k-2})=E_iL^{(i+2)}_{k-2}
\end{align*}
by induction.
\end{proof}
\begin{corollary}\label{groblecor:2}
Let $f,k$ be integers, $0<f\le[n/2]$ and $1\le k\le n$. Then
\begin{align*}
E_1E_3\cdots E_{2f-1}L_k
=
\begin{cases}
0,&\text{if $k$ is odd, $1\le
k\le2f+1$;}\\
(1-z)E_1E_3\cdots E_{2f-1,}&\text{if $k$ is even, $1<
k\le2f$;}\\
E_1E_3\cdots E_{2f-1}L_{k-2f}^{(2f+1)},&\text{if $2f+1<k\le n$.}
\end{cases}
\end{align*}
\end{corollary}
\begin{proof}
If $k$ is odd, $1<k\le2f+1$, then by Proposition~\ref{fstep},
\begin{multline}\begin{split}\label{froble}
E_1E_3\cdots E_{k}L_k=E_1E_3\cdots E_{k}L^{(1)}_k=E_1E_3\cdots
E_{k}L^{(3)}_{k-2}=\cdots\\ \cdots=E_1E_3\cdots
E_{k}L_1^{(k)}=0.
\end{split}
\end{multline}
Since $E_{k+2}E_{k+3}\cdots E_{2f-1}$ commutes with $L_k$, the first case follows. If $k$ is even and $1<k\le 2f$, then the
relations $E_{i}s_{i}=E_{i}$ and $E_i^2=zE_i$, together with~\eqref{froble}, show that
\begin{align*}
E_1E_3\cdots E_{2f-1}L_k&=E_1E_3\cdots
E_{2f-1}(s_{k-1}-E_{k-1}+s_{k-1}L_{k-1}s_k)\\
&=(1-z)E_1E_3\cdots E_{2f-1}+E_1E_3\cdots E_{2f-1}L_{k-1}s_{k-1}\\
&=(1-z)E_1E_3\cdots E_{2f-1}.
\end{align*}
The final case follows in a similar manner.
\end{proof}
Let $f$ be an integer, $0\le f\le [n/2]$, and $\lambda$ be a
partition of $n-2f$. For each path
$\mathfrak{t}\in\mathfrak{T}_n(\lambda)$, define the polynomial
\begin{align*}
P_{\mathfrak{t}}(k)=
\begin{cases}
j-i &\text{if $[\lambda^{(k)}]=[\lambda^{(k-1)}]\cup
\{(i,j)\}$}\\
i-j+1-z&\text{if $[\lambda^{(k)}]=[\lambda^{(k-1)}]\setminus
\{(i,j)\}$}.
\end{cases}
\end{align*}
The proof of the next statement is identical to the proof of
Proposition~\ref{utrangular:3} given above; for the proof of
Proposition~\ref{prd}, we refer to the proof of Proposition~\ref{danaming}.
\begin{proposition}
If $\lambda$ is a partition of $n$ and $k$ is an integer with
$1\le k\le n$, then
$m_{\mathfrak{t}^\lambda}L_k=P_{\mathfrak{t}^\lambda}(k)m_{\mathfrak{t}^\lambda}$.
\end{proposition}
\begin{proposition}\label{prd}
Let $f$ be an integer, $0<f\le [n/2]$, and $\lambda$ be a
partition of $n-2f$. Then
$m_{\mathfrak{t}^\lambda}L_k=P_{\mathfrak{t}^\lambda}(k)m_{\mathfrak{t}^\lambda}$.
\end{proposition}
\begin{proposition}\label{brauercentral}
Let $f$ be an integer, $0\le f\le[n/2]$, and $\lambda$ be a
partition of $n-2f$. Then there exists an invariant $\alpha\in R$
such that $L_2+L_3+\cdots+L_n$ acts on $S^\lambda$ as a scalar
multiple by $\alpha$ of the identity.
\end{proposition}
\begin{proof}
As in the proof of Proposition~\ref{central}, we obtain
$\alpha=\sum_{k=2}^nP_{\mathfrak{t}^\lambda}(k)$.
\end{proof}
\begin{theorem}\label{br:utran}
Let $f$ be an integer $0\le f\le [n/2]$ and $\lambda$ be a
partition of $n-2f$. If $\mathfrak{t}\in\mathfrak{T}_n(\lambda)$,
then there exist $a_\mathfrak{v}\in R$, for
$\mathfrak{v}\in\mathfrak{T}_n(\lambda)$ with
$\mathfrak{v}\rhd\mathfrak{t}$, such that
\begin{align*}
m_\mathfrak{t} L_k=P_{\mathfrak{t}}(k)m_\mathfrak{t}+
\sum_{\substack{\mathfrak{v}\in\mathfrak{T}_n(\lambda)\\\mathfrak{v}\rhd\mathfrak{t}}}
a_\mathfrak{v}m_\mathfrak{v}.
\end{align*}
\end{theorem}
\begin{proof}
By repeating word for word the argument given in the proof of Theorem~\ref{utrangular:2}, we show that the statement holds true when $1\le k<n$. 

That $L_n$ acts triangularly on $S^\lambda$, can then be observed
using Proposition~\ref{brauercentral}:
\begin{align*}
m_\mathfrak{t}L_n=\sum_{k=1}^{n}P_\mathfrak{t}(k)m_\mathfrak{t}-m_{\mathfrak{t}}(L_2+L_3+\cdots
+L_{n-1}).
\end{align*}
Thus the generalised eigenvalue for $L_n$ acting on
$m_\mathfrak{t}$ is $P_\mathfrak{t}(n)$.
\end{proof}
\section{Semisimplicity Criteria for Brauer Algebras}\label{brauerss}
Below are analogues for the Brauer algebras of the results of
Section~\ref{b-m-w-sc}. Let $\kappa$ be a field and take
$\hat{z}\in\kappa$. Then $z\mapsto \hat{z}$ determines a
homomorphism $R\to \kappa$, giving $\kappa$ an $R$--module
structure. A Brauer algebra over $\kappa$ is a specialisation
$B_n(\hat{z})=B_n(z)\otimes_{R}\kappa$. For
$\mathfrak{t}\in\mathfrak{T}_n(\lambda)$ and $k=1,\dots,n$, let
$\hat{P}_\mathfrak{t}(k)$ denote the evaluation of the monomial
$P_\mathfrak{t}(k)$ at $\hat{z}$,
\begin{align*}
\hat{P}_\mathfrak{t}(k)=
\begin{cases}
j-i&\text{if $[\lambda^{(k)}]=[\lambda^{(k-1)}]\cup
\{(i,j)\}$}\\
i-j+1-\hat{z}&\text{if $[\lambda^{(k)}]=[\lambda^{(k-1)}]\setminus
\{(i,j)\}$},
\end{cases}
\end{align*}
and as previously, define the ordered $n$-tuple
$\hat{P}(\mathfrak{t})=(\hat{P}_\mathfrak{t}(1),\dots,\hat{P}_\mathfrak{t}(n))$. The operators $L_i$ provide conditions necessary for the existence of a homomorphic image of one cell module for $B_n(\hat{z})$ in another cell module for $B_n(\hat{z})$.
\begin{theorem}\label{ss:3}
Let $\kappa$ be a field. Suppose that for each pair of partitions
$\lambda$ of $n-2f$ and $\mu$ of $n-2f'$, for integers $f,f'$ with
$0\le f,f'\le [n/2]$, and for each pair of partitions
$\mathfrak{s}\in\mathfrak{T}_n(\lambda)$ and
$\mathfrak{t}\in\mathfrak{T}_n(\mu)$, the conditions
$\lambda\unrhd\mu$ and
$\hat{P}(\mathfrak{s})=\hat{P}(\mathfrak{t})$ together imply that
$\lambda=\mu$. Then $B_n(\hat{z})$ is a semisimple algebra over
$\kappa$.
\end{theorem}
By an analogous statement to Lemma~\ref{ram:dist}, the Jucys--Murphy
elements do in fact distinguish between the cell modules of
$B_n(z)$ in Theorem~\ref{ss:3}.

The results of this section can be used to derive the next statement which is Theorem~3.3 of~\cite{doran-wales}. As in Theorem~\ref{d-w:analogue}, the statement may be generalised to the setting where $|\lambda|>|\mu|$.
\begin{theorem}\label{doran-wales-thm:1}
Let $\lambda$ be a partition of $n$ and $\mu$ be a partition of $n-2f$, where $f>0$. If $\Hom_{B_n(\hat{z})}(S^\lambda,S^\mu)\ne 0$, then
\begin{align*}
\sum_{(i,j)\in[\lambda]}(j-i)-\sum_{(i,j)\in[\mu]}(j-i)=f(1-\hat{z}).
\end{align*}
\end{theorem}
\begin{proof}
Suppose that  $\mathbf{u}\in S^\lambda$, $\mathbf{v}\in S^\mu$ are non--zero and that $\mathbf{u}\mapsto\mathbf{v}$ under some element in  $\Hom_{B_n(\hat{z})}(S^\lambda,S^\mu)$. Then, using Proposition~\ref{brauercentral}, 
\begin{align*}
\sum_{i=1}^n\mathbf{u} L_i=\sum_{(i,j)\in[\lambda]}(j-i)\mathbf{u}
\end{align*}
while
\begin{align*}
\sum_{i=1}^n\mathbf{v} L_i=f(1-\hat{z})\mathbf{v}+\sum_{(i,j)\in[\mu]}(j-i)\mathbf{v}.
\end{align*}
Since $\mathbf{v}$ is the homomorphic image of $\mathbf{u}$, it follows that
\begin{align*}
\sum_{(i,j)\in[\lambda]}(j-i)=f(1-\hat{z})+\sum_{(i,j)\in[\mu]}(j-i).
\end{align*}
Hence the result.
\end{proof}

Theorem~\ref{ss:3} gives a sufficient but not the necessary
condition for $B_n(\hat{z})$ to be a semisimple algebra over
$\kappa$.  Necessary and sufficient conditions on the semicimplicity of $B_n(\hat{z})$ have been given by H.~Rui in~\cite{rui:brss}.
\begin{example}\label{brauer3:ex}
Let $\kappa=\mathbb{Q}$ and $\hat{z}=4$. Take $n=3$, $\lambda=(1)$
and $\mu=(1,1,1)$. In characteristic zero the cell modules
corresponding to the partitions $(3)$, $(2,1)$ and $(1,1,1)$ are
absolutely irreducible. But, taking
\begin{align*}
\mathfrak{t}=(\varnothing,\text{\tiny$\begin{matrix}\yng(1)\end{matrix}$}\,,
\text{\tiny$\begin{matrix}\yng(1,1)\end{matrix}$}\,,
\text{\tiny$\begin{matrix}\yng(1)\end{matrix}$}\,)\in\mathfrak{T}_n(\lambda)&&\text{and}&&
\mathfrak{u}=\left(\varnothing,\text{\tiny$\begin{matrix}\yng(1)\end{matrix}$}\,,
\text{\tiny$\begin{matrix}\yng(1,1)\end{matrix}$}\,,
\text{\tiny$\begin{matrix}\yng(1,1,1)\end{matrix}$}\,\right)\in\mathfrak{T}_n(\mu),
\end{align*}
then
\begin{align*}
\hat{P}(\mathfrak{t})=(0,-1,2-\hat{z})=(0,-1,-2)&&\text{and}
&&\hat{P}(\mathfrak{u})= (0,-1,-2).
\end{align*}
Since $\hat{P}(\mathfrak{t})=\hat{P}(\mathfrak{u})$ whilst
$\lambda\rhd\mu$, the pair $\mathfrak{t},\mathfrak{u}$ violates the
hypotheses of Theorem~\ref{ss:3}. However,  by reference to the
determinant of Gram matrix associated to $S^\lambda$ in
Example~\ref{bilinear:ex2}, it follows that $S^\lambda$ is
absolutely irreducible and hence that $B_3(\hat{z})$ remains
semisimple by appeal to appropriate analogues of Theorems~\ref{g-lthm:1}
and~\ref{g-lthm:2}.
\end{example}
\section{Conjectures}\label{farce}
Define a sequence of polynomials $(p_i(z)\,|\,i=1,2,\dots,)$ by $p_1(z)=(z+2)(z-1)$ and 
\begin{align*}
p_i(z)=
\begin{cases}
(z+2i)(z-i)(z+i-2)p_{i-1}(z) &\text{if $i$ is odd;}\\
(z+2i)(z-i)p_{i-1}(z) &\text{if $i$ is even.}
\end{cases}
\end{align*}
\begin{conjecture}
For $\kappa$ a field, $\hat{z}\in\kappa$ and an algebra over $\kappa$, with $n\ge 2$, the following statements hold:

(i) If $n=2k+1$, then the bilinear form on the $B_n(\hat{z})$--module $S^{(1)}$ determined by~\eqref{formdef:1} is non--degenerate if and only if $p_k(\hat{z})\ne 0$. 

(ii) If $n=2k$, then the bilinear form on the $B_n(\hat{z})$--module $S^{\varnothing}$ determined by~\eqref{formdef:1} is non--degenerate if and only if $\hat{z}\ne 0$ and $p_k(\hat{z})\ne 0$. 
\end{conjecture}
\begin{conjecture}
For $\kappa$ a field, $\hat{z}\in\kappa$ and an algebra over $\kappa$, with $n\ge 2$, the following statements hold:

(i) If $n=2k+1$, then $B_n(\hat{z})$ is semisimple and only if $\kappa\mathfrak{S}_n$ is semisimple and $p_{2k-1}(\hat{z})\ne 0$.

(ii)  If $n=2k$, then $B_n(\hat{z})$ is semisimple and only if $\kappa\mathfrak{S}_n$ is semisimple, $\hat{z}\ne 0$ and $p_{2k-2}(\hat{z})\ne 0$.
\end{conjecture}


\bibliographystyle{plain}       






\end{document}